\documentclass[11pt,a4paper]{article}

\usepackage[utf8]{inputenc}

\usepackage{color}
\usepackage{subfigure}
\usepackage{url}
\usepackage{graphicx}
\usepackage{amsmath}
\usepackage{amsthm}
\usepackage{bm}
\usepackage[plainpages=false]{hyperref}

\usepackage{titlesec}
\titleformat{\section}
{\normalfont\Large\bfseries\filcenter}{\thesection.}{1 ex}{}
\titleformat{\subsection}%
{\normalfont\normalsize\bfseries}{\thesubsection.}{1 ex}{}
\titleformat{\subsubsection}[runin]
{\normalfont\normalsize\bfseries\filcenter}{\thesubsubsection.}{1 ex}{}

\usepackage[charter]{mathdesign}
\usepackage[mathcal]{eucal}

\usepackage[margin=1.2 in]{geometry}

\usepackage[square,numbers,sort&compress]{natbib}

\newcommand{\emailaddr}[1]{\href{mailto:#1}{\texttt{#1}}}

\usepackage{thmtools,thm-restate}
\declaretheoremstyle[qed=$\diamond$,headpunct={ --- },headfont=\normalfont\itshape]{myremark}
\declaretheoremstyle[qed=$\blacksquare$,bodyfont=\normalfont]{mydefinition}

\declaretheorem[name=Theorem,within=section]{Thm}
\declaretheorem[within=section,name=Lemma]{Lem}
\declaretheorem[sibling=Lem,name=Proposition]{Prop}
\declaretheorem[sibling=Lem,name=Corollary]{Cor}

\declaretheorem[style=myremark,sibling=Lem,name=Remark]{Rem}
\declaretheorem[style=mydefinition,sibling=Lem,name=Definition]{Def}
\declaretheorem[style=mydefinition,sibling=Lem,name=Notation]{Not}

\declaretheorem[style=mydefinition,sibling=Lem,name=Example]{Ex}

\declaretheorem[numbered=no,name={Theorem \ref{Thm:crypto} (informal)}]{cryptointro}
\declaretheorem[numbered=no,name={Lemma \ref{Lem:circuitpoly} (informal)}]{circpolythm}
\declaretheorem[numbered=no,name={Theorem \ref{Thm:circuitsize} (informal)}]{finitenessthm}
\declaretheorem[numbered=no,name={Theorem \ref{Thm:topdegfiniteness} (informal)}]{topdegthm}
\declaretheorem[numbered=no,name={Theorem \ref{Thm:multpoly} (informal)}]{multipolythm}

\usepackage{enumerate}
\usepackage{fancyhdr}
\usepackage{verbatim}
\usepackage{array}
\usepackage{fullpage}
\usepackage{enumitem}
\usepackage[vcentermath, enableskew]{youngtab}
\usepackage{multicol}
\usepackage{longtable}
\usepackage[small,bf,singlelinecheck=off]{caption}
\usepackage{bbm}

\renewcommand{\vec}[1]{\mathbf{#1}}
\newcommand{\cb}{\mathbf{c}}
\newcommand{\mb}{\mathbf{m}}
\newcommand{\nb}{\mathbf{n}}

\newcommand{\Frac}{\operatorname{Frac}}
\newcommand{\Stab}{\operatorname{Stab}}

\newcommand{\supp}{\operatorname{supp}}

\newcommand{\vsupp}{\operatorname{vsupp}}

\newcommand{\rk}{\operatorname{rk}}
\newcommand{\ark}{\rho}
\newcommand{\drk}{\delta\!\!\operatorname{rk}}
\newcommand{\crk}{\alpha}
\newcommand{\rsize}{\kappa}

\newcommand{\trdeg}{\operatorname{trdeg}}

\newcommand{\topdeg}{\top\!\!\deg}

\newcommand{\rank}{\operatorname{rank}}

\newcommand{\chr}{\operatorname{char}}
\newcommand{\sgn}{\operatorname{sgn}}

\newcommand{\compresslist}{\setlength{\itemsep}{1pt}
\setlength{\parskip}{0pt}
\setlength{\parsep}{0pt} }

\newcommand{\spn}{\operatorname{span}}

\newcommand{\matrex}{\mathcal{M}}         %

\newcommand{\matrlin}{\vec L}             %
\newcommand{\matralg}{\vec A}             %
\newcommand{\matrbas}{\vec B}             %
\newcommand{\matrcrd}{\vec C}             %

\newcommand{\sym}{\oslash}            %

\newcommand{\detvar}{\mathcal{M}}
\newcommand{\detvarsym}{\detvar^\sym}

\newcommand{\detI}{\mathcal{I}}
\newcommand{\detIsym}{\detI^\sym}

\newcommand{\detM}{\mathbf{D}}
\newcommand{\detMsym}{\detM^\sym}

\newcommand{\CMvar}{\mathcal{C\!M}}
\newcommand{\CMvarsym}{\CMvar^\sym}

\newcommand{\CMI}{\mathfrak{C}}
\newcommand{\CMIsym}{\CMI^\sym}

\newcommand{\CMM}{\mathbf{C\!M}}
\newcommand{\CMMsym}{\CMM^\sym}

\newcommand{\circnum}{c}
\newcommand{\stabavg}{\beta}

\newcommand{\calA}{\mathcal{A}}
\newcommand{\calB}{\mathcal{B}}
\newcommand{\calC}{\mathcal{C}}
\newcommand{\calD}{\mathcal{D}}
\newcommand{\calE}{\mathcal{E}}

\newcommand{\calG}{\mathcal{G}}
\newcommand{\calI}{\mathcal{I}}

\newcommand{\calM}{\mathcal{M}}
\newcommand{\calN}{\mathcal{N}}

\newcommand{\calP}{\mathcal{P}}

\newcommand{\calR}{\mathcal{R}}
\newcommand{\calS}{\mathcal{S}}
\newcommand{\calT}{\mathcal{T}}

\newcommand{\calV}{\mathcal{V}}
\newcommand{\calX}{\mathcal{X}}
\newcommand{\calY}{\mathcal{Y}}

\newcommand{\frakI}{\mathfrak{I}}
\newcommand{\frakJ}{\mathfrak{J}}
\newcommand{\frakP}{\mathfrak{P}}

\newcommand{\frakS}{\mathfrak{S}}
\newcommand{\frakY}{\mathfrak{Y}}

\usepackage{multirow}
\usepackage{tikz}
\usetikzlibrary{matrix, shapes, arrows, calc, backgrounds}
\usetikzlibrary{arrows,chains,matrix,positioning,scopes}
\makeatletter
\tikzset{join/.code=\tikzset{after node path={%
\ifx\tikzchainprevious\pgfutil@empty\else(\tikzchainprevious)%
edge[every join]#1(\tikzchaincurrent)\fi}}}
\tikzset{>=stealth',every on chain/.append style={join},
every join/.style={->}}

\usepackage{matrixPics}

\newcommand{\ZZ}{\ensuremath{\mathbb{Z}}}
\newcommand{\RR}{\ensuremath{\mathbb{R}}}

\newcommand{\CC}{\ensuremath{\mathbb{C}}}
\newcommand{\NN}{\ensuremath{\mathbb{N}}}

\newcommand{\alphabf}{\bm{\alpha}}

\newcommand{\bbf}{\mathbf{b}}

\newcommand{\dbf}{\mathbf{d}}

\newcommand{\vbf}{\mathbf{v}}

\newcommand{\xbf}{\mathbf{x}}
\newcommand{\ybf}{\mathbf{y}}
\newcommand{\zbf}{\mathbf{z}}

\newcommand{\CM}{\operatorname{CM}}

\DeclareMathOperator*{\Id}{I}
\DeclareMathOperator*{\Van}{V}

\makeatletter
\providecommand*{\diff}%
{\@ifnextchar^{\DIfF}{\DIfF^{}}}
\def\DIfF^#1{%
\mathop{\mathrm{\mathstrut d}}%
\nolimits^{#1}\gobblespace
}
\def\gobblespace{%
\futurelet\diffarg\opspace}
\def\opspace{%
\let\DiffSpace\!%
\ifx\diffarg(%
\let\DiffSpace\relax
\else
\ifx\diffarg\[%
\let\DiffSpace\relax
\else
\ifx\diffarg\{%
\let\DiffSpace\relax
\fi\fi\fi\DiffSpace}
\makeatother

\begin{document}
\title{Algebraic Matroids with Graph Symmetry}
\author{
\href{http://www.ucl.ac.uk/statistics/people/franz-kiraly}{Franz J. Király} \thanks{Department of Statistical Science, University College London; and MFO; \emailaddr{f.kiraly@ucl.ac.uk}}
\and \href{http://math.berkeley.edu/~zhrosen/}{Zvi Rosen}
\thanks{Department of Mathematics, University of California, Berkeley, \emailaddr{zhrosen@math.berkeley.edu}}
\and
\href{http://math.fu-berlin.de/~theran}{Louis Theran}\thanks{Inst. Math., Freie Universität Berlin, \emailaddr{theran@math.fu-berlin.de}} }
\date{}
\maketitle
\begin{abstract}
\begin{normalsize}

This paper studies the properties of two kinds of matroids: (a) algebraic matroids and (b) finite and infinite
matroids whose ground set have some canonical symmetry, for example row and column symmetry and transposition symmetry.

For (a) algebraic matroids, we expose cryptomorphisms making them accessible to techniques from commutative algebra.
This allows us to introduce for each circuit in an algebraic matroid an invariant called \emph{circuit polynomial},
generalizing the minimal polynomial in classical Galois theory, and studying the matroid structure
with multivariate methods.

For (b) matroids with symmetries we introduce combinatorial invariants capturing structural properties of the rank
function and its limit behavior, and obtain proofs which are purely combinatorial and do not assume algebraicity of
the matroid; these imply and generalize known results in some specific cases where the matroid is also algebraic.
These results are motivated by, and readily applicable to \emph{framework rigidity},
\emph{low-rank matrix completion} and \emph{determinantal varieties}, which lie in the intersection of
(a) and (b) where additional results can be derived. We study the corresponding matroids and their
associated invariants, and for selected cases, we characterize the matroidal structure and the circuit
polynomials completely.

\end{normalsize}
\end{abstract}

\newpage
\tableofcontents
\newpage

\section{Introduction}\label{sec:intro}

\subsection{Representations of Algebraic Matroids}

Since their discovery by van der Waerden~\cite{ModAlgVol1Ed2}, algebraic matroids have been one of the central objects of interest in matroid theory, and therefore in mathematics, as they describe algebraic dependence relations in \emph{Galois extensions with multiple elements}. Their structure, the possible ways to realize them, and their relation to linear matroids has been extensively explored.\\

In recent literature, a different way of viewing algebraic matroids has arisen - namely describing satisfiability of equations. One preliminary link in this direction can be found in the literature for the \emph{framework realization} problem, which asks for the reconstruction of an unknown set of points $p_1,p_2,\ldots,p_n\subset \mathbb{R}^r$ given a subset of the $\binom{n}{2}$ possible pairwise (squared) distances $d_{ij} = p_i^\top p_i - 2p_i^\top p_j + p_j^\top p_j$. A slight rephrasing of the question yields on of the central problems in- the field of \emph{combinatorial rigidity} (see~\cite{GSS93} for a recent review on the field), which asks the following: Given some subset $d_{ij}, (ij)\in S$ of all distances, which of the other distances are determined either uniquely, or up to a finite choice? By studying the Jacobian of the map $\{p_i\}\mapsto \{d_{ij}\}$, this can be recast in the language of linear matroids, yielding novel insights about the rigidity problem~\cite{L70,LY82,W96,ST10,T12}.\\

Similar and more recent links can be found in the case of the \emph{low-rank matrix completion} problem, which asks for the reconstruction of a large matrix $A\in \mathbb{C}^{m\times n}$ of known rank $r$ from a subset $A_{ij},(ij)\in S$ of known entries, and which is also of great practical interest in engineering and computer science (see \cite{cr2009}, or the introduction~\cite{KTTU12} for an overview). For matrix completion, a link to linear matroids has first been seen by the authors of \cite{sc2010} by several conjectural generalizations of known properties of the rigidity problem to the case of matrix completion.\\

In~\cite{KTTU12}, it has been established for the case of matrix completion, that these relations to linear matroids are consequences of a deep connection to algebraic matroids - the linear matroids are concrete instances of the non-constructive realizability theorems which link algebraic matroids constructed from Galois extensions to linear matroids obtained from their differential, or Jacobian. In particular, there is an algebraic matroid associated to the matrix completion problem which gives rise to the matrix completion problem, which in turn explains the links to linear matroids which have been previously observed.\\

In this paper, we show that to each algebraic matroid over an algebraically closed ground field, there is a \emph{reconstruction problem} of the following type:
\begin{center}
\emph{Reconstruct a point $x$ on an irreducible variety $\calX\subset K^N$ from a fixed set $x_i,i\in S$
of coordinates in a fixed \textbf{basis}.}
\end{center}
Moreover, we will show that the converse is equally true; for each such reconstruction problem, there is an algebraic matroid from which it arises. That is, each algebraic matroid can be represented by such a basis projection problem, properties of the projections giving rise to the independent set of a matroid. Furthermore, if the ground field is not algebraically closed, the matroid can be represented in terms of regular sequences in a ring. These triple correspondence, which is reminiscent of other such hidden equivalences in matroid theory, we term \emph{algebraic cryptomorphism}.

\subsection{Matroids with Graph Symmetry}

A different phenomenon which is observed in the two motivating examples is additional combinatorial structure. For matrix completion, the entries of $A$ are naturally identified with the indices of the matrix $[m]\times [n]$, and the pairwise distances in rigidity with unordered pairs $\binom{[n]}{2}$. Thus, there is an identification between sets of possible observations, and labelled (bipartite) graphs. Since permuting the rows and columns of $A$ or relabeling the points $p_i$ does not affect which entries can be reconstructed, we see that:
\begin{center}
\emph{Whether a set of coordinates determines $x$ depends only on an \textbf{unlabelled graph}.}
\end{center}
Rephrased in terms of the associated algebraic matroid, the corresponding tower of Galois extensions is invariant under the symmetry action of the permutation groups.\\

What we will show that much can already be inferred purely from this combinatorial graph symmetry of the matroids. That is, even when the initial assumption that the matroid is algebraic is dropped, several phenomena which have been observed in both combinatorial rigidity and matrix completion - such as asymptotic moves and limiting behaviour of dimension - are still present in terms of matroid invariants, and can be proven to hold by combinatorial means.

\subsection{Multivariate Minimal Polynomials}

As briefly discussed above, the algebraic cryptomorphism allows to relate algebraic matroids to regular sequences in a polynomial ring. Namely, the reconstruction problem is closely related to the question
\begin{center}
\emph{Given a prime ideal $\frakP$, which sets of coordinates $X_i,i\in S$ give rise to a regular sequence/regular \textbf{coordinates} in the integral domain $K[X_1,\dots, X_n]/\frakP$ ?}
\end{center}
Again, each algebraic matroid gives rise to such a regular coordinate problem, and each coordinate problem gives rise to an algebraic matroid. Through the ring theoretic formulation, we will be able to derive algebraic invariants for algebraic matroids, which generalize classical concepts from Galois theory such as minimal polynomial and extension degree to the multivariate setting of any algebraic matroid. Namely, to each circuit of the underlying algebraic matroid, we associate a single polynomial called \emph{circuit polynomial}, which is a multivariate generator of all finite extensions in the circuit. The collection of its degrees, the so-called \emph{top-degree}, is a further invariant which completely characterizes the degree behavior of those finite extensions.\\

Furthermore, if the assumption of the matroid being algebraic is added again, structural statements about algebraic invariants such as the ideal of the irreducible variety $\calX$ or the possible top-degrees can be inferred. Those, in turn, allow novel insights on the matrix completion and combinatorial rigidity problem, such as on the number of possible reconstructions or the reconstruction process.

\subsection{Main Contributions}
Summarizing, the main contributions in this paper are as follows:
\begin{itemize}
\item We introduce \emph{basis matroids}, arising from a reconstruction problem as described above, and \emph{coordinate matroids}, related to regular sequences in polynomial rings. We show how algebraic matroids can be equivalently - cryptomorphically - realized as either of those and that the realization spaces are isomorphic.
\item We introduce a notion of \emph{matroids with symmetries}, or orbit matroids, and show how the matroid properties extend to those. We identify the main invariants associated to those and derive structure theorems in the case of graph symmetry.
\item We exploit the equivalence of algebraic and basis matroids to obtain invariant polynomials of algebraic matroids - the circuit polynomials - which generalize the minimal polynomial from classical Galois theory.
\item We provide further characterization, structure and finiteness results in the case the matroid is both algebraic and graph symmetric, which partly imply known results in matrix completion and combinatorial rigidity, and partly are novel.
\item We provide a full characterization of some matrix completion and combinatorial rigidity matroids.
\end{itemize}

\subsection{Context and Novelty}
The results of this paper sit at the intersection of several active fields.  Here is the general context.

\subsubsection{Matroid Theory}
One of the most attractive features of matroid theory is the availability of many superficially-different-seeming,
but equivalent axiom systems~\cite[Section 1]{Oxley}.  Our constructions for algebraic matroids follow this
general theme, but with respect to providing different \emph{realizations} that capture the same \emph{algebraic
invariances}.  The program here is also different than the important questions of realizability over different
fields or of understanding the different \emph{orientations} of the same underlying matroid. We furthermore study matroids exhibiting certain characteristic symmetries, which is a reoccurring topic in matroid theory.

Some motives in our work have already appeared in existing work in different context: the matroidal families studied in~\cite{Schmidt197993} can be seen as a specific case of our orbit matroids, compare Remark~\ref{Rem:matrfam}. The circuit polynomial has briefly occurred under different name in~\cite{DL87}.

\subsubsection{Rigidity Theory}
In the preprint \cite{B02}, Borcea formulated framework rigidity in terms of coordinate projections of the
Cayley-Menger variety.  The paper \cite{BS04} then used relationships between determinantal and
Cayley-Menger varieties to derive enumerative results on the number of realizations of minimally
rigid frameworks.  We use some similar ideas to obtain upper and lower bounds on the top-degrees,
but studying circuit polynomials and top-degrees is a new contribution.

\subsubsection{Matrix Completion}
Low-rank matrix completion has been studied extensively, with most of the effort directed towards
heuristics for the reconstruction problem.  Work of Cucuringu-Singer \cite{CS10} on the
decision problem anticipates some of our results, but works in only one of our algebraic
matroid regimes and applies only to the determinantal variety.

It is fair to say that the most powerful tools for
this are based on \emph{convex relaxations}, introduced by Candès-Recht \cite{cr2009}, and then
developed further by a large number of authors (most notably, \cite{CT10}).  Spectral methods \cite{KMO10}
have also given essentially optimal bounds in a number of regimes.  Typical results of either
type give a (a) \emph{``with high probability'' reconstruction guarantee} for the \emph{whole matrix}, given: (b) some
\emph{analytic hypothesis about the true matrix $A$}; and (c) \emph{a sampling hypothesis for
the observations}.

By contrast, the methods here provide much more precise information: under only the assumption
of genericity, we can determine \emph{exactly} the entries which are reconstructable from \emph{any fixed} set of
observations; moreover, top-degrees and circuit polynomials provide algebraic certificates bounding the number of
solutions for any fixed entry, making the decision version of the problem more quantitative.

\subsubsection{Universality versus Genericity}
Let us consider again the reconstruction problem above, with $\pi_{\vec b(S)}$ denoting the projection onto the span of $\{ b_i : i\in S\}$.
It is now a very natural question to consider the following
``decision version'':
\begin{center}
\emph{For $x\in\calX$, and a fixed set of coordinates $S\subseteq E$, determine the dimension and cardinality of $\pi^{-1}_{\vec b(S)}(x)$}.
\end{center}
For frameworks, the decision version is the ``rigidity problem''. If we insist on treating the reconstruction problem for \emph{every} $x\in\calX$, phenomena related to so-called Mnëv-Sturmfels Universality imply that due to a complexity which is arbitrarily bad, even the decision problem has no satisfactory answer - see \cite{LV11,M88,R96,S91} for the theory and \cite{D11} for the specialization to matrix completion).
Our results show that there is only one \emph{generic} behavior of the decision problem - that is, if $x$ is chosen from a generic/Zariski open set, the fiber dimension of $\pi_{\vec b(S)}$ is controlled completely by the rank of an \emph{algebraic matroid} on $E$, which does not depend on $x$ anymore.  What this means, is that the reconstruction problem has a \emph{combinatorial part}, captured by a matroid $\calM = \vec B_K(\calX,\vec b)$ on $E$, and an
\emph{algebraic-geometric} part which depend on $\calX$ and $\vec b$.\\

However, the algebraic-geometric, or universal, behaviour occurs only on a Zariski closed, and therefore probability zero, subset of the variety $\calX$. Therefore, while universality makes statements of the kind
\begin{center}
\emph{Given some arbitrarily (universally) complex behaviour, there is a variety $\calX$, a projection $\pi$, and a point $x\in\calX$ where it occurs,}
\end{center}
our results can be interpreted as a more optimistic converse
\begin{center}
\emph{Given some arbitrarily complex variety $\calX$, the behaviour of an an arbitrary projection $\pi$ at a generical point $x\in \calX$ is governed by an algebraic matroid.}
\end{center}
Arguing that the variety $\calX$ is usually fixed in the applications, and $x$ is usually generic, we conclude that universality does not occur in such applications, making way to efficient algorithms - such as those discussed in~\cite{KTTU12} - that exploit genericity.

\subsection{A Reading Guide to the Main Theorems}
In this paper, we generalize and strengthen the results mentioned above.  The determinantal
matroid in its two realizations (by an ideal and a variety) is an instance of a construction
that applies to \emph{any irreducible variety}.  Moreover, certain structural results
derived by algebraic means in \cite{KTTU12} result only from the underlying symmetry, and
can be proved purely combinatorially.  Here is an overview of the results.

\subsubsection{Algebraic Matroids and Circuit Polynomials}
Our first results are about different ways of realizing algebraic matroids.  Full definitions
are given in Section \ref{sec:matroids}.
\begin{cryptointro}%
If $\calM$ is an algebraic matroid on ground set $E$, it has a realization as any of: (a) a field and an extension by a
set of transcendentals; (b) an irreducible variety and a specific coordinate system; (c) a prime ideal and
a specific choice of coordinate variables.  Moreover, there are natural constructions for converting
from one type to the other two, and these constructions commute.
\end{cryptointro}
Theorem \ref{Thm:crypto} says that, depending on what is more convenient for an application,
we can switch algebraic formalisms \emph{without losing any} of the additional structure carried by a
specific realization.  This should be compared with classical results saying, e.g., that
all algebraic matroid are linear in characteristic zero.

As mentioned above, we can associate to the circuits of an algebraic matroid polynomial invariants called
\emph{circuit polynomials}.  These witness the minimal algebraic dependencies among the circuits in $\calM$.
\begin{circpolythm}
Let $\calM = \vec C_K(\frakP,\bf x)$ be an algebraic matroid realized via an ideal $\frakP$.  For
each circuit $C\in \calC(\calM)$, there is a unique (up to scalar multiplication) \emph{circuit polynomial}
$\theta_{\vec x(C)}(\frakP)\in \frakP$.  The circuit polynomials are the minimal supports in $\frakP$.
\end{circpolythm}
In our applications, the uniqueness of circuit polynomials implies that, for a circuit, $C$ containing
an entry $(i,j)$, there is exactly one way to solve for the variable $\vec x(i,j)$ given the
variables $\vec x(C\setminus (i,j))$.
Thus, the degree of  $\vec x(i,j)$ in the circuit polynomial of $C$  yields information about the the fiber
cardinality of $\pi_{\vec b(C\setminus (i,j))}$
at a generic point.  Considering the degrees of every variable in a circuit polynomial
at once leads to the notion of \emph{top-degrees} of circuit polynomials and, indeed, of any polynomial
in an algebraic matroid's underlying ideal.

\subsubsection{Bipartite Graph Matroids}
Now consider the general case of a matroid $\calM$ defined on $[m]\times [n]$,
with the property that for all $S\subset [m]\times [n]$, $\sigma\in \frakS(m)$ and $\tau\in \frakS(n)$,
the set $\left\{(\sigma(i),\tau(j)) : (i,j)\in S\right\}$ has the same rank in $\calM$.  We call
such a matroid a \emph{bipartite graph matroid}, since the rank depends only on the
\emph{unlabeled bipartite graph} corresponding to $S$.  This theory is developed in
Section \ref{sec:size}.

Using only the symmetries we can show:
\begin{finitenessthm}
Let $\calM$ be a bipartite graph matroid on the ground set $[m]\times[n]$, with $m, n\gg 0$.  Then
for each $m' \le m$, there is an $n'$, depending only on $m'$ such that any circuit in $\calM$
supported on at most $m'$ rows has support on at most $n'$ columns.
\end{finitenessthm}
In particular, if we regard $m$ as fixed and $n$ increasing to infinity, we see that the number of
isomorphism classes of circuits is finite, answering a question Bernd Sturmfels asked us.  On the other
hand, when both indices grow, we construct infinite families of circuits in the determinantal and rigidity
matroids.

\subsubsection{Structure of Algebraic Bipartite Graph Matroids}
Combining the algebraic and combinatorial theories together, we are able to give similar
finiteness and non-finiteness statements in Section \ref{sec:structure} for \emph{algebraic bipartite graph matroids}.
Here, the combinatorial objects of study are \emph{top-degrees} of polynomials, which are matrices recording the
degree in each variable.
\begin{topdegthm}
Let $\calM = \vec C_K(\frakP,\vec x)$ be an algebraic bipartite graph matroid on the ground set
$[m]\times[n]$, with $m, n\gg 0$.
Then for each $m' \le m$, there is an $n'$, depending only on $m'$ such that any minimal top-degree in $\calM$
supported on at most $m'$ rows has support on at most $n'$ columns.  This, in particular, implies that the
number of solutions to any minimal top-degree is bounded by a function of the size of its row support.
\end{topdegthm}
In particular, this says that, for a fixed number of rows, the complexity of the reconstruction problem is bounded.
As before, on a square ground set, the complexity increases rather rapidly, as we will show via examples.
Section \ref{sec:examples}
contains a detailed treatment of small determinantal matroids, for which we can compute all the
matroidal structures and polynomial
invariants exactly.

\subsubsection{Multivariate Galois Theory}
By combining the notion of top-degree with circuit polynomials, we
obtain the following reformulation of Lemma \ref{Lem:circuitpoly} in terms of field extensions.
\begin{multipolythm}
Let $L = K(\alphabf)$ be a field extension by transcendentals
$\alphabf = (\alpha_1,\ldots,\alpha_n)$, and suppose that the
$\alpha_i$ form an algebraic circuit.  There is a (unique up to scalar multiplication)
polynomial $\theta$ with top degree $t_i = [L : K(\alphabf\setminus \alpha_i)]$ in each variable $\alpha_i$ and
$\theta(\alphabf) = 0$.
\end{multipolythm}
Here is the interpretation.  We first recall the minimal polynomial from Galois theory
of an element $\alpha$ of a field extension $L/K$; this is a univariate polynomial $\theta$
of degree $[L : K]$, with the property that $\theta(\alpha) = 0$.  This is unique, up to
a scalar in $K^\times$, and it encodes $\alpha$, the degree of $L/K$, and (via the conjugate
solutions of $\theta$), information about the symmetries of $L$.  The polynomial $\theta$
provided by Theorem \ref{Thm:multpoly} is, then, the analogous object for
extensions by multiple elements: the top-degrees encode the degrees of the ``one element''
extensions by each variable, and symmetries of $\theta$ mirror those of the underlying algebraic
matroids.

\subsection{Acknowledgements}
We thank Bernd Sturmfels for helpful comments. We thank Winfried Bruns for insightful discussions and
many helpful comments.

LT is funded by the European Research Council under the European Union's Seventh
Framework Programme (FP7/2007-2013) / ERC grant agreement no 247029-SDModels.
This research was initiated at the Mathematisches Forschungsinstitut Oberwolfach,
when LT and ZR were guests supported by FK's Leibniz Fellowship.

\section{Matroid Realizations and their Symmetries}\label{sec:matroids}

\subsection{Preliminaries}
We first set some notation and recall standard definitions and results required in the sequel.  The
basic notions of matroids, and in particular, over finite sets can be found in the monograph~\cite{Oxley}.

\begin{Not}\label{Not:matr}
We adopt the convention of a matroid $\matrex$ being an ordered pair $(E,\calI)$, with $E=E(\matrex)$ being the (finite) \emph{ground set} of $\matrex$ (where $E$ can be a multiset), and $\calI\subseteq \calP(E)$ the collection of independent sets of $\matrex$ (each element of $\calI$ is a submultiset of $E$). However, instead of repeatedly specifying a matroid as a pair, for readability, we will specify each matroid by its ground set $E(\matrex)$, and by making the following invariants explicit, each sufficient for an equivalent definition of $\matrex$:\\
\begin{tabular}{cl}
$\calI(\matrex)$ & the set of independent sets of $\matrex$\\
$\calD(\matrex)$ & the set of dependent sets of $\matrex$\\
$\calB(\matrex)$ & the set of bases of $\matrex$\\
$\calC(\matrex)$ & the set of circuits of $\matrex$\\
$\rk_\matrex(\cdot):\calP(E)\rightarrow \NN$ & the rank function of $\matrex$
\end{tabular}\\
$\calI(\matrex),\calD(\matrex),\calB(\matrex),\calC(\matrex)$ are all subsets of the power set $\calP(E)$. If clear from the context, we will omit the dependence on $\matrex$.
\end{Not}

As usual, we consider two matroids isomorphic if their is an isomorphism on the ground sets inducing an isomorphism on any of $\calI,\calD,\calB,\calC$ or $\rk$:

\begin{Def}
Let $\matrex$ and $\calN$ be matroids on ground sets $E$ and $F$. A \emph{morphism} of matroids is a pair
$(\varphi,\psi)$ of maps $\varphi :\calI(\matrex)\rightarrow \calI(\matrex)$ and $\psi: E\rightarrow F$, such that, for
all $S\in \calI$
\[
\varphi(S) = \{ \psi(s) : s\in S \}
\]
We will say that $\psi$ \emph{induces the morphism} from $\matrex$ to $\calN$. A morphism $(\varphi,\psi)$ is an
\emph{isomorphism} if $\varphi$ and $\psi$ are bijections.  A self-isomorphism of a matroid is called an
\emph{automorphism}.  When it is clear from the context, we will use the same symbol for both maps.
\end{Def}

Therefore, a matroid $\matrex$ is considered as a purely \emph{combinatorial} object, which does not formally carry any information on what $E,\calI,$ etc.~describe.\\

Conversely, one can construct matroids which \emph{describe} dependences on a certain ground set which comes with algebraic or combinatorial structure. Obtaining a matroid $\matrex$ in this way is called a \emph{realization} of $\matrex$, associated to the ground set with the dependence structure.
\begin{Not}
Let $\matrex=(E,\calI)$. A realization of $\matrex$ is some data $D$, such that $\matrex$ can be obtained from some constructors $E=E(D)$ and $\calI=\calI(D)$. The construction process, if clear from the context, will be understood. Usually, the data $D$ also contains some (multi-)set $\dbf=\{d_e\;:\;e\in E\};$ in this case, for $S\subseteq E$, we will denote $\dbf (S):=\{d_s\in\dbf\;:\; s\in S\}.$
\end{Not}

We will think of a realization as a matroid which formally carries the additional structure given by the data $D$ and the associated construction process, as compared to the matroid alone which has no structure or constructor associated. We illustrate this subtle difference in the well-known case of linear matroids~\cite[Proposition~1.1.1]{Oxley}:

\begin{Def}\label{Def:linmat}
Let $E$ be some finite ground set, let $K$ be a field. Let $\vbf = \{v_e\;:\; e\in E\}$ be a collection of elements $v_e\in V$ (possibly a multiset), where $V$ is some vector space over $K$. We construct a matroid $\matrlin_K (\vbf)$ from $\vbf$ and $K$ in the following way:\\
\begin{tabular}{cl}
$E$ & is the ground set\\
$S\in\calI$ & iff $\vbf(S)$ is linearly independent over $K$\\
$S\in\calD$ & iff $\vbf(S)$ is linearly dependent over $K$\\
$S\in\calB$ & iff $\vbf(S)$ is a basis of the $K$-vector space $\spn\vbf$\\
$S\in\calC$ & iff $\vbf(S)$ is a linear circuit over $K$\\
$\rk:\calP(E)\rightarrow \NN,$ & $\vbf(S)\mapsto \dim\spn \vbf(S)$
\end{tabular}\\
We call $\matrlin_K (\vbf)$ (and therefore any matroid isomorphic to it) a \emph{linear matroid} (over the field $K$). The data $(K,\vbf)$ is called a \emph{linear realization} of the matroid $\matrlin_K (\vbf)$ over the field $K$, and $\matrlin_K (\vbf)$ is said to be \emph{realized} by $(K,\vbf)$.
\end{Def}

Another important way of realization is through field extensions in the case of algebraic matroids~\cite[Theorem~6.7.1]{Oxley}::

\begin{Def}\label{Def:algmat}
Let $E$ be some finite ground set, let $K$ be a field. Let $\alphabf = \{\alpha_e\;:\; e\in E\}$  be a collection of elements $\alpha_e\in L$ (possibly a multiset), where $L$ is some finite field extension of $K$. We construct a matroid $\matralg_K (\alphabf)$ from $\alphabf$ and $K$ in the following way:\\
\begin{tabular}{cl}
$E$ & is the ground set\\
$S\in\calI$ & iff $\alphabf(S)$ is algebraically independent over $K$\\
$S\in\calD$ & iff $\alphabf(S)$ is algebraically dependent over $K$\\
$S\in\calB$ & iff $\alphabf(S)$ is a basis for the field extension $K(\alphabf)/K$\\
$S\in\calC$ & iff $\alphabf(S)$ is an algebraic circuit over $K$\\
$\rk:\calP(E)\rightarrow \NN,$ & $\alphabf(S)\mapsto \trdeg K(\alphabf)/K$
\end{tabular}\\
We call $\matralg_K (\alphabf)$ (and therefore any matroid isomorphic to it) an \emph{algebraic matroid} (over the field $K$). The data $(K,\alphabf)$ is called an \emph{algebraic realization} of the matroid $\matralg_K (\alphabf)$ over the field $K$, and $\matralg_K (\alphabf)$ is said to be \emph{realized} by $(K,\alphabf)$.
\end{Def}

Note that by our convention, saying that a matroid is linear or algebraic is a \emph{qualifier} for a matroid, which does not formally include the specification of how - or in which ways - the matroid can be obtained from data. Conversely, for realizations, we will always think of the matroid associated to it by canonical construction.\footnote{An alternative way of formalizing this would be to consider the category of matroids and categories of realizations, e.g., linear matroids or algebraic matroids. Each of the realization categories comes with a forgetful functor to matroids. Saying that a matroid can be realized means that it is in the image of the respective forgetful functor. We refrain from doing so in order to avoid confusion and unnecessary notational overhead.} These two viewpoints give rise to different flavours of questions: namely, which matroids can be realized by a certain constructor, e.g., which matroids are linear, algebraic etc., and if such a matroid is given, what are the possible realizations for it? When considering more than one constructor, e.g., linear and algebraic, one might ask how being realizable by the one and being realizable by the other relate, and if both concepts coincide, how do the realizations relate?\\

For illustration, we phrase a well-known theorem on linear and algebraic representability~\cite[Propositions~6.7.10f]{Oxley} in these terms:
\begin{Thm}\label{Thm:alglinmatr}
Let $\matrex$ be a matroid, let $K$ be a field of characteristic $0$.  The following statements are equivalent:
\begin{itemize}
\item[(a)] $\matrex$ is realizable by a linear matroid over $K$.
\item[(b)] $\matrex$ is realizable by an algebraic matroid over $K$.
\end{itemize}
Moreover, given a presentation
$$\matrex =\matralg_K(\alphabf) = \matrlin_K(\vbf)$$
there are constructions producing any of $(K,\alphabf), (K,\vbf)$, given the other.
\end{Thm}

It is a difficult and unsolved question to characterize all possible linear and algebraic realizations, and to relate them to each other - a question which will not be the main concern of this paper. Instead, we will be interested in matroids that are realized by natural algebraic constructions - namely, matroids associated to varieties, and to prime ideals. We will be interested in how these different kinds of realizations relate to each other, and we will use these relations to derive structure theorems for those matroids in the presence of certain symmetries.

\subsection{Matroids of Ideals and Varieties}\label{Sec:cryptomatroids}
In the following, we provide two novel classes of matroid realizations, basis and coordinate matroids. We will show later in Theorem~\ref{Thm:crypto} that realizations as algebraic, basis and coordinate matroids can be canonically transformed into each other, therefore yield the same class of matroids. Later, when studying their structure and symmetries, it will turn out that each of the three representation has its situational advantage.

\begin{Def}\label{Def:matrvar}
Let $E$ be some finite ground set, let $K$ be a field. Let $\bbf = \{b_e\;:\; e\in E\}$ be a basis of $K^n$, let $\calX\subset K^n$ be an irreducible variety. We construct a matroid $\matrbas (\calX, \bbf)$ from $\bbf$ and $\calX$ in the following way:\\
\begin{tabular}{cl}
$E$ & is the ground set\\
$S\in\calI$ & iff the canonical projection map $\pi_{\bbf (S)} : \calX\to \spn \bbf (S)$ is surjective\\
$S\in\calD$ & iff the canonical projection map $\pi_{\bbf (S)} : \calX\to \spn \bbf (S)$ is not surjective\\
$S\in\calB$ & iff the canonical projection map $\pi_{\bbf (S)} : \calX\to \spn \bbf (S)$ is finite surjective
\end{tabular}\\
We call $\matrbas (\calX, \bbf)$ (and therefore any matroid isomorphic to it) a \emph{basis matroid} over $K$. The data $(\calX, \bbf)$ is called an \emph{basis realization} of the matroid $\matrbas (\calX, \bbf)$, and $\matrbas (\calX, \bbf)$ is said to be \emph{realized} by $(\calX, \bbf)$.
\end{Def}

\begin{Rem}\label{Rem:projdim}
Because $\calX$ is irreducible, any image $\pi_{\bbf (S)} (\calX)$ also is for every $S$. Therefore, an alternative definition for the matroid can be obtained by defining $S\in\calI$ iff $\dim \pi_{\bbf (S)} (\calX) = \# S$.
\end{Rem}

\begin{Def}\label{Def:idealmat}
Let $E$ be some finite ground set, let $K$ be a field. Let $\xbf = \{X_e\;:\; e\in E\}$ be a set (not a multiset) of coordinate variables, let $\frakP\subseteq K[X_e,e\in E]$ be a prime ideal. We construct a matroid $\matrcrd (\frakP, \xbf)$ from $\xbf$ and $\frakP$ in the following way:\\
\begin{tabular}{cl}
$E$ & is the ground set\\
$S\in\calI$ & iff $\xbf(S)$ is a regular sequence modulo $\frakP$\\
$S\in\calD$ & iff $\xbf(S)$ is not a regular sequence modulo $\frakP$\\
$S\in\calB$ & iff $\xbf(S)$ modulo $\frakP$ is a generating regular sequence for the ring $K[\xbf]/\frakP$
\end{tabular}\\
We call $\matrcrd (\frakP, \xbf)$ (and therefore any matroid isomorphic to it) a \emph{coordinate matroid} over $K$. The data $(\frakP, \xbf)$ is called a \emph{coordinate realization} of the matroid $\matrcrd (\frakP, \xbf)$, and $\matrcrd (\frakP, \xbf)$ is said to be \emph{realized} by $(\frakP, \xbf)$.
\end{Def}

That the definitions for the basis/coordinate matroid via any of $\calI,\calD,\calB$ are equivalent follows from properties of algebraic maps and regular sequences; that both give rise to matroids is not. However, one can infer from the proof of the subsequent Theorem~\ref{Thm:crypto} that basis and coordinate matroids are indeed matroids, therefore we state this explicitly to assert well-definedness:

\begin{Prop}
The ground set/independent set pairs, as defined in Definitions~\ref{Def:matrvar} and~\ref{Def:idealmat} for basis and coordinate matroids, are matroids (in the sense of Notation~\ref{Not:matr}).
\end{Prop}

However, even more is true. Algebraic, basis and coordinate realizations can be regarded as different realizations of the same kind of matroid, all three existing whenever one of the three does. Moreover, the following theorem asserts that there are canonical, commuting ways to obtain any of these realizations from the other:

\begin{Thm}\label{Thm:crypto}
Let $\matrex$ be a matroid, let $K$ be an algebraically closed field. The following statements are equivalent:
\begin{itemize}
\item[(a)] $\matrex$ is algebraic over $K$
\item[(b)] $\matrex$ is a basis matroid over $K$.
\item[(c)] $\matrex$ is a coordinate matroid over $K$.
\end{itemize}
Moreover, given a presentation
$$\matrex =\matralg_K(\alphabf) = \matrbas(\calX,\bbf) = \matrcrd(\frakP,\xbf),$$
there are canonical constructions producing any two of $(\alphabf,K), (\calX,\bbf), (\frakP,\xbf)$ from the third, such that these constructions commute.
\end{Thm}
\begin{proof}
Without loss of generality, and for ease of notation, we can assume that the ground set of $\matrex$ is $[n]$ by isomorphism.
The following paragraphs construct the realizations for $\matrex$ on the right side of the arrow, assuming the left side is a realization for $\matrex$. In particular, for $S\subseteq [n]$ it is shown that $S\in \calI$ on the left side if and only if $S\in \calI$ on the right side.

\paragraph{$\matralg_K(\alphabf)\to \matrcrd(\frakP,\xbf)$:} Let $R = K[\alphabf]$.  Since $\Frac R \subset L$ exists,
$R$ is an integral domain.  This implies that the kernel of the canonical surjection $K[X_1,\ldots, X_n]\to R$ is a
prime ideal $\frakP$ and that $R = K[X_1,\ldots, X_n]/\frakP$.  By the construction, $\alphabf(S)$ is
algebraically independent if and only if $\xbf(S)$ is a regular sequence on
$K[X_1,\ldots, X_n]/\frakP.$%

\paragraph{$\matrcrd(\frakP,\xbf) \to \matrbas(\calX,\bbf)$:}  Let $\calX = \Van(\frakP)$.  Since $\frakP$ is prime, the
variety $\calX$ is irreducible.  For each $i\in [n]$ the variety $\calV_i$
defined by $X_1 = X_2 = \cdots = X_{i-1}=X_{i+1} = \cdots X_n = 0$ is a line.  Pick $b_i$ as any non-zero vector in
$\calV_i$. Since $K$ is algebraially closed, $\xbf(S)$ is a regular sequence on $K[X_1,\ldots, X_n]/\frakP$ if and only if
the projection $\pi_{\bbf(S)}$ is surjective.%

\paragraph{$\matrbas(\calX,\bbf) \to \matralg_K(\alphabf)$:} For each $i\in E$, let $\calV_i:=\spn b_i$. Define $\calA := \calV_1\times \calV_2\times \cdots \times \calV_n$. It holds that $K(\calV_i)=K(X_i)$ for some transcendental variable $X_i$. The canonical map $\iota: \calX\hookrightarrow \calA$ then induces a canonical field homomorphism $\psi : K(X_1,\ldots, X_n)\to K(\calX)$ which is surjective since $\iota$ is injective. Taking $L = K(\calX)$ and
$\alpha_i = \psi(X_i)$ we see that, by construction, for $S\subset [n]$, the set $\alphabf(S)$
is algebraically independent in $L/K$ if and only if the projection $\pi_{\bbf(S)}$ is surjective.
\end{proof}

There are several remarks in order.
\begin{Rem}\label{Rem:algclos}
If $K$ is not algebraically closed in Theorem~\ref{Thm:crypto}, e.g., the real numbers $\RR$, it can happen that the variety $\calX$ is empty, therefore the correspondence does not hold. However, the equivalence can be made to hold for general $K$ by replacing the variety $\calX$ by the corresponding scheme, with the proof of Theorem~\ref{Thm:crypto} holding almost verbatim.
\end{Rem}

\begin{Rem}\label{Rem:rellin}
Theorem~\ref{Thm:crypto} is, by a subtlety, stronger than the analogue for linear versus algebraic matroids given by Theorem~\ref{Thm:alglinmatr}: the constructors producing realizations in Theorem~\ref{Thm:crypto} for algebraic/basis/coordinate matroids commute, i.e., by converting realizations to each other in a closed cycle, the conversion constructors will yield isomorphic data. It seems to be unclear whether the constructors yielding algebraic from linear data, see~\cite[Propositions~6.7.10]{Oxley}, and linear from algebraic data, see~\cite[Propositions~6.7.11 and~6.7.12]{Oxley}, can be chosen to commute. Therefore, the constructors in Theorem~\ref{Thm:crypto} for the algebraic types of matroid are canonical, while the ones in Theorem~\ref{Thm:alglinmatr} given for linear and algebraic matroids are not, or not known to be.
\end{Rem}

\begin{Rem}\label{Rem:universality}
For an algebraic matroid $\matrex$, the \emph{realization space} is defined as
\[
\calR:= \{(K,\alphabf) : \text{$\matralg_K(\alphabf)=\matrex$} \}.
\]

The constructions used to prove Theorem \ref{Thm:crypto} supply canonical, commuting bijections
between $\calR$ and the spaces
\begin{align*}
\calS &:= \{(\calX,\bbf) : \text{$\matrbas(\calX,\bbf)=\matrex$} \},\quad\mbox{and}\\
\calT &:= \{(\frakP,\xbf) : \text{$\matrcrd(\frakP,\xbf)=\matrex$} \}.
\end{align*}
In particular, these spaces are canonically isomorphic to each other,
and thus, the second part of Theorem \ref{Thm:crypto} represents a substantial
strengthening of the first part.

Another way of looking at this phenomenon is via the canonical forgetful functors from the category of the realizations to the category of matroids, the realization spaces being fibers of the respective functors.
\end{Rem}
In light of the discussion above, the following definition is non-trivial.
\begin{Def}\label{Def:cryptomorphic}
Let $\matrex$ be a matroid, such that $\matrex =\matralg_K(\alphabf) = \matrbas(\calX,\bbf) = \matrcrd(\frakP,\xbf).$
The realizations $(K,\alphabf)$,$(\calX,\bbf)$ and $(\frakP,\xbf)$
of $\matrex$ are called \emph{cryptomorphic realizations of $\matrex$} if they can be obtained from each
other via the canonical constructions of Theorem \ref{Thm:crypto}.
\end{Def}
We want to remark that Theorem~\ref{Thm:crypto} yields an alternative proof for part of the equivalence of linear and algebraic matroids presented earlier in Theorem~\ref{Thm:alglinmatr}, together with a structural interpretation.
\begin{Prop}\label{Prop:cryptocharzero}
Let $\matrex$ be an algebraic matroid, with cryptomorphic realizations
$$\matrex =\matralg_K(\alphabf) = \matrbas(\calX,\bbf) = \matrcrd(\frakP,\xbf).$$
Assume that $K$ has characteristic zero. Let $\Omega_{L/K}^1$ be the $K$-vector space of K\"ahler differentials of $L/K$, and let
$\diff: L\rightarrow \Omega^1_{L/K}$ be the canonical map. Furthermore, let $\xi$ be a generic\footnote{chosen from a suitable open dense subset of $\calX$} closed point of $\calX$, and denote by $\diff_\xi: L\rightarrow \Omega^1_{L/K}\otimes_K K(\xi)$ the canonical map which is obtained from composition with evaluation at $\xi$. Then, the following matroids are isomorphic:
\begin{itemize}
\item[(i)] the algebraic matroid $\matralg_K(\alphabf).$
\item[(ii)] the linear matroid $\matrlin_{K'}(\vbf)$, where $K'=K(\calX),$ and $\vbf = \{\diff \kappa\;:\;\kappa\in\alpha\}$.
\item[(iii)] the linear matroid $\matrlin_{K}(\vbf),$ where $\vbf = \{\diff_\xi \kappa\;:\;\kappa\in\alpha\}$.
\end{itemize}
\end{Prop}
\begin{proof}
(a)$\Leftrightarrow$ (b): Use Theorem \ref{Thm:crypto} to pass to $\matrbas(\calX,\bbf)$, and observe that, in for polynomial rings the
canonical map is given by the usual partial derivatives. For $S\subset [n]$, and characteristic zero,
the projection $\pi_{\bbf(S)}$ is surjective exactly when its Jacobian matrix has full rank.\\

(b)$\Leftrightarrow$ (c): the Jacobian matrix has full rank if and only if it has full rank when evaluated on an open dense subset of $\calX$.
\end{proof}
Proposition~\ref{Prop:cryptocharzero} implies that algebraic matroids are linear, and in particular that a linear realization can be obtained over the same base field, compare~\cite[Propositions~6.7.11 and 6.7.12]{Oxley}, or~\cite{L89}. It furthermore provides a structural statement how to obtain a linear realization from an algebraic one.
We finish this subsection by describing the rank functions.
\begin{Prop}\label{Prop:rankcrypto}
Consider a matroid $\matrex$ on the ground set $E$ with its cryptomorphic realizations
$$\matrex =\matralg_K(\alphabf) = \matrbas(\calX,\bbf) = \matrcrd(\frakP,\xbf).$$
Then, the following rank functions $\rk_.:\calP (E)\to \NN$, defined by taking an arbitrary $S\subseteq E$, are equal:
\begin{itemize}
\item[(a)] $\rk_{\matralg_K(\alphabf)}(S) := \trdeg K(\alphabf(S))/K$, where $\trdeg$ denotes transcendence degree.
\item[(b)]$\rk_{\matrbas(\calX,\bbf)}(S) := \dim \pi_{\bbf(S)}(\calX)$, where $\pi_{\cdot}$ is the projection as in
Definition \ref{Def:matrvar}, and $\dim$ denotes Krull dimension.
\item[(c)] $\rk_{\matrcrd(\frakP,\xbf)}(S) := \dim \left(K[\xbf]/\frakP \cap K[\xbf(S)]\right)$, where $\dim$ denotes Krull dimension.
\item[(d)] $\rk_{\matrex}(S)$, the rank function of $\matrex$.
\end{itemize}
\end{Prop}
\begin{proof}
Recall that the rank of a set $S$ is just the size of a maximal independent set contained in
$S$.  The proposition now follows from chasing definitions: (a) the transcendence degree is the
size of a maximal set of algebraically independent elements of $\alphabf(S)$; (b)
follows from Remark \ref{Rem:projdim}; (c) follows from the fact that all the rings here
have the Cohen-Macaulay property.
\end{proof}

\subsection{Graph Matroids}

Some ground sets come equipped with natural symmetry actions.
The motivating example for this paper is the set $E=[m]\times[n],$ with the group $G=\frakS(m)\times \frakS(n)$ acting freely on the factors. This is motivated by the fact that $G$ is the symmetry group of the complete bipartite graph $K_{m,n}$ (arrange the edges in a table), and choosing a subset of $E$ up to $G$-action can be viewed as choosing a set of edges of $K_{m,n}$.

Similarly, if $\matrex$ is a matroid on the ground set $E=[m]\times[n]$ and and $G=\frakS(m)\times \frakS(n)$ is a group of automorphisms of $\matrex$,
the the rank of a set $S\subseteq E$ is invariant under the $G$-action, thus depending only on the isomorphism type of $S$ as a bipartite subgraph of $K_{m,n}$. This motivates our terminology of graph matroids. Having such a graph symmetry is quite a strong property for $\matrex$, and we will see in the sequel that such matroids have many interesting properties, especially when they are also algebraic (or, equivalently, basis/coordinate matroids).

We start by setting the notation:
\begin{Not}
Let $E$ be a set and $G$ a group of automorphisms of $E$. For $S\subseteq E,$ we will denote by $S/G=\{\sigma S\;:\;\sigma\in G\}$ the set of $G$-orbits\footnote{We deliberately write $S/G$ instead of the usual notation $GS$ or $G\cdot S$, since we \emph{identify} certain elements of the power set $\calP (E)$ under the $G$-action, therefore yielding a concept closer to a categorical quotient of $\matrex$ than of a categorical orbit.} of $S$; similarly, for the power set $\calP(E)$, we will denote by $\calP(E)/G :=\{S/G\;:\; S\subseteq E\}$. Note that there is a canonical surjection $\calP(E)\rightarrow \calP(E)/G,$ sending $G$-stable subsets $\calI\subseteq \calP(E)$ to subsets $\calI/G\subseteq \calP(E)/G$. Also note, that in general $\calP(E)/G$ is \emph{not} the same as $\calP(E/G)$.
\end{Not}

If the ground set of a matroid has a symmetry, then whether a subset is independent only depends on the orbit. We formalize this with the following notion:
\begin{Def}\label{Def:matrsymm}
Let $\matrex = (E,\calI)$ be a matroid, let $G$ be a group of automorphisms of $\matrex$. The pair $\matrex/G := (E/G, \calI/G)$ is called the corresponding \emph{orbit matroid}. Independent, dependent sets, circuits, and bases of $\matrex/G$ are elements of $\calP(E)/G$ which are images of independent, dependent sets, circuits and bases of $\matrex$ under the quotient map $\calP(E)\rightarrow \calP(E)/G$. Furthermore, we introduce terminology for specific symmetries: \vspace{-2mm}
\begin{center}
\begin{tabular}{ |p{.1\textwidth}|c|p{.5\textwidth}|}
\hline
{\bf Notation}  &  {\bf Ground Set $E$} & {\bf Symmetry Group $G$} \\      \hline
{\it Graph }   & ${[n]\choose 2}:=[n]\times [n]/\frakS(2) $  & $\frakS(n)$ acting diagonally on $[n]\times [n]$, i.e., $\sigma (i,j)=(\sigma i,\sigma j)$. \\   \hline
{\it (Bipartite) Graph }     & $[m]\times [n]$  & $\frakS(m)\times \frakS(n)$ acting on $E$ by the canonical product action, i.e., $(\sigma,\tau) (i,j)=(\sigma i,\tau j)$.   \\
\hline
\end{tabular}
\end{center}
The canonical symmetry group $G$ for any type of matroid $\matrex$ in the table above will be assumed to be a group of automorphisms, when talking about graph matroids, or bipartite graph matroids, and will be denoted by $\frakS(\matrex)$.\\

An orbit matroid $\matrex=(E,\calI)$ with (bipartite) graph symmetries will be called a \emph{(bipartite) graph matroid}.
We will identify the ground set of a bipartite graph matroid with the complete bipartite graph, and we will write $K_{m,n}=[m]\times [n] / \frakS(m)\times \frakS(n)$; similarly, we identify the ground set of a graph matroid with the complete graph, writing $K_n={[n]\choose 2}/\frakS (n)$. Similarly, we identify subsets of $\calP(E)/G$ with the corresponding subgraphs of $K_{m,n}$ and $K_n$. As for matroids, these subgraphs will fall into different sets:\\
\begin{tabular}{cl}
$\calI(\matrex/G)=\{S/G\;:\;S\in\calI(\matrex)\}$ & the set of independent graphs of $\matrex/G$\\
$\calD(\matrex/G)=\{S/G\;:\;S\in\calD(\matrex)\}$ & the set of dependent graphs of $\matrex/G$\\
$\calB(\matrex/G)=\{S/G\;:\;S\in\calB(\matrex)\}$ & the set of basis graphs of $\matrex/G$\\
$\calC(\matrex/G)=\{S/G\;:\;S\in\calC(\matrex)\}$ & the set of circuit graphs of $\matrex/G$
\end{tabular}\\
That is, we will also call independent, dependent sets, circuits and bases of $\matrex/G$ {\it independent, dependent, circuit and orbit graphs}. If clear from the context, we will not make a sharp distinction between $\matrex$ and $\matrex/G$ and, for example, talk about circuit graphs of $\matrex$ when meaning a circuit of $\matrex/G$. Similarly, if $\calG$ is the graph corresponding to some $E\subseteq [m]\times [n]$, we will also write $\rk(\calG):=\rk(E)$.
\end{Def}

Conceptually, the ground set/subset structure of a matroid is replaced by a graph/subgraph structure in the case of graph matroids, or more generally, an orbit/suborbit structure for orbit matroids.

\begin{Rem}
While in this paper we will focus on graph and bipartite graph matroids, it is straightforward to apply the concept of orbit matroids in order to model other combinatorial objects:
\begin{center}
\begin{tabular}{ |p{.15\textwidth}|c|p{.5\textwidth}|}
\hline
{\bf Notation}  &  {\bf Ground Set $E$} & {\bf Symmetry Group $G$} \\      \hline
{\it Multiset }     & $[n]^{\times N}$  & $\frakS(N)$ permuting the $N$ copies of $[n]$.   \\ \hline
{\it Directed Graph }     & $[n]\times [n]$  & $\frakS(n)$ acting diagonally on $[n]\times [n]$.   \\ \hline
{\it Multigraph}     & ${[n]\choose 2}^{\times N}$  & $\frakS(n)\times \frakS (N)$, where $\frakS(N)$ permutes the $N$ copies of ${[n]\choose 2}$.   \\ \hline
$d$-{\it hypergraph}   & ${[n]\choose d}:=[n]^{\times d}/\frakS(d)$ &  $ \frakS(n)$ acting diagonally on $[n]^{\times d}$. \\   \hline
($d$-partite) \linebreak $d$-{\it hypergraph}  & $[n_1]\times \dots\times [n_d]$ & $\frakS(n_1)\times\dots\times \frakS(n_d)$ acting on $E$ by the canonical product action. \\
\hline
\end{tabular}
\end{center}
The types of orbit matroids above will however not be discussed further in the paper.
\end{Rem}

For better readability, we now provide some notation for matroids with (bipartite) graph symmetry:

\begin{Def}
Let $\matrex = ([m]\times [n],\calI)$  be a matroid with bipartite graph symmetry, let $S\subseteq [m]\times [n]$. The \emph{vertex support} of $S$ is the set
$$\vsupp S := E_1\times E_2,\quad\mbox{with}\; E_1 = \{i\;:\;(i,j)\in S\}\;\mbox{and}\;E_2 = \{j\;:\;(i,j)\in S\}.$$
The \emph{signature} of $S$ is the pair of numbers $(m',n')$, where
\begin{align*}
m' &= \# \{i\;:\;(i,j)\in S\}\\
n' &= \# \{j\;:\;(i,j)\in S\}.
\end{align*}
Similarly, for a matroid $\matrex = \left({[n]\choose 2},\calI \right)$ with graph symmetry, and $S\subseteq {[n]\choose 2}$, we define
$$\vsupp S := \{i\;:\;(i,j)\in S\}.$$
The \emph{signature} of $S$ is the number $\#\vsupp S$.
\end{Def}
Notice that for (bipartite) graph matroids, the signature is an invariant under the group action, therefore an invariant of the (bipartite) graph $S/G$.

\begin{Not}
Let $\matrex = ([m]\times [n],\calI)$  be a matroid with bipartite graph symmetry, let $S\subseteq [m]\times [n]$ with signature $(m,n)$. We will depict $S$ as
\begin{itemize}
\item a mask: the mask corresponding to $S$ is an $(m\times n)$-matrix with entries in $\{\circ,\bullet\}$, where $\circ$ in the $(i,j)$ position indicates that $(i,j)\not\in S,$ while $\bullet$ indicates $(i,j)\in S$.
\item a bipartite graph: the corresponding graph $S/G$. The unlabeled vertex set corresponding to $[m]$ will be depicted with red color, and the vertex set corresponding to $[n]$ will be depicted with blue color.
\end{itemize}

As an example, we depict the mask and the graph for the subset $S= \{ (1,1), (1,2), (1,3), (2,2), (3,3)\}$ in the matroid $\matrex = ([3]\times [3],\calI):$
$$ S,\; \mbox{as a mask:}  \hspace{1cm} \left( \begin{array}{ccc} \bullet & \bullet & \bullet \\ \circ & \bullet & \circ \\ \circ & \circ & \bullet \end{array} \right)  \hspace{1cm} S,\;\mbox{as a bipartite graph:}  \hspace{1cm} \parbox{\wd0}{\box0}{\begingroup
\setbox0=\hbox{\includegraphics[width=3cm]{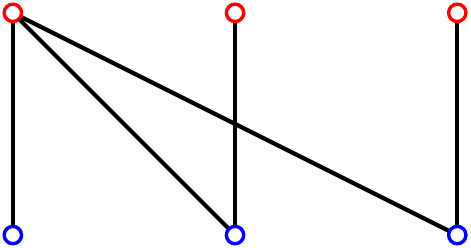}}%
\parbox{\wd0}{\box0}\endgroup} $$
The mask can be understood as adjacency matrices of the graph (with $1=\bullet$ and $0 = \circ$ in the usual convention), therefore we will use mask and graph interchangeably when it is clear that we are referring to a graph.

Similarly, for matroids with graph symmetry, symmetric masks and (ordinary) graphs will be used as depictions. In analogy, the vertex support is a subset of $[n]$, and the signature is its cardinality.
\end{Not}

For this, we introduce some notation:
\begin{Not}
Let $\matrex$ be a bipartite graph matroid (or a graph matroid) with symmetry group $G\cong \frakS(m)\times \frakS(n)$. Let $S\subseteq \calE(\matrex)$. Then, we will denote the group of $S$, which naturally acts on $\vsupp S$, and which is isomorphic to the stabilizer of $S$ in $G$, by
$$\frakS(S):=\{\sigma\in G\;:\; \sigma(S)=S\}.$$
We call $\frakS (S)$ the \emph{automorphism group} of $S$.
\end{Not}

\begin{Rem}
Note that $\frakS(S)$ does not depend on the choice of $m$ or $n$ nor on the matroid $\matrex$ - only on the orbit symmetry that $\matrex$ exhibits, and on the combinatorics of $S\subseteq \vsupp S$. Moreover, $\frakS (S)$ coincides with the usual graph theoretical definition of automorphism group of the graph $S/G$. Therefore, the isomorphy type of $\frakS (S)$ and in particular the cardinality $\# \frakS (S)$ are invariants of the graph $S/G$. Elementary group theory implies that it holds that $\# G = \# \frakS (S)\cdot \# (S/G)$, therefore if the signature of $S$ is $(\mu,\nu)$, it must hold that $\# \frakS(S)$ divides $\mu!\nu!$.
\end{Rem}

We would like to make a final remark that axiomatization of orbit matroids is possible without indirectly defining them through matroids:

\begin{Rem}\label{Rem:matrfam}
We would like to mention that graph matroids (and similarly the other orbit matroids) could be axiomatically defined, without resorting to quotients. For the case of graph matroids, the concept has been already investigated under the name of matroidal family of graphs, see~\cite{Schmidt197993} for a review. Namely, staying close to the axiomatization there, a graph matroid can be defined as a collection $\calC$ of (circuit) graphs, such that
\begin{itemize}
\item[(G0)] No $\calG\in \calC$ has isolated vertices.
\item[(G1)] If $\calG\in\calC$, and $\calG'\in \calC$ is a proper subgraph of $\calG$, then $\calG' = \calG$.
\item[(G2)] Let $\calG$ be a graph with subgraphs $\calG_1,\calG_2\in\calC$. Let $e\in\calG_1\cap \calG_2$ be an edge. Then, there is a subgraph $\calG'$ of $\calG\setminus \{e\}$ such that $\calG'\in \calC$.
\end{itemize}
This definition corresponds to the later definition of limit graph matroids in section~\ref{sec:size}. Finite graph matroids, as defined earlier, are obtained by replacing axioms G0 and G2:
\begin{itemize}
\item[(G0)] No $\calG\in \calC$ has isolated vertices, each $\calG\in\calC$ has at most $N$ vertices.
\item[(G2)] Let $\calG$ be a graph having at most $N$ vertices, with subgraphs $\calG_1,\calG_2\in\calC$. Let $e\in\calG_1\cap \calG_2$ be an edge. Then, there is a subgraph $\calG'$ of $\calG\setminus \{e\}$ such that $\calG'\in \calC$.
\end{itemize}
We want to note that the derivation of matroidal family of graphs, while it can be regarded equivalent, is sightly different from that of a graph matroid by the fact that we have constructed the latter as a quotient.
Furthermore, completely analogous definitions can be obtained for bipartite graph matroids or the other orbit matroids, replacing the word ``graph'' with the respective orbit structure. For the particular case of bipartite graph matroids, the infinite variant is completely analogous to the graph matroid, by replacing graphs by bipartite graphs, it corresponds to the two-sided limits from section~\ref{sec:size}. Two finite variants are obtainable by either restricting the number of vertices in one class, or in both classes, corresponding either to the one-sided limits of bipartite graph matroids from section~\ref{sec:size}, or to the bipartite matroids with finite ground set defined earlier.

A very interesting and still widely open question is which collections - and in particular, infinite collections - of circuit graphs give rise to graph matroids. However, since this is not a question we directly pursue in the present work, we do not make the alternate definition above. Examples for such families of (bipartite) graph matroids, however, and structure statements which can be interpreted as restrictions, can be found in section~\ref{sec:size} and later. The same question naturally arises for general orbit matroids. Also, the relation of the realization spaces of matroids with orbit symmetries and the orbit matroids would be interesting to investigate.
\end{Rem}

\subsection{Symmetrization and Bipartition}\label{sec:matroids.symm}
Let $G$ be a graph on $n$ vertices with adjacency matrix (mask) $M$.  If we interpret $M$ as the mask of a bipartite
graph  of signature $(n,n)$, we obtain a bipartite  graph $\tilde{G}$ (the ``Kronecker double cover'')
that is fixed under the involution $(i,j)\mapsto (j,i)$ that exchanges the parts of the bipartition
of $K_{n,n}$.  On the other hand, if $B$ is a bipartite graph of signature of $(m,n)$, it has a
natural embedding as a subgraph $\tilde{B}$ of $K_{m+n}$ that has a block-diagonal adjacency matrix.  In this
section, we explore how algebraic graph and bipartite graph matroids behave under these combinatorial lifting and
projection operations.

We start by considering how graph matroids induce bipartite graph matroids.
\begin{Def}\label{Def:matroidbipartiton}
Let $\matrex=(E,\calI)$ be a graph matroid on $\binom{[m+n]}{2}$. Write $S(m,n):=\{(i,j)\in [m]\times [n]\;:\; (i,j+m)\in S\}$ for all $S\in E$.
The \emph{$(m,n)$-bipartition} $\matrex|_{m,n}$ of $\matrex$
is the pair $(E',\calI')$ with $E'=[m]\times[n],$ and $\calI'=\left\{ S(m,n)\;:\; S\in \calI\right\}$.
\end{Def}

\begin{Prop}\label{Prop:bipmatr}
Let $\matrex$ be a graph matroid with ground set $\calE(\matrex)=\binom{[m+n]}{2}$. Then the $(m,n)$-bipartition $\matrex|_{m,n}$
is a bipartite graph matroid.  %
\end{Prop}
\begin{proof}
The bipartition construction is an instance of matroid restriction. Namely, $\matrex|_{m,n}$, is isomorphic to the restriction of $\matrex$ to the set $\{(i,j)\;:\;1\le i\le m, m+1\le j\le m+n\}\subseteq \calE(\matrex)$ and therefore a matroid, see~\cite[section~1.3]{Oxley}.
\end{proof}

Proposition~\ref{Prop:bipmatr} shows well-definedness of the bipartition, as a matroid, and implies therefore:

\begin{Cor}\label{Cor:bip}
Let $\matrex$ be a graph matroid with ground set $E={[m + n]\choose 2}$, let $S\subseteq E$. Let $\matrex|_{m,n}$ be the bipartition of $\matrex$, with ground set $[m]\times [n]$. Write $E':=[m]\times \{m+1,\dots m+n\}.$ Then, the following statements hold:
\begin{description}
\item[(i)] If $S$ is independent in $\matrex$, then $S(m,n)$ is independent in $\matrex|_{m,n}$.
\item[(ii)] If $S\subseteq E',$ then $S(m,n)$ is independent in $\matrex$ if and only if $S$ is independent in $\matrex|_{m,n}$.
\item[(iii)] If $S\subseteq E',$ then $S(m,n)$ is a circuit of $\matrex$ if and only if $S$ is a circuit of $\matrex|_{m,n}$.
\end{description}
\end{Cor}

While bipartition of a graph matroid can be described in a purely combinatorial manner, as above, it also entails bipartitions for the different kinds of realizations, given the matroid is both a graph matroid, and algebraic.

\begin{Prop}
\label{Prop:bip}
Let $\matrex$ be a graph matroid with $\calE(\matrex)={[m+n]\choose 2}$, with cryptomorphic realizations
$$\matrex =\matralg_K(\alphabf) = \matrbas(\calX,\bbf) = \matrcrd(\frakP,\xbf).$$
Denote $E|_{m,n}=\{(i,j)\;:\;1\le i\le m, m+1\le j\le m+n\}\subseteq \calE(\matrex).$ Then, the following are cryptomorphic realizations for $\matrex|_{m,n}$:
\begin{description}
\item[(a)] $(K, \alphabf')$ with $\alphabf' \alphabf\left(E|_{m,n}\right)$
\item[(b)] $(\calX',\bbf')$ with $\bbf' = \bbf\left(E|_{m,n}\right)$ and $\calX' = \pi_{\bbf'}(\calX).$
\item[(c)] $(\frakP',\xbf')$ with $\xbf' = \xbf\left(E|_{m,n}\right)$ and $\frakP' = K[\xbf']\cap \frakP.$
\end{description}
\end{Prop}
\begin{proof}
It holds that $\matrex|_{m,n}=\matralg_K(\alphabf')$ by substituting the two definitions. Consequently, by applying the definitions of basis and coordinate matroids, one obtains $\matralg_K(\alphabf')=\matrbas(\calX',\bbf')$ once it is established that $\calX'$ is irreducible, and $\frakP'$ is prime. Irreducibility of $\calX'$ follows from irreducibility of $\calX$, being the image of $\calX$ under the algebraic map $\pi_{\bbf'}$, and primeness óf $\frakP'$ follows from that of $\frakP$, being a coordinate section.
\end{proof}

Proposition \ref{Prop:bip} implies that the cryptomorphic realizations of $\matrex$ descend to
cryptomorphic realizations of $\matrex|_{m,n}$. %

Symmetrizing bipartite graph matroids, in order to obtain back a graph matroid, is a more delicate operation that de-symmetrizing graph matroids by bipartitioning them.  In particular, there is not a natural combinatorial construction that is guaranteed to produce a graph matroid (though it is not hard to produce a polymatroid). Thus, we will start directly with an algebraic construction.

\begin{Def}\label{Def:algsym}
Let $\matrex$ be an algebraic bipartite graph matroid on the ground set $E = [n]\times [n]$,
and let $\matrex=\matrbas(\calX,\bbf)$ be a basis realization of $\matrex$.  Let $H^\sym \subset K^E$ be the
linear space
\[
H^{\sym} := \{ \vec x\in K^E : \text{$x_{ij} = x_{ji}$ for all $(i,j)\in E$}\}
\]
and define $\calX^\sym := \calX\cap H^{\sym}$.  If $\calY$ is irreducible, we define the \emph{algebraic
symmetrization} of $\matrex$ to be the matroid $\matrex^\sym:=\matrbas_K\left(\calX^\sym,\bbf^\sym\right)$,
where $\bbf^\sym:=(\{(i,j)\; :\; i \le j\})$.
(If $\calX^\sym$ is not irreducible, the algebraic symmetrization is undefined.)
\end{Def}
\begin{Prop}\label{Prop:algsym}
Let $\matrex$ be an algebraic bipartite graph matroid of $E = [n]\times [n]$ with algebraic symmetrization $\matrex^\sym$ (assumed to exist).
Then $\matrex^\sym$ is a well-defined algebraic graph matroid on the ground set $\binom{[n]}{2}$.
\end{Prop}
\begin{proof}
By hypothesis, the variety $\calX^\sym$ is irreducible, so the matroid $\matrex^\sym$ will be well-defined once we
check that $H^\sigma$ is minimally spanned by $\bbf' := \bbf(\{(i,j)\; :\; i \le j\})$.  This also follows by
construction.

Now we check that $\matrex^\sym$ is a graph matroid. Consider the diagonal subgroup
$$G=\{(\sigma,\tau)\in\frakS(n)\times \frakS(n)\;:\;\sigma=\tau\}\subseteq \frakS(n)\times\frakS(n).$$
Note that $G$ stabilizes all sets of the form $\{(i,j),(j,i)\}$, and induces an action on $E=[n]\times [n]$. Since $H^\sigma$ is the set fixed by this same action on $\bbf$, we see that $\bbf'$, and then $\matrex^\sym$ inherits $\frakS(n)=\frakS(n)\times \frakS(n)/G$ as symmetry group, and ${[n]\choose 2}=[n]\times [n]/G$ as ground set.
\end{proof}

\begin{Rem}
The algebraic symmetrization is different than a construction introduced by Lovász \cite{L77} which
builds a matroid $\calN$ on the rank $2$ flats of a matroid $\matrex$ by picking a generic element
from each flat.
\end{Rem}

\begin{Rem}
The cryptomorphisms of Theorem~\ref{Thm:crypto} allow us to reformulate the algebraic symmetrization in
all three languages.  In terms of coordinate matroids, the irreducibility of $\calY$ corresponds to the
ideal $\frakP\cap \left\langle x_{ij} - x_{ji}\;:\; (i,j)\in[n]\times [n]\right\rangle$ being primary.
\end{Rem}
Whether the algebraic symmetrization exists is an interesting question. Since the variety $H^\sym$ is irreducible, the algebraic symmetrization will always exist for generic varieties $\calX$; that is, for an open dense set in the fixed-function Hilbert variety. However, this implies nothing for fixed varieties of interest, and algebraic, or structural criteria for existence of the symmetrization would be desirable.

\section{Determinantal and Rigidity Matroids}\label{sec:matroids.ex}
Definition \ref{Def:matrsymm} is motivated primarily by two examples with graph symmetry: determinantal matroids,
which control the algebraic aspects of low-rank matrix completion, and rigidity matroids from geometry. In this section, we will introduce these matroids in detail. We will assume that $K$ is algebraically closed.

\subsection{The Determinantal Matroid}\label{Sec:detmat}

The determinantal matroid, which we will denote by $\detM(m\times n,r)$, is the (purely combinatorial) matroid describing dependence relations of the entries of a rank $r$ matrix of rank $(m\times n)$. We give three realizations for the matroid, their equivalence follows from Proposition~\ref{Thm:crypto}; for notations, recall Definition~\ref{Def:matrvar}. We will assume that $K$ is a field of characteristic zero.

\begin{Def} \label{Def:detmatroid}
We define the \emph{determinantal matroid}, to be the matroid $\detM (m\times n,r)$ induced equivalently by any of the following realizations:
\begin{description}
\item[(a)] As an algebraic matroid: let $U_{ik}, 1\le i\le m, 1\le k\le r$ be a collection of doubly indexed transcendentals over $K$. Let $V_{jk}, 1\le j\le n, 1\le k \le r$ be another such collection. Let $\alphabf = \{X_{ij}= \sum_{k=1}^r U_{ik} V_{jk} \;:\; 1\le i\le m, 1\le j\le n\}$.
We define $\detM (m\times n,r) := \matralg_K (\alphabf)$ as the algebraic matroid of the $\alphabf$ over $K$.
\item[(b)] As a basis matroid: let $\detvar (m\times n,r)=\{A\in K^{m\times n}\;:\;\rank \, A\le r\}$ be the \emph{determinantal variety}. Let $\bbf =\{B^{(ij)}\;:\; 1\le i\le m, 1\le j\le n\}$ be the set of standard unit matrices $B^{(ij)}= e_i \cdot \tilde{e}_j^\top$, where $e_i,1\le i\le m$ is the standard orthonormal basis of $K^m$, and $\tilde{e}_j, 1\le j\le n$ is the standard orthonormal basis of $K^n$. We define $\detM (m\times n,r) := \matrbas (\detvar (m\times n,r), \bbf)$ as the basis matroid of $\detvar (m\times n,r)$.
\item[(c)] As a coordinate matroid: let $\xbf =\{X_{ij}\;:\;1\le i\le m, 1\le j\le n\}$ be a set of doubly indexed coordinates. Let $\detI(m\times n,r) \subset K[\xbf]$ be the \emph{determinantal ideal}, that is, the ideal in $K[\xbf]$ generated by the $(r+1) \times (r+1)$ minors of the $(m\times n)$ matrix whose entries are the coordinates $X_{ij}$. We define $\detM (m\times n,r) := \matrcrd (\detI(m\times n,r), \xbf)$ as the coordinate matroid of $\detI(m\times n,r)$.
\item[(d)] As a linear matroid: let $U_{ik},V_{jk},$ as in (a), and $L=K(U_{ik},V_{jk}).$  Let $\{u_{ik}, 1\le i\le m, 1\le k\le r\}\cup \{v_{ik}, 1\le j\le n, 1\le k\le r\}$ be any basis of $K^{r(m+n)}.$ Let $\vbf = \{v_{ij} = \sum_{k=1}^r U_{ik} v_{jk} + u_{ik} V_{jk} \;:\; 1\le i\le m, 1\le j\le n\}$ be a collection of vectors in $L^{r(m+n)}$. We define $\detM (m\times n,r) := \matrlin_L (\vbf)$ as the linear matroid of the $\vbf$.
\end{description}
In particular, the pairs $(K,\alphabf), (\detvar (m\times n,r),\bbf)$ and $(\detI(m\times n,r),\xbf)$ are cryptomorphic.
\end{Def}

\paragraph{Proof of well-definedness:} For well-definedness of (b), it needs to be checked that $\detvar (m\times n,r)$ is irreducible.
For (c), it needs to be checked that $\detI(m\times n,r)$ is prime. These are known facts: irreducibility of $\detvar (m\times n,r)$ can be found
in~\cite[Proposition~1.1~(a)]{Bruns}, or alternatively follows from Remark~\ref{Rem:detfacts} below; primeness of $\detI (m\times n,r)$ follows from~\cite[Theorem~2.10, Remark~2.12, and Corollary~5.17f]{Bruns}.

\paragraph{Proof of equivalence of the definitions:} (a) $\Leftrightarrow $ (b) follows from the observation that the $X_{ij}$ form a
rank $r$ matrix by definition. Therefore, the field $K(\xbf)$ is isomorphic to the fraction field of $\detvar (m\times n,r)$ in accordance
with the canonical isomorphism exposed in Proposition~\ref{Thm:crypto}. (b) $\Leftrightarrow $ (c):
This is a direct consequence of Proposition~\ref{Thm:crypto}. In addition, one has to note
that: $\detI(m\times n,r)$ is a prime ideal; $\detvar (m\times n,r) = \Van (\detI(m\times n,r))$, therefore it is
also irreducible. (a) $\Leftrightarrow $ (d): This follows from Proposition~\ref{Prop:cryptocharzero}, identifying $u_{ik}$ with $\diff U_{ik}$ and $v_{jk}$ with $\diff V_{jk},$ and noting that these generate the $r(m+n)$-dimensional $K$-vector space $\Omega_{L/K}$.

\begin{Rem}\label{Rem:detfacts} Some basic facts about the determinantal matroid:
\begin{description} %
\item[(i)] If $r\le \min(m,n)$, then $\dim \detvar (m\times n,r) = r \cdot (m+n-r)$, see~\cite[section 1.C, Proposition 1.1]{Bruns}. Therefore, by Proposition~\ref{Prop:rankcrypto} the rank of the determinantal matroid is as well $\rk \detM (m\times n,r) = r\cdot (m+n-r).$
\item[(ii)] $\detI(m\times n,r)$ is a toric ideal if and only if $r=1$ or $r\ge \min (m,n)$. %
\item[(iii)] There is a canonical surjective algebraic map
\begin{align*}
\Upsilon:& \CC^{m\times r}\times \CC^{n\times r}\longrightarrow \detvar (m\times n,r)\\
& (U,V)\mapsto U V^\top
\end{align*}
Therefore, $\detvar (m\times n,r)$ is irreducible.
\end{description}
\end{Rem}

\begin{Ex} \label{Ex:detm442}
We will explicitly describe the determinantal matroid $\detM(4\times 4,2)$.
\begin{enumerate} \compresslist
\item The prime ideal $\detI(4\times 4,2)$ is generated by the sixteen $(3\times 3)$-minors of a $(4 \times 4)$-matrix of formal variables.
\item The rank of $\detM(4\times 4,2)$ is $r(m+n-r) = 12$.
\item The circuit graphs are isomorphic to one of the following masks (with corresponding bipartite graphs, dotted lines indicating excluded edges):

\centerline{\matrixaMod \hspace{4mm} \includegraphics[scale=.7]{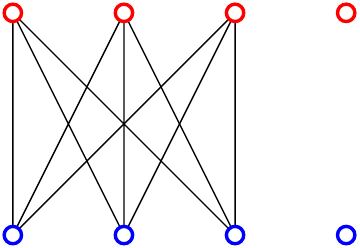}
\hspace{1cm} \matrixb \hspace{4mm}\includegraphics[scale=.7]{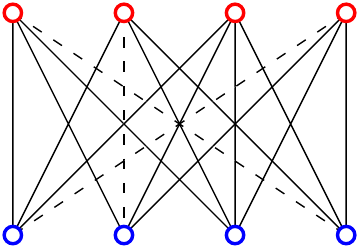}}
\item Bases of $\detM(4\times 4,2)$ are all sets of size $12$ not containing the complete bipartite graph $K_{3,3}$. There are nine non-isomorphic bipartite basis graphs (seven if we consider transpose symmetry).

\end{enumerate}
\end{Ex}

The use of colored squares to define the mask is motivated by an observation: all our circuit graphs can be constructed by sequential cancellation of minors. %

\subsection{The Symmetric Determinantal Matroid}
The symmetric determinantal matroid $\detMsym(n\times n,r)$ is the analogue of the determinantal matroid for symmetric $(n\times n)$ matrices. The determinantal matroid has an intrinsic $\ZZ/2$-symmetry which comes from the fact that rank is invariant under transposition of a matrix, therefore being a potential target of symmetrization. We will show that this symmetrization is well-defined.

\begin{Def} \label{Def:detmatroidsym}
We define the \emph{symmetric determinantal matroid}, to be the matroid $\detMsym (n\times n,r)$ induced equivalently by any of the following realizations:
\begin{description}
\item[($\oslash$)] As a symmetrization: define $\detMsym (n\times n,r)$ to be the symmetrization of $\detM (n\times n,r)$, in the sense of Definition~\ref{Def:algsym}.
\item[(a)] As an algebraic matroid: let $V_{ik}, 1\le i\le n, 1\le k\le r$ be a set of doubly indexed transcendentals over $K$. Let $\alphabf = \{X_{ij}= \sum_{k=1}^r V_{ik} V_{jk} \;:\; 1\le i\le j\le n\}$.
We define $\detMsym (n\times n,r) := \matralg_K (\alphabf)$ as the algebraic matroid of the $\alphabf$ over $K$.
\item[(b)] As a basis matroid: let $\detvarsym (n\times n,r)=\{A\in K^{n\times n}\;:\;\rank A\le r, A^\top = A\}$ be the \emph{symmetric determinantal variety}. Let $\bbf =\left\{\frac{1}{2}\left(B^{(ij)}+B^{(ji)}\right)\;:\; 1\le i\le j\le n\right\}$, where the $B^{(ij)}$ are standard unit matrices $B^{(ij)}= e_i \cdot e_j^\top$, where $e_i,1\le i\le n$ is the standard orthonormal basis of $K^n$. We define $\detMsym (n\times n,r) := \matrbas (\detvarsym (n\times n,r), \bbf)$ as the basis matroid of $\detvarsym (n\times n,r)$.
\item[(c)] As a coordinate matroid: let $\xbf =\{X_{ij}\;:\;1\le i\le j\le n\}$ be a set of doubly indexed coordinates. Let $\detIsym(n\times n,r) \subset K[\xbf]$ be the \emph{symmetric determinantal ideal}, that is, the ideal in $K[\xbf]$ generated by the $(r+1) \times (r+1)$ minors of the symmetric $(n\times n)$ matrix whose entries are the coordinates $X_{ij}$. We define $\detMsym (n\times n,r) := \matrcrd (\detIsym(n\times n,r), \xbf)$ as the coordinate matroid of $\detIsym(n\times n,r)$.
\item[(d)] As a linear matroid: let $V_{jk},$ as in (a), and $L=K(V_{jk}).$  Let $\{v_{ik}, 1\le j\le n, 1\le k\le r\}$ be any basis of $K^{rn}.$ Let $\vbf = \{v_{ij} = \sum_{k=1}^r V_{ik} v_{jk} + v_{ik} V_{jk} \;:\; 1\le i\le j\le n\}$ be a collection of vectors in $L^{rn}$. We define $\detMsym (n\times n,r) := \matrlin_K (\vbf)$ as the linear matroid of the $\vbf$.
\end{description}
In particular, the pairs $(K,\alphabf), (\detvarsym (n\times n,r),\bbf)$ and $(\detI(n\times n,r),\xbf)$ are cryptomorphic.
\end{Def}

{\bf Proof of well-definedness:} For (b), we need to check that $\detvarsym (n \times n,r)$ is irreducible; for (c), we must check that $\detI(m\times n,r)$ is prime.  Irreducibility of $\detvarsym (n\times n,r)$ can be demonstrated by the existence of the map in Remark \ref{Rem:detsymfacts}.
Primeness of $\detI (m\times n,r)$ follows from~\cite[Corollary to Theorem~1]{Kutz74}.

{\bf Proof of equivalence of the definitions:} Since $\detvarsym(n \times n,r)$ is irreducible, the symmetrization is by definition equivalent to the basis matroid of (b). The remaining equivalences follow from Proposition 2.2 and Proposition 2.13 as in the asymmetric case.

\begin{Rem}\label{Rem:detsymfacts} Some basic facts about the symmetric determinantal matroid:
\begin{description}
\item[(i)] If $r\le n$, then $\dim \detvarsym (n\times n,r) = r \cdot \left(n-\frac{1}{2}r(r-1)\right)$, see %
Therefore, by Proposition~\ref{Prop:rankcrypto} the rank of the symmetric determinantal matroid is as well $\rk \detMsym (n\times n,r) = r \cdot \left(n-\frac{1}{2}r(r-1)\right).$
\item[(ii)] There is a canonical surjective algebraic map
\begin{align*}
\Upsilon:& \CC^{n\times r}\longrightarrow \detvarsym (n\times n,r)\\
& V \mapsto V V^\top
\end{align*}
Therefore, $\detvarsym (n\times n,r)$ is irreducible.
\end{description}
\end{Rem}

\begin{Ex} \label{Ex:detmsym442}
We will explicitly describe the symmetric determinantal matroid $\detMsym(4\times 4,2)$.
\begin{enumerate} \compresslist
\item The prime ideal $\detIsym(4\times 4,2)$ is generated by the ten $(3\times 3)$-minors of a symmetric $(4 \times 4)$-matrix of formal variables.
\item The rank of $\detMsym(4\times 4,2)$ is $r \cdot \left(n-\frac{1}{2}r(r-1)\right) = 7$.
\item The circuit graphs are isomorphic to one of the following graphs labeled with top-degree, and depicted with a corresponding bipartite mask:
\begin{footnotesize}
\begin{table}[h]
\begin{center}
\begin{tabular}{|c|c|c|c|}
\hline
\includegraphics[scale=.6]{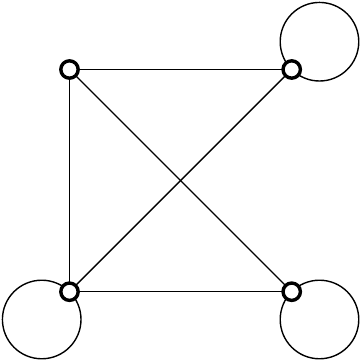} &
\includegraphics[scale=.6]{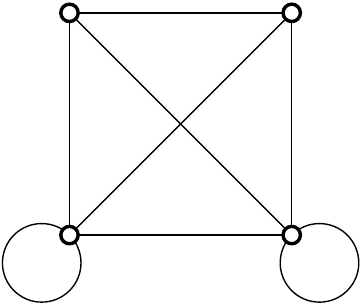} &
\includegraphics[scale=.6]{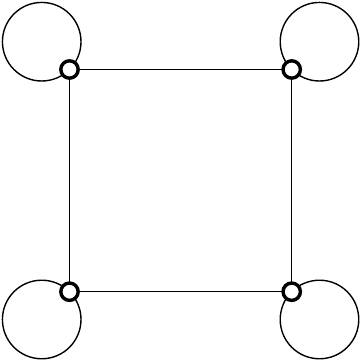} &
\includegraphics[scale=.6]{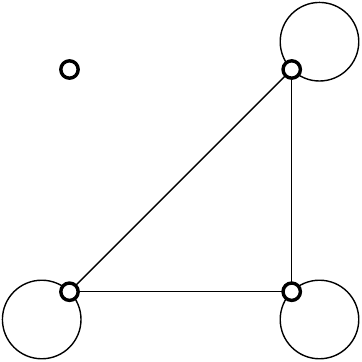} \\ \hline
\symmMatrixa & \symmMatrixb & \symmMatrixc & \symmMatrixd \\
\hline
\end{tabular}
\end{center}
\end{table}%
\end{footnotesize}
Again, we include colored squares to indicate the minors generating each circuit polynomial.
\item Bases of $\detM(4\times 4,2)$ are the sets of size $7$ not containing the rightmost circuit above. There are nine such basis graphs.
\end{enumerate}
\end{Ex}

\subsection{The Rigidity Matroid}
The other family of examples we consider are \emph{rigidity matroids}, which capture the
degrees of freedom in \emph{bar-joint frameworks}. These are, intuitively, mechanical structures
made of \emph{fixed-length bars}, connected by joints. %
It is assumed that the allowed continuous motions of the joints preserve the lengths and connectivity of the bars.
A framework is called (locally) \emph{rigid} if all such continuous motions extend to Euclidean motions of the ambient $d$-dimensional space.

In rigidity theory, the position of the joints are usually modelled by points $p_1,\dots,p_n\in \RR^r$, and the lengths of the bars are given by the pairwise Euclidean distances $\|p_i-p_j\|_2$. If the number of joints $n$ is large enough, not all of these distances can be chosen independently; the rigidity matroid models the dependence relations between these distances. In order to transform this scenario into an algebraic problem, we replace the Euclidean distances by the pairwise squared distances
$$d_{ij}=\|p_i-p_j\|_2^2=p_i^\top p_i - 2 p_i^\top p_j + p_j^\top p_j,$$
and consider this as an equation over an arbitrary field $K$ of characteristic zero. The corresponding algebraic matroid is then straightforwardly obtained as follows:

\begin{Def} \label{Def:detrigmatralg}
We define the \emph{rigidity matroid}, or \emph{Cayley-Menger matroid}, to be the matroid $\CMMsym (n\times n,r)$ induced by the following algebraic realization:\\

Let $P_{ik}\;: 1\le i\le n, 1\le k\le r$ be a set of doubly indexed transcendentals over $K$. Let $\alphabf = \{D_{ij}= \sum_{k=1}^r (P_{ik}-P_{jk})^2 \;:\; 1\le i\le j\le n\}$. We define $\CMMsym (n\times n,r) := \matralg_K (\alphabf)$ as the algebraic matroid of the $\alphabf$ over $K$.
\end{Def}

The equivalent realizations as basis and coordinate matroid whose existence is guaranteed by Theorem~\ref{Thm:crypto} is slightly less straightforward, and considerably less well studied than in the case of the determinantal matroids. Therefore, before stating the cryptomorphisms explicitly, we define the corresponding objects and collect some useful results on those.

The first relevant object is the variety for the basis representation:
\begin{Def}\label{Def:CMvarsym}
The \emph{Cayley-Menger variety} is the variety of
$$\CMvarsym (n\times n,r)=\{d\in K^{n\times n}\;:\; d_{ij} = p_i^\top p_i - 2 p_i^\top p_j + p_j^\top p_j\;\mbox{for some}\; p_i\in \CC^r, 1\le i\lneq j\le n\}.$$
\end{Def}

\begin{Rem}\label{Rem:CMsymfacts} Some basic facts about the Cayley-Menger variety:
\begin{description}
\item[(i)] The dimension of the Cayley-Menger variety is
$$\dim \CMvarsym (n\times n,r) = rn - {r+1\choose 2}.$$
This follows from~\cite[Theorem~4.3]{B02}. %
\item[(ii)] By definition of $\CMvarsym (n\times n,r)$, there is a canonical surjective algebraic map, the so-called \emph{length map},
\begin{align*}
\ell:& \CC^{n\times r}\longrightarrow \CMvarsym (n\times n,r)\\
& p \mapsto  d,\quad\mbox{where}\; d_{ij}= \sum_{k=1}^r (p_{ik}-p_{jk})^2.
\end{align*}
In particular, $\CMvarsym (n\times n,r)$ is an irreducible variety.
\end{description}
\end{Rem}

The second object is the ideal for the coordinate representation:
\begin{Def}\label{Def:CMIsym}
Let $\xbf =\{X_{ij},1\le i\lneq j\le n\}$ be a collection of formal coordinates. The \emph{Cayley-Menger matrix} associated to the $X_{ij}$ is the matrix
\[
\CM(n) :=
\begin{pmatrix}
0 & 1 & 1 & \cdots & 1 \\
1 & 0 & X_{12} & \cdots & X_{1n} \\
1 & X_{12} & 0 & \ddots & \vdots  \\
\vdots & \vdots & \ddots &  \ddots & X_{(n-1)n} \\
1 & X_{1n} & \cdots & X_{(n-1)n}& 0
\end{pmatrix}.
\]
The \emph{Cayley-Menger ideal} is the ideal $\CMIsym(n,r)\subset K[\xbf]$ generated by all
the $(r+3)\times (r+3)$-minors of the matrix $\CM (n)$.
\end{Def}

Since it apperas to be unknown what the ideal of the Cayley-Menger variety is, we briefly prove that it coincides with the Cayley-Menger ideal.

\begin{Thm}\label{Thm:CMIsym}
The Cayley-Menger ideal is the prime ideal corresponding to the Cayley-Menger variety, i.e.,
$$\CMIsym(n,r)=\Id(\CMvarsym (n\times n,r))$$
\end{Thm}
\begin{proof}
Results of Menger, which can be found as \cite[Proposition 6.2.11]{DL97} and \cite[Theorem 6.2.13]{DL97}, imply that
variables $X_{ij}$ are realizable as the set of pairwise distances between points in $\RR^r$
if and only if the matrix $\CM(n)$ has rank at most $r+2$, and a certain additional
semi-definiteness condition holds.  In particular, this shows that $\CMvarsym (n\times n,r)$
is the algebraic closure of the set of squared-distance matrices of $r$-dimensional point
configurations, so $\sqrt{\CMIsym(n,r)}$  is the vanishing ideal of $\CMvarsym (n\times n,r)$.

In \cite[Equation 6.2.9]{DL97}, the Cayley-Menger determinant is identified with the determinant of a symmetric matrix with linear entries; therefore, the Cayley-Menger ideal is a generic symmetric determinantal ideal under a change of coordinates. Therefore, by the same result in \cite[Corollary to Theorem 1]{Kutz74}, $\CMIsym(n,r)$ is prime, and therefore radical. By the Nullstellensatz, it is therefore identical to $\Id(\CMvarsym (n\times n,r))$.
\end{proof}

With these observations, Theorem \ref{Thm:crypto} and Proposition \ref{Prop:cryptocharzero} immediately yields:
\begin{Prop}\label{Prop:rigiditycrypto}
The following are equivalent realizations of $\CMMsym (n\times n,r)$: %
\begin{description}
\item[(a)] The algebraic matroid $\matralg_K (\alphabf)$, as defined in~\ref{Def:detrigmatralg}.
\item[(b)] The basis matroid $\matrbas (\CMvarsym (n\times n,r),\bbf )$ of the Cayley-Menger variety (Definition~\ref{Def:CMvarsym}), where $\bbf =\{B^{(ij)}\;:\; 1\le i\lneq j\le n\}$ is the set of standard unit matrices $B^{(ij)}= e_i \cdot \tilde{e}_j^\top$, and $e_i,1\le i\le n$ is the standard orthonormal basis of $K^n$
\item[(c)] The coordinate matroid $\matrcrd (\CMIsym (n\times n,r),\xbf),$ keeping the notation of Definition~\ref{Def:CMIsym}.
\item[(d)] The linear matroid associated to the differential of the length map $\ell$ given in Remark~\ref{Rem:CMsymfacts}: let $P_{jk},$ as in Definition~\ref{Def:detrigmatralg}, and $L=K(P_{jk}).$  Let $\{d_{ik}, 1\le j\le n, 1\le k\le r\}$ be any basis of $K^{rn}.$ Let $\vbf = \{v_{ij} = \sum_{k=1}^r (p_{ik}-p_{jk})(P_{ik}-P_{jk}) \;:\; 1\le i\le j\le n\}$ be a collection of vectors in $L^{rn}$. The linear matroid in question is $\matrlin_K (\vbf)$.
\end{description}
In particular, the pairs $(K,\alphabf), (\CMvarsym (n\times n,r),\bbf)$ and $(\CMIsym(n\times n,r),\xbf)$ are cryptomorphic.
\end{Prop}

Informally, Proposition \ref{Prop:rigiditycrypto} says that three definitions of the rigidity matroid are equivalent:
(a) via the intrinsic property of the measured lengths being generically realizable in dimension $d$;
(b) via the geometry of the Cayley-Menger variety; (c) via the extrinsic relations between the lengths; (d) in terms of the invertibility of a specific map, which can be analyzed via its differential (the ``rigidity matrix'' of combinatorial rigidity theory).

\begin{Ex} \label{Ex:cmmsym552}
We describe the rigidity matroid $\CMMsym(5\times 5, 2)$.
\begin{enumerate} \compresslist
\item The prime ideal $\CMIsym(5,2)$ is generated by the 36 $(5 \times 5)$-minors of the $(6 \times 6)$ Cayley-Menger matrix.
\item The rank of $\CMMsym(5\times 5,2)$ is $rn - \binom{r+1}{2} = 7$.

\item $\CMMsym(5\times 5,2)$ has exactly two circuit graphs: \\
\centerline{ \includegraphics[scale=.7]{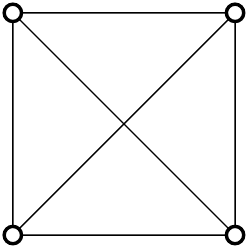}\hspace{1cm}
\includegraphics[scale=.7]{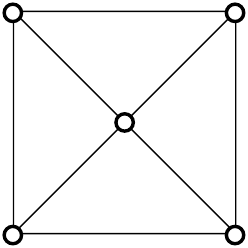}}

\item $\CMMsym(5\times 5,2)$ has three basis graphs: \\
\centerline{ \includegraphics[scale=.4]{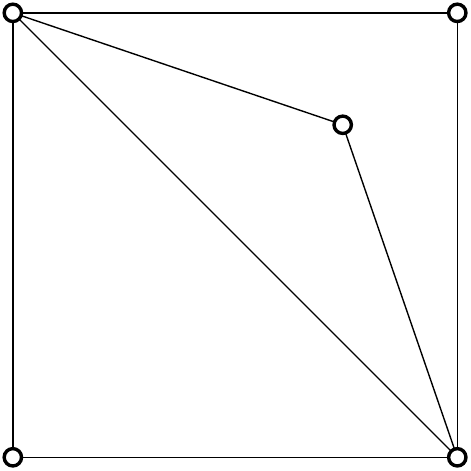}\hspace{1cm}
\includegraphics[scale=.4]{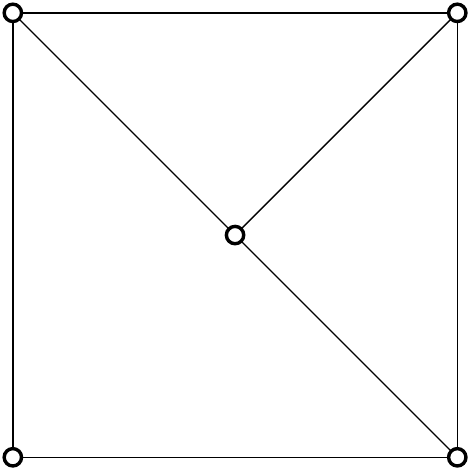}\hspace{1cm}
\includegraphics[scale=.5]{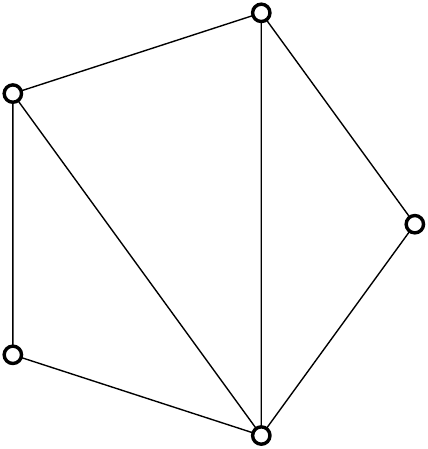}}
\end{enumerate}
\end{Ex}

\subsection{The Bipartite Rigidity Matroid}
The bipartite rigidity matroid is the non-symmetric version of the rigidity matroid; it can be interpreted as the matroid describing (squares of) distances $\|p_i-q_j\|_2$, where $p_1,\dots, p_m,q_1,\dots, q_n\in\RR^r$ are two distinct sets of joints. In order to be as concise as possible, and for convenience of the reader, we define the bipartite rigidity matroid as bipartition of the rigidity matroid.

\begin{Def}
The \emph{bipartite rigidity matroid}, or the \emph{bipartite Cayley-Menger matroid}, denoted by $\CMM (m\times n, r)$ is defined as the bipartition of the rigidity matroid $\CMMsym(m+n\times m+n, r)$.
\end{Def}

\begin{Def}\label{Def:CMvar}
The \emph{bipartite Cayley-Menger variety} is the variety of
$$\CMvar (m\times n,r)=\{d\in K^{m\times n}\;:\; d_{ij} = p_i^\top p_i - 2 p_i^\top q_j + q_j^\top q_j\;\mbox{for some}\; p_i,q_j\in \CC^r, 1\le i\le m, 1\le j\le n\}.$$
\end{Def}

The second object is the ideal for the coordinate representation:
\begin{Def}\label{Def:CMI}
Let $\xbf =\{X_{ij},1\le i\le m, 1\le j\le n\}$ be a collection of formal coordinates. The \emph{Cayley-Menger matrix} associated to the $X_{ij}$ is the matrix
\[
\CM(m\times n) :=
\begin{pmatrix}
0 & 1 & 1 & \cdots & 1 \\
1 & X_{11} & X_{12} & \cdots & X_{1n} \\
1 & X_{21} & X_{22} & \ddots & \vdots  \\
\vdots & \vdots & \ddots &  \ddots & X_{(n-1)n} \\
1 & X_{n1} & \cdots & X_{n(n-1)}& X_{nn}
\end{pmatrix}.
\]
The \emph{bipartite Cayley-Menger ideal} is the ideal $\CMI(m\times n,r)\subset K[\xbf]$ generated by all
the $(r+3)\times (r+3)$-minors of the matrix $\CM (m\times n)$.
\end{Def}

Again, we need to prove what the prime ideal for the bipartite Cayley-Menger variety is:
\begin{Thm}
The Cayley-Menger ideal is the prime ideal corresponding to the Cayley-Menger variety, i.e.,
$$\CMI (m\times n,r)=\Id(\CMvar (m\times n,r))$$
\end{Thm}
\begin{proof}
Since $\CMI (m\times n,r)$ is a section of the determinantal ideal $\detI(m\times n, r+2)$ it holds that $\Van (\CMI (m\times n,r)) = \CMvar (m\times n,r)$. By \cite[Corollary to Theorem 1]{Kutz74}, $\CMI(n,r)$ is prime, and therefore identical to $\Id(\CMvar (m\times n,r))$ by the Nullstellensatz.
{Thm:CMIsym}
\end{proof}

As a direct consequence of Proposition~\ref{Prop:bip}, we obtain the cryptomorphic realizations of $\CMM(m\times n,r)$:

\begin{Cor} \label{Cor:CMmatroid}
The bipartite rigidity matroid $\CMM (m\times n,r)$ is, as a matroid, isomorphic to the following realizations:
\begin{description}
\item[(a)] The algebraic matroid $\matralg_K (\alphabf),$ defined as follows: let $P_{ik}\;: 1\le i\le m, 1\le k\le r$ be a set of doubly indexed transcendentals over $K$. Let $Q_{jk}\;: 1\le j\le n, 1\le k \le r$ be another such set. Let $\alphabf = \{D_{ij}= \sum_{k=1}^r (P_{ik} - Q_{jk})^2 \;:\; 1\le i\le m, 1\le j\le n\}$.
\item[(b)] The basis matroid $\matrbas (\CMvar (n\times n,r),\bbf )$ of the Cayley-Menger variety (Definition~\ref{Def:CMvar}),  where $\bbf =\{B^{(ij)}\;:\; 1\le i\le m, 1\le j\le n\}$ is the set of standard unit matrices $B^{(ij)}= e_i \cdot \tilde{e}_j^\top$, with $e_i,1\le i\le m$ being the standard orthonormal basis of $K^m$, and $\tilde{e}_j, 1\le j\le n$ the standard orthonormal basis of $K^n$.
\item[(c)] The coordinate matroid $\matrcrd (\CMIsym (n\times n,r),\xbf),$ keeping the notation of Definition~\ref{Def:CMI}.
\item[(d)] The linear matroid $\matrlin_K (\vbf)$, defined as follows: let $P_{jk},Q_{ik}$ as in (A), and $L=K(P_{jk},Q_{ik}).$  Let $\{p_{ik}, 1\le j\le m, 1\le k\le r\}$ be any basis of $K^{rm},$ and $\{q_{ik}, 1\le j\le n, 1\le k\le r\}$ be any basis of $K^{rn}.$ Let $\vbf = \{v_{ij} = \sum_{k=1}^r (p_{ik}-q_{jk})(P_{ik}-Q_{jk}) \;:\; 1\le i\le m, 1\le  j\le n\}$ be a collection of vectors in $L^{r(m+n)}$.
\end{description}
In particular, the pairs $(K,\alphabf), (\CMvar (m\times n,r),\bbf)$ and $(\CMI(m\times n,r),\xbf)$ are cryptomorphic. Note that $\alphabf,\bbf, \xbf$ depend on $m,n,$ and $r$, while $\bbf, \xbf$ depend only on $m,n$.
\end{Cor}

\begin{Rem}\label{Rem:CMfacts} Some basic facts about the bipartite rigidity matroid:
\begin{description} %
\item[(i)] The dimension of the bipartite Cayley-Menger variety is
$$\dim \CMvar (m\times n,r) = r\cdot (m+n) - \frac{1}{2}r(r+1) + (m+n - k) - max(0, \frac{1}{2}r(r+3) + 1 - k),$$
where $k = m \cdot \mathbbm{1} (m \ge r+1) + n \cdot \mathbbm{1} (n \ge r+1)$, and $\mathbbm{1}(x)$ is the indicator function, being $1$ when $x$ is true and $0$ if $x$ is false. This follows from~\cite[Theorem 10]{BR80}.
\item[(ii)] In complete analogy to the case of the rigidity matroid, there is a canonical surjective algebraic map, the so-called \emph{length map},
\begin{align*}
\ell:& \CC^{m\times r}\times \CC^{n\times r}\longrightarrow \CMvar (m\times n,r)\\
& (p,q) \mapsto  d,\quad\mbox{where}\; d_{ij}= \sum_{k=1}^r (p_{ik}-q_{jk})^2.
\end{align*}
In particular, $\CMvar (m\times n,r)$ is an irreducible variety.
\end{description}
\end{Rem}

\begin{Ex} \label{Ex:cmm442}
We describe the bipartite rigidity matroid $\CMM(4\times 4,2)$.
\begin{enumerate} \compresslist
\item The prime ideal $\CMI(4\times 4,2)$ is generated by the seventy $5\times 5$-minors of the $9 \times 9$ Cayley-Menger matrix whose only diagonal entry is $\CM(8)_{1,1}$.
\item Because $m = n = 4 \geq 3$, the rank of $\CMM(4\times 4,2)$ reduces to
\[r(m+n) - \binom{r+1}{2} - \max(0,\frac{1}{2}r(r+3) + 1 - m - n)) = 13.\]
\item A subset of $[4]\times [4]$ of size $13$ is a basis of  $\CMM(4\times 4,2)$ if and only if the corresponding basis graph does not contain the complete bipartite graph $K_{3,4}$. There are three such bipartite graphs, represented by the following masks: \begin{footnotesize}
\[ \left( \begin{array}{cccc} \bullet & \bullet & \bullet & \bullet \\
\bullet & \bullet & \circ & \bullet \\
\bullet & \circ & \bullet & \bullet \\
\circ & \bullet & \bullet & \bullet \end{array} \right) \hspace{1cm}
\left( \begin{array}{cccc} \bullet & \bullet & \bullet & \bullet \\
\bullet & \bullet & \bullet & \bullet \\
\bullet & \bullet & \circ & \bullet \\
\circ & \circ & \bullet & \bullet \end{array} \right)\hspace{1cm}
\left( \begin{array}{cccc} \bullet & \bullet & \bullet & \bullet \\
\bullet & \circ & \bullet & \bullet \\
\circ & \bullet & \bullet & \bullet \\
\circ & \bullet & \bullet & \bullet \end{array} \right) \]
\end{footnotesize}

\item The circuit graphs are isomorphic to one of the following masks: \\
\begin{footnotesize}
\centerline{\brmatrixone \hspace{1cm} \brmatrixtwo \hspace{1cm} \brmatrixa}
\end{footnotesize}
\end{enumerate}
\end{Ex}

\section{Limit Matroids and their Rank Function}\label{sec:size}
\subsection{Direct Limits of Matroids}\label{sec:size.directlims}
One basic observation is that a circuit graph $C$ of signature $(k,\ell)$ of $\detM (m\times n,r)$ is also a circuit graph
of $\detM (m'\times n',r)$, provided that $m'\ge k$ and $n'\ge \ell$.  In this section, we show how to construct
limits of graph matroid sequences that allows us to speak of a limiting determinantal matroid $\detM(\NN\times \NN,r)$, or, in general, of limits of matroids with intrinsic symmetries.

We start with a general construction of direct limit for matroids.
\begin{Def}\label{Def:matrinjseq}
A collection of matroids $\matrex_\nu=(E_\nu,\calI_\nu),\nu\in \NN$, such that $E_\nu\subseteq E_{\nu+1}$, and $\calI_\nu=\{S\cap E_\nu\;:\; S\subseteq \calI_{\nu+1}\}$ for all $n\in \NN$, is called an \emph{injective sequence} of matroids, and will be denoted by $(\matrex_\nu)$. If furthermore the matroids $\matrex_\nu$ come with groups of automorphisms $G_\nu$, we call $(\matrex_\nu)$ \emph{compatible with the $G_\nu$-action} if $G_\nu\subseteq G_{\nu+1}.$
\end{Def}

We will now introduce what a direct limit of such an injective sequence should mean, but we will not develop the theory of direct limit matroids in detail. Instead, we will just state by definition what such a direct limit should be.

\begin{Def}\label{Def:matrlimit}
Let $(\matrex_\nu),\nu\in\NN$ be an injective sequence of matroids. We define the \emph{direct limit matroid} to be the the pair
$$\matrex:=\varinjlim_{\nu\in \NN} \matrex_\nu := \left(\calE(\matrex),\calI(\matrex)\right),$$
for which we \emph{define}\\

\begin{tabular}{lr}
the ground set &$\calE(\matrex):=\bigcup_{\nu\in \NN}\calE(\matrex_\nu)$\\
the set of circuits &$\calC(\matrex):=\bigcup_{\nu\in \NN}\calC(\matrex_\nu)$\\
the set of dependent sets &$\calD(\matrex):=\{S\subseteq \calE(\matrex)\;:\;\exists C\in\calC(\matrex)\;\mbox{s.t.}\;C\subseteq S\}$\\
the set of independent sets &$\calI(\matrex):=\calP(\calE(\matrex))\setminus \calD(\matrex)$
\end{tabular}
\end{Def}
Note that we have refrained from defining a set of bases for $\matrex$; this is possible in different ways, but we avoid it since we will not need to make use of bases of any kind for $\matrex$. We also want to note that, while it is possible to define different kinds of direct limits for matroids, in this paper, the direct limits will \emph{always} come from injective sequences. Furthermore, by the recent results of \cite{BDetal13}, the direct limit matroids defined as above lie in the class of ``finitary'' type infinite matroids. Before proceeding, we explicitly describe some functorial properties of the direct limit inherited from the injective sequence:

\begin{Rem}\label{Rem:limpropts}
Let $(\matrex_\nu)$ with $E_\nu=\calE(\matrex_\nu)$ be an injective sequence of matroids, let $\matrex = \varinjlim_{\nu\in \NN} \matrex_\nu$ its direct limit, having ground set $E$. By definiton, and elementary matroid properties, it holds that:
\begin{itemize}
\item[(i)] $S\subseteq E$ is a circuit in $\matrex$ if and only if $S\cap E_\nu$ is a circuit in $\matrex_\nu$ for all $\nu\in\NN$ such that $S\subseteq E_\nu$
\item[(ii)] $S\subseteq E$ is dependent in $\matrex$ if and only if $S\cap E_\nu$ is dependent in $\matrex_\nu$
for some $\nu\in \NN$.
\item[(iii)] $S\subseteq E$ is independent in $\matrex$ if and only if $S\cap E_\nu$ is independent in $\matrex_\nu$ for all $\nu\in\NN$.
\end{itemize}
\end{Rem}

By the injectivity of the limit sequence, the rank function extends as well to the direct limit matroid:
\begin{Lem}\label{Lem:ranklim}
Let $\matrex=\varinjlim_{\nu\in \NN} \matrex_\nu$ be a direct limit matroid, with ground sets $E=\calE(\matrex),E_\nu=\calE(\matrex_\nu)$. Then there is a
unique function $\rk_\matrex : \calP^* (E)\to \NN$ such that
$$	\rk_\matrex (S)  = \rk_{\matrex_\nu}(S)$$
for all $S\subseteq E_\nu$ and $\nu\in \NN$, where $\cal^*(E)$ is the set
of finite subsets of $E$.
\end{Lem}
\begin{proof}
This follows directly from Remark~\ref{Rem:limpropts} and any characterization of the rank
function in terms of dependent/independent sets, see e.g.~\cite[Theorem~1.3.2]{Oxley}.
\end{proof}
By the injectivity of the limit sequence, the group action extends as well:
\begin{Lem}\label{Lem:grouplim}
Let $\matrex=\varinjlim_{\nu\in \NN} \matrex_\nu$ be a direct limit matroid, compatible with a sequence of groups of automorphisms $G_\nu$. Let $G=\varinjlim_{\nu\in \NN} G_\nu$ be the corresponding direct limit of groups. Then, $G$ is a group of automorphisms of $\matrex$ in the sense that $G$ acts on $\calE(\matrex)$, and
$$\sigma \calI(\matrex) = \calI(\matrex),\; \sigma \calD(\matrex) = \calD(\matrex),\;\sigma \calC(\matrex) = \calC(\matrex)$$
for all $\sigma\in G$.
\end{Lem}
\begin{proof}
This follows directly from Remark~\ref{Rem:limpropts}, the fact that for each $\sigma\in G$, there is $n$ such that $\sigma\in G_\nu$, and the properties of automorphisms.
\end{proof}
The notion of Lemmata~\ref{Lem:ranklim} and~\ref{Lem:grouplim} allows to talk about orbits in $\matrex_\nu/G_\nu$ and $\matrex/G$.
Finally, we would like to remark that the limit construction in Definitions~\ref{Def:matrinjseq} and~\ref{Def:matrlimit} can be iterated without problems a countable number of times. That is, one can take $\matrex_\nu$ in both definitions to be direct limit matroids, obtaining a new limit matroid.

\subsection{One-sided Bipartite Graph Matroid Limits}
\label{sec:size.onesided}
For graph and bipartite graph matroids, we will want to impose some additional structure.  For now, we
will concentrate on bipartite graph matroids, deferring the case of graph matroids to
Section \ref{sec:size.graphs}.

\begin{Def}\label{Def:graphseq}
Let $(\matrex_\nu)$ be an injective sequence of bipartite graph matroids, such that $\matrex_\nu$ has ground set $\calE (\matrex_\nu)=[m]\times [\nu]$; the sequence is compatible with the graph symmetry. The sequence $(\matrex_\nu)$ is called a \emph{one-sided bipartite graph matroid sequence},
and the limit $\matrex=\varinjlim_{\nu\in \NN} \matrex_\nu$ a \emph{bipartite graph matroid limit}, or the limit of $(\matrex_\nu)$.
\end{Def}
Note that, as defined above, $\matrex$ has ground set $[m]\times \NN$. When dealing with one-sided limits, we will reserve and fix $m$, the number of rows in the ground set.

\begin{Rem}
The ground sets $E_\nu$ can be interpreted as the edges of complete bipartite graphs $K_{m,\nu}$ of a one-sided bipartite graph
matroid sequence $(\matrex_\nu/G_\nu).$ The sequence adds one vertex at a time, always in the same part of the
bipartition (i.e., with the same color) which we call the \emph{growing part}. In the limit, one obtains the graph $K_{m,\infty}$ in the graph matroid $\matrex/G$. From now on, we will switch freely between the formalisms.
\end{Rem}

\begin{Lem}\label{Lem:graphmatroidlimit}
Let $(\matrex_\nu),\nu\in\NN$ be a one-sided bipartite graph matroid sequence with limit matroid $\matrex$. Consider the automorphism groups $G_{\nu}=\frakS(m)\times \frakS(\nu)$ acting on $\matrex_{\nu}$. Then, the group $G=\frakS(m)\times \frakS(\NN)$ acts canonically on $\matrex$ and its ground set $[m]\times \NN$. Moreover, for a graph $\calG$, it holds:
\begin{itemize}
\item[(i)] $\calG$ is a circuit graph in $\matrex/G$ if and only if $\calG$ is a circuit graph in $\matrex_\nu/G_\nu$ for all $\nu\in\NN$ such that $G\subseteq K_{m,\nu}$.
\item[(ii)] $\calG$ is a dependent graph in $\matrex/G$ if and only if $\calG$ is a dependent graph in $\matrex_\nu/G_\nu$ for all $\nu\in\NN$ such that $G\subseteq K_{m,\nu}$.
\item[(iii)] $\calG$ is an independent graph in $\matrex/G$ if and only if $\calG\cap K_{m,\nu}$ is an independent graph in $\matrex_\nu/G_\nu$ for all $\nu\in\NN$, where $\calG\cap K_{m,\nu}$ is any edge intersection under an arbitrary (disjoint) identification of vertices with same color.
\item[(iv)] For a graph $\calG$ in $\matrex/G$, it holds that $\rk_\matrex(\calG)$ = $\rk_{\matrex_\nu}(\calG)$ for all $\nu\in\NN$ such that $G\subseteq K_{m,\nu}$.
\end{itemize}
\end{Lem}
\begin{proof}
The first part follows from the fact that the permutation group $\frakS(\NN)$ is the injective limit of the finite permutation groups $\frakS(\nu)$, therefore $\frakS(m)\times \frakS(\NN)$ is the injective limit of the $G_\nu=\frakS(m)\times \frakS(\nu),$ and the action on the ground sets $E_\nu=[m]\times[\nu]$ of $\matrex_\nu$ is compatible with the group inclusion $G_\nu\subseteq G_{\nu+1}$. The statements on the graphs $\calG$ are then direct consequences of Lemma~\ref{Lem:grouplim}.
\end{proof}
Lemma~\ref{Lem:graphmatroidlimit} allows us to consider circuit, dependent graphs,
independent graphs, and ranks of graphs in terms of only the limit of a one-sided sequence
of bipartite graph matroids.

To illustrate the concepts in this section, we will use two running examples.
\begin{Def}\label{Def:limitmatr}
Let $m\in \NN$ and $r\in \NN$ be fixed. The injective sequence of determinantal matroids $\detM(m\times \nu,r)$
forms a one-sided bipartite graph matroid sequence in the running parameter $\nu$. We denote its limit as
$$\detM(m\times \NN,r) := \varinjlim_{\nu\in \NN} \detM(m\times \nu,r).$$
Similarly, the injective sequence of bipartite rigidity matroids $\CMM(m\times \nu,r)$
forms a one-sided bipartite graph matroid sequence. We denote its limit as
$$\CMM(m\times \NN,r) := \varinjlim_{\nu\in \NN} \CMM(m\times \nu,r).$$
\end{Def}

The realizations, and their cryptomorphisms in Theorem~\ref{Thm:crypto} naturally
extend to cryptomorphisms of the direct limits.
\begin{Prop}\label{Prop:limitcrypto}
Let $(\matrex_\nu),\nu\in\NN$ be a one-sided bipartite graph matroid sequence with limit matroid $\matrex$, such that $\matrex_\nu$ is algebraic, realized by cryptomorphic pairs (a) $(K,\alphabf_\nu)$, (b) $(\calX_\nu, \bbf_{\nu})$, (c) $(\calI_\nu, \xbf_{\nu})$, as described in Definition~\ref{Def:detmatroid}. Then, the following pairs are cryptomorphic:
\begin{description}
\item[(a)] $(\alphabf, K)$, where $\alphabf=\bigcup_{\nu=1}^\infty \alphabf_{\nu}$
\item[(b)] $(\calX,\bbf)$, where $\calX$ is the inverse limit of the projection morphisms $\calX_\nu\rightarrow \calX_\nu'$ for $\nu'\le \nu$ which project to the the span of the $\bbf_\nu'$, and $\bbf =\bigcup_{\nu=1}^\infty \bbf_{\nu}$
\item[(c)] $(\calI,\xbf)$, where $\xbf=\bigcup_{\nu=1}^\infty \xbf_{\nu},$ and $\calI = \left\langle \bigcup_{n=1}^\infty \calI_\nu\right\rangle \subseteq K[\xbf],$ with $K[\xbf]$ being the polynomial ring in the infinitely many variables $\xbf$
\end{description}
That is, the matroids\footnote{as in Definition~\ref{Def:matrvar} by taking $n$ as countable infinity} $\matralg_K (\alphabf), \matrbas (\calX, \bbf)$ and $\matrcrd (\calI, \xbf)$ are canonically isomorphic, and furthermore, they are isomorphic to the matroid $\matrex$ as given in Definition~\ref{Def:limitmatr}.
\end{Prop}
\begin{proof}
The statement follows from noting that taking $n$ as countable infinity in Definition~\ref{Def:matrvar} does not affect well-definedness of the respective matroids, nor the correctness of the proof of Theorem~\ref{Thm:crypto}.
\end{proof}

Note that the variety $\calX$ in Proposition~\ref{Prop:limitcrypto}~(b) is an irreducible sub-variety of the span of $\bbf$, as an inverse limit of irreducible varieties; furthermore, the ideal $\calI$ in Proposition~\ref{Prop:limitcrypto}~(c) is prime, as a direct limit of prime ideals.

\subsection{Relative Rank, Average Rank and Realization Size}\label{sec:size.growth}
We introduce some tools for analyzing the asymptotic growth of the rank along
a directed system of matroids.
\begin{Def}
Let $\matrex$ be a matroid, with ground set $E$ and rank function $\rk$.
The \emph{relative rank function} is the function
\begin{align*}
\rk(\cdot|\cdot) :& E\times E \rightarrow \NN\\
& (T,S)\mapsto \rk (T\cup S) - \rk(S),
\end{align*}
The number $\rk (T | S)$ is called the relative rank of $T$, given $S$.
\end{Def}
\begin{Prop}\label{Prop:relrk}
For any matroid $\matrex$ on a finite ground set $E$, the relative rank function satisfies:
\begin{description}
\item[(i)] $\# T \ge \rk(T|S)\ge 0$ for all $T,S\subseteq E$
\item[(ii)] $\rk(A|S) \ge \rk(A|T)$ for any $A\subseteq E$ and $S\subseteq T\subseteq E$.
\item[(iii)] $\rk(A|S\cup B) \le \rk(A|S)$ for any $A,B,S\subseteq E$
\item[(iv)] $\rk(\sigma T|\sigma S) = \rk(T|S)$ for any $S,T\subseteq E$, and any automorphism $\sigma$ of $\matrex$
\item[(v)] The set $S\cup T$ contains a basis with $\rk(T|S)$ elements from $T$ for all $T,S\subseteq E$.
\end{description}
\end{Prop}
\begin{proof}
(iv) follows from the fact that $\rk(\sigma S)= \rk (S)$ for any $S$, then substituting the definition of the relative rank.

(v) follows from the fact that any independent set can be extended to a basis, and that the rank of a set is the size of a basis.

All the other statements follow from submodularity of the rank function, see e.g.~\cite[Corollary1.3.4~(R3)]{Oxley}.
Namely, an equivalent characterization of $\rk$ being submodular is the statement
$$
\rk(S\cup A) - \rk(S) \ge \rk(T\cup A) - \rk(T)\;\mbox{for all}\; S\subset T\subset E, A\subset E
$$
The inequality (ii) then follows from inserting the definition of the relative rank; and (i) follows from taking once $T=\varnothing$, once $S=\varnothing$; (iii) from taking $T=S\cup B$ in (ii).
\end{proof}

\begin{Def}\label{Def:drk}
Let $\matrex$ be a one-sided bipartite graph matroid limit, on the ground set $\calE(\matrex)=[m]\times \NN$. We define the \emph{growth function} of the sequence $(\matrex_\nu)$ as
$$\drk_{\matrex}(\cdot) : \NN \rightarrow \NN,\quad \nu\mapsto \rk \left(K_{m,\nu+1} | K_{m,\nu}\right).$$
If clear from the context, we will just write $\drk(\cdot)$, omitting the dependence on $\matrex$.
\end{Def}
Note that due to Lemma~\ref{Lem:ranklim}, the relative rank can be taken in any $\matrex_k$ for $k\ge \nu+1$.
For a one-sided bipartite graph matroid sequence, the growth function is monotone:
\begin{Prop}\label{Prop:drk}
Let $\matrex$ be a one-sided bipartite graph matroid limit. Then, the growth function $\drk(\cdot)$ of $\matrex$ is non-negative and non-increasing. That is, $0\le \drk(\nu+1)\le \drk(\nu)$ for all $\nu\in\NN$.
\end{Prop}
\begin{proof}
Non-negativity follows at once from Proposition~\ref{Prop:relrk}~(i). For the increasing
property, write $\Delta_i = E_{i+1}\setminus E_i$.  Then, for $\nu\in \NN$, we have
\begin{eqnarray*}
\drk(\nu) = & \rk(\Delta_\nu|E_\nu) & \\
= & \rk(\Delta_{\nu + 1}|E_\nu) & \text{by Proposition \ref{Prop:relrk} (iv)} \\
\ge & \rk(\Delta_{\nu + 1}|E_{\nu + 1}) & \text{by Proposition \ref{Prop:relrk} (ii)} \\
= & \drk(\nu + 1)
\end{eqnarray*}
\end{proof}
Note that up to here, the $\frakS (m)$-symmetry on the first component of $E$ has not been used, therefore Proposition~\ref{Prop:drk} also holds in the somewhat more general setting where the symmetry group $G_\nu$ of $\matrex_\nu$ is only $\frakS(\nu)$.

\begin{Prop}\label{Prop:addany}
Let $\matrex$ be an one-sided bipartite graph matroid limit, with ground set $[m]\times \NN$. Denote $E_\nu:=[m]\times [\nu]$, and $\Delta_\nu := E_{\nu+1}\setminus E_{\nu}$, that is, $\Delta_\nu=[m]\times \{\nu+1\}.$
Then, for all $\nu\in \NN$ with $S\subset E_\nu$ and $T\subset \Delta_\nu$ such that $\#T\le \drk(\nu)$, it holds that
$$\rk(T|S) = \# T$$
\end{Prop}
\begin{proof}
Assume, for the moment, that $\#T = \drk(\nu)$. By Proposition \ref{Prop:relrk} (v),
there is some $T'\subset \Delta_\nu$ such that $\rk(T'|E_\nu) = \drk(\nu) = \#T$.
We then have the sequence of relations
\begin{eqnarray*}
\#T \ge & \rk(T|S) & \text{by Proposition~\ref{Prop:relrk}~(i)} \\
\ge & \rk(T|E_\nu) & \text{by Proposition \ref{Prop:relrk} (iii)} \\
= & \rk(T'|E_\nu) & \text{by row symmetry and Proposition \ref{Prop:relrk} (iv)} \\
= & \#T
\end{eqnarray*}
which imply that the inequalities are all tight.  The case in which $\#T < \drk(\nu)$ follows
from monotonicity and normalization of matroid rank functions.
\end{proof}
\begin{Rem}\label{Rem:addany}
Proposition~\ref{Prop:addany} is sharp in the sense that for each $\nu$, in the sense that
there exist sets $S$ such that
$\rk(S) < \rk(S\cap E_{\nu}) + \#(S\cap \Delta_\nu)$ if $\# (S\cap \Delta_\nu) > \drk(\nu+1)$.
Any $S$ with the property that $\# (S\cap E_{\nu})$ is a basis of $\matrex_{\nu}$ has this property, from the
definition of the growth function, but this is not the only type of obstruction.
\end{Rem}
Proposition \ref{Prop:addany} implies that any one-sided bipartite graph matroid limit
has a special basis from which we can read off properties.  This will be useful in the next
section, where we analyze two-sided matroid limits.
\begin{Def}\label{Def:onesided-staircase}
Let $\matrex$ be an one-sided bipartite graph matroid limit, with ground set $[m]\times \NN$.  The
\emph{staircase basis} of $\matrex$ is defined as the union of the (column)
rectangles $[\nu]\times [\drk(\nu)]$ for all $\nu\in \NN$.

We define the numbers $d_1,d_2,\ldots, d_t$ to be the \emph{jump sequence} of indices $\nu$ such that
$\drk(\nu) \gneq \drk(\nu + 1)$.  A \emph{corner} of a staircase basis is an index $(\mu,\nu)$ such that
$(\mu-1,\nu)$ and $(\mu,\nu - 1)$ are both in the staircase basis, but $(\mu,\nu)$ is not; in
particular, all the corners have column indices of the form $d_i + 1$.
\end{Def}
From Proposition~\ref{Prop:drk}, the growth function $\drk$ is lower bounded and non-increasing.
Therefore, its limit exists:
\begin{Def}\label{Def:avgrk}
Let $\matrex$ be a one-sided bipartite graph matroid limit. The \emph{average rank} $\rho(\matrex)$ of
$\matrex$ is defined to be $\lim_{\nu\to\infty} \drk(\nu)$.
\end{Def}
\begin{Rem}
The terminology average rank can be interpreted in several ways for a one-sided bipartite limit graph matroid $\matrex=\varinjlim_{\nu\in \NN} \matrex_\nu$, if we denote the average vertex degree of a graph $\calG$ by $\rho (\calG)$:
\begin{description}
\item[(i)] It holds that $\ark(\matrex)=\lim_{\nu\to\infty}\frac{1}{\nu}\rk \matrex_\nu$.
\item[(ii)]  Choose arbitrary basis graphs $\calG_\nu$ in $\matrex_\nu$. Then $\ark (\matrex)= \lim_{\nu\to\infty} \rho \left(\calG_\nu\right).$
\item[(iii)] It holds that $\ark (\matrex) = \sup \{\rho(\calG)\;:\;\calG\in\calI(\matrex)\}.$
\end{description}
\end{Rem}
Proposition \ref{Prop:addany} has a useful specialization in terms of the average rank.  %
\begin{Lem}\label{Lem:addany}
Let $\matrex$ be an one-sided bipartite graph matroid limit with average rank $\rho$,
and let $\calG$ be any finite bipartite graph with a vertex $v$ of degree $d\le \rho$.  Then
any edge incident on $v$ is a bridge in $\calG$.
\end{Lem}
\begin{proof}
The graph $\calG$ has some representative set $S\subset E_\nu$, with $E_\nu = [m]\times [\nu]$, and
without loss of generality, we may assume that the vertex $v$ is mapped to the column $\nu$.  Then
by Proposition \ref{Prop:drk}, $\drk(\nu)\ge \rho\ge d$.  If $\rk(G) = \rk(G\setminus \{vj\})$
for some edge $vj$ incident on $v$, then we have a contradiction to Proposition
\ref{Prop:addany}.
\end{proof}
The average rank also controls the asymptotic behavior of $\rk(\calM_\nu)$ in a one-sided bipartite graph matroid
sequence.
\begin{Prop}\label{Prop:largerank}
Let $\matrex$ be a one-sided bipartite graph matroid limit on the ground set $\calE(\matrex)=[m]\times \NN$, with average rank $\ark$. Then, there are unique numbers $\crk,\rsize\in\NN$, such that
\begin{align*}
\rk(K_{m,\nu})&=(\nu-\rsize)\cdot \ark + \crk\quad\mbox{for all}\;\nu\ge\rsize\\
\rk(K_{m,\nu})&\gneq(\nu-\rsize)\cdot \ark + \crk\quad\mbox{for all}\;\nu\lneq\rsize.
\end{align*}
Moreover, $\crk = \rk(K_{m,\rsize})$.
\end{Prop}
\begin{proof}
Write $E_\nu:=[m]\times [\nu]$. By definition, $\rk(K_{m,\nu+1})-\rk(K_{m,\nu})=\rk(E_{\nu+1}|E_\nu)=\drk(\nu)$, therefore, $\rk(K_{m,\nu})=\sum_{n\lneq \nu}\drk(n)$. Furthermore, by Proposition~\ref{Prop:drk}, $\drk(\cdot)$ is non-negative non-increasing, so there is a unique number $\rsize$ such that $\drk(n)= \ark$ for all $n\gneq\rsize$ and $\drk(n)\gneq \ark$ for all $n\le \rsize$. Summing up these inequalities for $\drk(n)$ for all $n\lneq\nu$, one obtains the claimed inequalities. Moreover, comparing coefficients yields $\crk = \rk(K_{m,\rsize}).$
\end{proof}
\begin{Def}\label{Def:relsize}
For a one-sided bipartite graph matroid limit $\matrex$ as in~\ref{Prop:largerank},
the number $\rsize(\matrex):=\rsize$ is
called \emph{realization size}, and the number $\crk(\matrex):=\crk$ the \emph{realization rank}.
\end{Def}

\subsection{Circuits and Bases of One-Sided Bipartite Graph Limits}\label{sec:size.obglstruct}
We will now characterize one-sided bipartite graph limits and their matroid structure through average rank and realization size.

We start with the simple situation when all matroids in question are free.
\begin{Prop}\label{Prop:arkfreem}
Let $(\matrex_\nu)$ be a one-sided bipartite graph matroid sequence with limit $\matrex$, such that $\matrex$ has ground set $[m]\times \NN$. Then, the following are equivalent:
\begin{itemize}
\item[(i)] For all $\nu$, the matroid $\matrex_\nu$ is free.
\item[(ii)] $\rk(K_{m,\nu})=m\nu$.
\item[(iii)] $\ark = m$.
\end{itemize}
In all cases, $\rsize= \crk = 0$.
\end{Prop}
\begin{proof}
(i)$\Leftrightarrow$(ii) follows from the definition of a free matroid, (ii)$\Rightarrow$(iii) from Proposition~\ref{Prop:largerank}. (iii)$\Rightarrow$(ii) is obtained from an inductive application of Proposition~\ref{Prop:addany} to $S=K_{m,\nu}$ and $T$ being one vertex of degree $m$ added to the growing side.
\end{proof}
It can happen that $\matrex$ is not free even if $\rsize=\crk=0$, e.g. if all elements in $\matrex_\nu$
are loops.  The realization size and average rank are linked by the existence of a special circuit.
\begin{Prop}\label{Prop:elcirc}
Let $\matrex$ be a one-sided bipartite graph matroid limit on the ground set $[m]\times\NN$, with average rank
$\ark=\ark(\matrex) < m$ and realization size $\rsize(\matrex) = \rsize$.  Then $K_{\ark+1,\rsize+1}$ is a circuit
graph in $\matrex$.
\end{Prop}
\begin{proof}
Proposition \ref{Prop:largerank} that for large enough $\ell$, $(\rho + 1)\ell > \rk(\calM_\ell)$, and so
there is some finite $\ell\in \NN$ for which $K_{\rho+1,\ell}$
is dependent.  Lemma \ref{Lem:addany} implies that any circuit $C$ in $K_{\rho+1,\ell}$ must be a complete
bipartite graph, since vertices in the growing part of degree at most $\rho$ are incident only on bridges.
For $\nu>\kappa+1$, Proposition \ref{Prop:largerank} implies that
$\rk(K_{\rho+1,\nu}) < \nu(\rho + 1) - 1$. The only possibility, then,
for the circuit $C$ is $K_{\rho + 1,\kappa+1}$.
\end{proof}
Proposition \ref{Prop:elcirc} implies that this next definition is sensible.
\begin{Def}\label{Def:elcirc}
Let $\matrex$ be a one-sided bipartite graph matroid limit on the ground set $[m]\times\NN$,
with average rank $\ark=\ark(\matrex) < m$ and realization size $\rsize(\matrex) = \rsize$.
The complete graph $K_{\rho + 1,\kappa+1}$ is defined to be the \emph{elementary circuit} of $\calM$.
\end{Def}

The existence of the elementary circuit implies several combinatorial characterizations of average rank and realization size:
\begin{Thm}\label{Thm:avgrk}
Let $\matrex$ be a one-sided bipartite graph matroid limit on the ground set $[m]\times\NN$, with average rank
$\ark=\ark(\matrex) < m$. Then, the following are equal:
\begin{itemize}
\item[(i)] The average rank $\ark$.
\item[(ii)] The largest $k\in \NN$ such that $K_{k,\ell}$ is independent for all $\ell\in \NN$.
\item[(iii)] The largest $k\in \NN$ such that if $\calG$ is an independent graph of signature $(i,j)$, then
adding a vertex of degree $k$ to $\calG$ always results in an independent graph of signature $(i,j+1)$.
\item[(iv)] The largest $k\in \NN$ such that: if $C$ is a circuit, and $(c,d)\in C$, then $\#\{i\;:\;(i,d)\in C\}\gneq k$.
\item[(v)] The smallest $k\in \NN$ such that $K_{k+1,\ell}$ is a circuit for some $\ell\in \NN$.
\item[(vi)] The smallest $k\in\NN$ such that there is a circuit of signature $(k+1,\ell)$ for some
$\ell\in \NN$.
\end{itemize}
\end{Thm}
\begin{proof}
We prove each statement in turn.

\textbf{(ii):} By Lemma \ref{Lem:addany}, for $k\le \rho$ every edge in $K_{k,\ell}$ is a bridge,
and hence $K_{k,\ell}$ is independent.  On the other hand, the existence of the elementary circuit
implies that $K_{\rho+1,\kappa+1}$ is dependent.

\textbf{(iii):}By Lemma \ref{Lem:addany}, for $k\le \rho$ the edges incident on vertices of degree at
most $k$ are bridges, and thus never in any circuit.  For $k > \rho$, Remark \ref{Rem:addany} implies that the
conclusion of (iii) is false.

\textbf{(iv):}  In the elementary circuit, every vertex in the growing part has degree $\rho+1$, and so there
exist circuits with minimum degree $\rho+1$ in the growing part.  On the other hand, a circuit with
a vertex of degree $\rho$ in the growing part contradicts Lemma \ref{Lem:addany}.

\textbf{(v)--(vi):}  These follows from (ii) and the existence of the elementary circuit.
\end{proof}

\begin{Thm}\label{Thm:relsize}
Let $\matrex$ be a one-sided bipartite graph matroid limit, with realization size $\rsize=\rsize(\matrex)\gneq 0$.
Then, the following are the same:
\begin{itemize}
\item[(i)] The realization size $\rsize$.
\item[(ii)] The smallest $\ell\in\NN$ such that $K_{\ark+1,\ell+1}$ is dependent.
\item[(iii)] The biggest $\ell\in\NN$ such that $K_{\ark+1,\ell}$ is independent.
\item[(iv)] The unique $\ell\in\NN$ such that $K_{\ark+1,\ell}$ is a circuit.
\end{itemize}
\end{Thm}
\begin{proof}
The equality of (i)--(iv) is immediate from the elementary circuit being minimally dependent.
\end{proof}
As a direct consequence of Theorems~\ref{Thm:avgrk}, we can also derive the following result bounding the
size of all circuits in a bipartite graph limit:
\begin{Thm}\label{Thm:circuitsize}
Let $\matrex$ be a one-sided bipartite graph limit, with average rank $\ark$, realization size $\rsize$ and realization rank $\crk$. Let $C$ be a circuit of $\matrex$, with signature $(k,\ell)$. Then,
$$\ell \le \crk - \ark\cdot \rsize + 1.$$
\end{Thm}
\begin{proof}
The statement is certainly true if $\matrex$ is a free matroid, since then a circuit $C$ does not exist. Therefore, by Proposition~\ref{Prop:arkfreem}, we may suppose that $\rsize\gneq 0$
By Theorem~\ref{Thm:avgrk}~(iv), it holds that
$$\# C\ge \ell\cdot (\ark + 1).$$
On the other hand, for any $e\in C$, the set $C\setminus \{e\}$ is independent, therefore by Proposition~\ref{Prop:largerank}, it holds that
$$\# C -1 \le \rk (K_{m,\ell}) = \crk + (\ell - \rsize)\cdot \ark.$$
Combining the inequalities, one obtains
$$\ell\cdot (\ark + 1) \le \# C \le \crk + (\ell - \rsize)\cdot \ark +1,$$
an elementary computation then yields the claimed inequality.
\end{proof}

As an important corollary, we obtain a finiteness result for the number of circuit graphs:
\begin{Cor}\label{Thm:circuitfinite}
Let $\matrex$ be a one-sided bipartite graph limit. Then, there are only a finite number of circuit graphs in $\matrex$.
\end{Cor}
\begin{proof}
Let $[m]\times \NN$ be the ground set of $\matrex$, let $\ark=\ark(\matrex),\rsize=\rsize(\matrex),\crk=\crk(\matrex)$. By Theorem~\ref{Thm:circuitsize}, any circuit graph $\calG$ must be contained in $K_{m,\ell},$ where $\ell =\crk - \ark\cdot \rsize + 1.$ Since there is only a finite number of graphs contained in $K_{m,\ell}$ there is only a finite number of circuit graphs of $\matrex$.
\end{proof}
We now return to our running examples.
\paragraph{Determinantal matroids}
We can interpret structural results on determinantal matroids from \cite{KTTU12}  in terms
of the parameters in this section.

\begin{Prop}\label{Prop:detM-ark}
Consider the infinite determinantal matroid $\detM(m\times \NN,r)$, as obtained in Definition~\ref{Def:limitmatr}, with $r\lneq m$, average rank $\ark$, realization size $\rsize$ and realization rank $\crk$. Then,
$$\ark = r,\quad \rsize  = r,\quad \crk = r\cdot m.$$
\end{Prop}
\begin{proof}
The dimension formula in Remark~\ref{Rem:detfacts}~(i) implies that
$$\rk (\detM(m\times n,r)) = r\cdot (m+n-r)$$
(recalling the usual assumption that $r\le m$). The equalities then follow from the uniquness implied by Proposition~\ref{Prop:largerank}, and comparing coefficients.
\end{proof}

\begin{Prop}\label{Prop:circuitsizedet}
Let $C$ be a circuit of $\detM (m \times \NN, r)$, with signature $(k,\ell)$. Then,
$$\ell\le r(k - r) + 1.$$
\end{Prop}
\begin{proof}
Note that if $C$ is a circuit of $\detM (m \times \NN, r)$, it is a circuit of $\detM (k \times \NN, r)$ as well, by Lemma~\ref{Lem:graphmatroidlimit}.
The statement then follows from substituting the numbers computed in Proposition~\ref{Prop:detM-ark} into Theorem~\ref{Thm:circuitsize}.
\end{proof}

\paragraph{Bipartite rigidity matroids} The behavior of the bipartite rigidity matroids is slightly more complicated.
\begin{Prop}\label{Prop:CMM-ark}
Consider the infinite bipartite rigidity matroid $\CMM(m\times \NN,r)$, as obtained in Definition~\ref{Def:limitmatr}, with $r\lneq m$, average rank $\ark$, realization size $\rsize$ and realization rank $\crk$. Then,
$$\ark = r,\quad \rsize  = \binom{r+1}{2},\quad \crk = r\cdot m + 3{r+1\choose 3}.$$
\end{Prop}
\begin{proof}
The dimension formula in Remark~\ref{Rem:CMfacts}~(i) implies the statement, together with uniquness implied by Proposition~\ref{Prop:largerank}, and comparing coefficients.
\end{proof}

\begin{Rem}
The geometric intuition behind the value of $\rsize$ is that for $k < \rsize(\CMM(m\times \NN,r))$,
the dimension formula in Remark~\ref{Rem:CMfacts}~(i) implies that the complete graph $K_{r + 1,k}$ is generically flexible.
If, instead of measuring distances
from $r$ generic points on a rigid body, we measure
distances to $r$ joints of a generic, flexible framework, the new point
is no longer determined up to a discrete set of choices.
\end{Rem}

\begin{Prop}\label{Prop:circuitsizeCM}
Let $C$ be a circuit of $\CMM (m \times \NN, r)$, with signature $(k,\ell)$. Then,
$$\ell \le r\cdot k - \binom{r+1}{2} +1.$$
\end{Prop}
\begin{proof}
Note that if $C$ is a circuit of $\CMM (m \times \NN, r)$, it is a circuit of $\CMM (k \times \NN, r)$ as well, by Lemma~\ref{Lem:graphmatroidlimit}.
The statement then follows from substituting the numbers computed in Proposition~\ref{Prop:CMM-ark} into Theorem~\ref{Thm:circuitsize}.
\end{proof}

\subsection{Two-sided Bipartite Graph Matroid Limits}
Until now, we have considered matrix-shaped ground sets which are allowed to grow, or to be infinite, in one direction. Now we will consider both direction growing, or being infinite, at the same time. For this, we directly introduce a corresponding sequence of bipartite matroids:
\begin{Def}\label{Def:injcompl}
Let $(\matrex_{\mu\nu}),$ with $\mu,\nu\in \NN,$ and $\matrex_{\mu\nu}=(E_{\mu\nu},\calI_{\mu\nu})$ be a collection of matroids. $(\matrex_{\mu\nu})$ is called \emph{two-sided injective complex} if $E_{\mu\nu}\subseteq E_{\mu'\nu'}$ and $\calI_{\mu\nu}=\{S\cap E_{\mu\nu}\;:\; S\subseteq \calI_{\mu'\nu'}\}$ for all $\nu\le \nu',\mu\le \mu'$. If a group $G_{\mu\nu}$ is a group of automorphisms of $(\matrex_{\mu\nu})$, the injective complex is called \emph{compatible} with the $G_{\mu\nu}$-action if $G_{\mu\nu}\subseteq G_{\mu'\nu'}$ for all $\nu\le \nu',\mu\le \mu'$.
\end{Def}
\begin{Def}\label{Def:twosidedseq}
Let $(\matrex_{\mu\nu})$ be an two-sided injective complex of matroids, such that $\matrex_{\mu\nu}$ has ground set $E_{\mu\nu}=[\mu]\times [\nu].$
The limit
$$\matrex=\varinjlim_{\mu\in \NN}\varinjlim_{\nu\in \NN} \matrex_{\mu\nu}$$
is called a \emph{two-sided matroid limit}, or the limit of $(\matrex_{\mu\nu})$.
If the $\matrex_{\mu\nu}$ are furthermore bipartite graph limits with ground set $\calE(\matrex_{\mu\nu})=[\mu]\times [\nu],$ such that the complex is compatible with the graph symmetry, the sequence $(\matrex_{\mu\nu})$ is called a\emph{two-sided bipartite graph matroid complex}, and $\matrex$ is called a \emph{two-sided bipartite graph matroid limit}.
\end{Def}
Note that the limits in Definition~\ref{Def:twosidedseq} exist, as iterating the direct limit construction in Definition~\ref{Def:matrlimit} twice is unproblematic, as it has also been remarked at the end of
section~\ref{sec:size.directlims}. We also introduce notation for passing back to the one-sided limits,
obtained by fixing one of the two indices:
\begin{Def}\label{Def:slice}
Let $(\matrex_{\mu\nu})$ be a two-sided injective complex of matroids, with limit $\matrex$. We will denote
$$\matrex_{\mu*}:=\varinjlim_{\nu\in \NN} \matrex_{\mu\nu}\quad\mbox{and}\quad \matrex_{*\nu}:=\varinjlim_{\mu\in \NN} \matrex_{\mu\nu},$$
noting that for any fixed $\mu\in\NN$, or any fixed $\nu\in\NN$, the $\matrex_{\mu\nu}$ form a one-sided injective sequence of matroids. We will call the $\matrex_{\mu*}$ and $\matrex_{*\nu}$, informally,  \emph{slices} of $\matrex$.
\end{Def}

Therefore, we directly obtain an analogue of Lemma~\ref{Lem:graphmatroidlimit} for the two-sided case:
\begin{Lem}\label{Lem:graphmatroidlimitmn}
Let $(\matrex_{\mu\nu}),\mu,\nu\in\NN$ be a two-sided bipartite graph matroid complex with limit matroid $\matrex$, consider the automorphism groups $G_{\mu\nu}=\frakS(\mu)\times \frakS(\nu)$ acting on $\matrex_{\mu\nu}$. Then, the group $G=\frakS(\NN)\times \frakS(\NN)$ acts canonically on $\matrex$ and its ground set $\NN\times\NN$. Moreover, for a graph $\calG$, it holds:
\begin{itemize}
\item[(i)] $\calG$ is a circuit graph in $\matrex/G$ if and only if $\calG$ is a circuit graph in $\matrex_{\mu\nu}/G_{\mu\nu}$ for all $\mu,\nu\in\NN$ such that $G\subseteq K_{\mu,\nu}$.
\item[(ii)] $\calG$ is a dependent graph in $\matrex/G$ if and only if $\calG$ is a dependent graph in $\matrex_{\mu\nu}/G_{\mu\nu}$ for all $\mu,\nu\in\NN$ such that $G\subseteq K_{\mu\nu}$.
\item[(iii)] $\calG$ is an independent graph in $\matrex/G$ if and only if $\calG\cap K_{\mu\nu}$ is an independent graph in $\matrex_{\mu\nu}/G_{\mu\nu}$ for all $\mu,\nu\in\NN$, where $\calG\cap K_{\mu\nu}$ is any edge intersection under an arbitrary (disjoint) identification of vertices with same color.
\item[(iv)] For a graph $\calG$ in $\matrex/G$, it holds that $\rk_\matrex(\calG)$ = $\rk_{\matrex_{\mu\nu}}(\calG)$ for all $\mu,\nu\in\NN$ such that $G\subseteq K_{\mu,\nu}$.
\end{itemize}
\end{Lem}
\begin{proof}
Iterate the proof of Lemma~\ref{Lem:graphmatroidlimit}; once for the limit in $\nu$, then for the limit in $\mu$.
\end{proof}

Therefore, the notion of circuits, dependent graphs, independent graphs, and ranks of graphs, are again independent of the matroid in which they are considered.

We return to our two running examples.
\begin{Def}\label{Def:limitmatrmn}
Let $r\in \NN$ be fixed. The injective complex of determinantal matroids $\detM(\mu\times \nu,r)$
forms a two-sided bipartite graph matroid sequence in the running parameters $\mu,\nu$.  We denote its limit
$$\detM(\NN\times \NN,r) := \varinjlim_{\mu\in \NN}\varinjlim_{\nu\in \NN} \detM(\mu\times \nu,r)=\varinjlim_{\mu\in \NN} \detM(\mu\times \NN,r).$$
Similarly, the injective complex of bipartite rigidity matroids $\CMM(\mu\times \nu,r)$
forms a two-sided bipartite graph matroid sequence.  We denote its limit by
$$\CMM(\NN\times \NN,r) := \varinjlim_{\mu\in \NN}\varinjlim_{\nu\in \NN} \CMM(\mu\times \nu,r)=\varinjlim_{\mu\in \NN} \CMM(\mu\times \NN,r).$$
\end{Def}

\subsection{Average Rank, and Realization Size for Two-Sided Limits}
The invariants from the one-sided case have analogues in the two-sided case, which
we now develop.

\begin{Def}\label{Def:drkmn}
Let $(\matrex_{\mu\nu}),\mu,\nu\in \NN$ be an injective complex of matroids, with $E_{\mu,\nu}=\calE(\matrex_{\mu\nu}).$ We define the \emph{growth function} of the sequence $(\matrex_\nu)$ as
$$\drk_{(\matrex_{\mu\nu})}(\cdot,\cdot) : \NN\times \NN \rightarrow \NN\times \NN,\quad (\mu,\nu)\mapsto \left(\rk (E_{\mu+1,\nu} | E_{\mu,\nu}),\rk (E_{\mu,\nu+1} | E_{\mu,\nu})\right).$$
If clear from the context, we will just write $\drk(\cdot,\cdot)$, omitting the dependence on the complex $(\matrex_{\mu,\nu})$ for readability. In this case, we will write $\drk(\mu,\nu)=:\left(\drk_1(\mu,\nu),\drk_2(\mu,\nu)\right)$ for the components. If $\matrex$ is the limit of $(\matrex_{\mu,\nu})$, we will talk about $\drk$ being the growth function of $\matrex$.
\end{Def}
The average ranks (and realization sizes, which we will define later) of a two-sided graph matroid
limit will be fundamentally linked to those of its slices. Therefore, we introduce some notation for this:
\begin{Not}
Let $(\matrex_{\mu\nu}),\mu,\nu\in \NN$ be a two-sided injective complex of matroids. We will denote
\begin{align*}
\ark^{(2)}_m(\matrex):=\ark (\matrex_{m,*}),\quad \ark^{(1)}_n(\matrex):=\ark (\matrex_{*,n}),\\
\rsize^{(2)}_m(\matrex):=\rsize (\matrex_{m,*}),\quad \rsize^{(1)}_n(\matrex):=\rsize (\matrex_{*,n})
\end{align*}
for average rank and realization size of the slices.
\end{Not}
As in the one-sided case, we will need to distinguish between limits of free matroids and
non-trivial sequences.
\begin{Def}\label{Def:drkmnfree}
Let $\matrex$ be a two-sided bipartite graph matroid limit. We call an index $(\mu,\nu)\in\NN^2$ such that
$K_{\mu,\nu}$ is a dependent graph in $\matrex$ \emph{realizing}. If there exists a realizing index for
$\matrex$, we call $\matrex$ realizing.
\end{Def}
\begin{Rem}
Let $\matrex$ be a two-sided bipartite graph matroid. Then, the following are equivalent:
\begin{itemize}
\item[(i)] $\matrex$ is not realizing.
\item[(ii)] $\matrex$ is a limit of free matroids.
\end{itemize}
\end{Rem}

\begin{Rem}\label{Rem:riqs}
Let $\matrex$ be a two-sided bipartite graph matroid. Then, since supersets of dependent sets are dependent, the following statement is true:
If $N\in\NN^2$ is a realizing index for $\matrex$, then $N'\in\NN^2$ is a realizing index for all $N'\ge N$; here, as usual, inequality is taken component-wise.
\end{Rem}

\begin{Prop}\label{Prop:drkmn}
Let $\matrex$ be a two-sided bipartite graph matroid limit with growth function $\drk(.,.)$. Then:
\begin{itemize}
\item[(i)] For all $\mu,\nu\in\NN^2,$ it holds that $\drk(\mu,\nu)\ge (0,0)$
\item[(ii)] For all $\mu,\nu\in\NN^2,$ the following inequalities hold:
$$
\drk_1(\mu+1,\nu)\le \drk_1(\mu,\nu),\quad \drk_2(\mu,\nu+1)\le \drk_2(\mu,\nu).
$$
\item[(iii)] For all realizing indices $N\in\NN^2$, it holds that $\drk(N')\le\drk(N)$ for all $N'\ge N$.
\item[(iv)] If $\matrex$ is realizing, the limit $\ark(\matrex)=\lim_{\nu\to\infty} \drk(\nu,\nu)$ exists.
\end{itemize}
By usual convention, inequalities between vectors/elements of $\NN^2$ above are component-wise.
\end{Prop}
\begin{proof}
We prove each statement in turn:

\textbf{(i):} This is directly implied by Proposition~\ref{Prop:relrk}~(i).

\textbf{(ii):} The statement follows Proposition~\ref{Prop:drk}, and from the fact that for fixed $\mu,\nu\in\NN,$ it holds that $\drk(\mu,\nu)=(\drk_1(\mu),\drk_2(\nu))$, where $\drk_1$ is the growth function of the injective sequence $(\matrex_{*\nu})$, and $\drk_2$ is the growth function of the injective sequence $(\matrex_{\mu*}).$\\

\textbf{(iii):} It suffices to prove: for realizing $N=(\mu,\nu)\in\NN^2$, it holds that
$$
\drk_2(\mu+1,\nu)\le \drk_2(\mu,\nu),\quad \drk_1(\mu,\nu+1)\le \drk_1(\mu,\nu).
$$
The statement then follows from (ii) and double induction on $\mu,\nu$. We prove the inequality for $\drk_2$. Write $\lambda:=\drk_2(\mu,\nu)$, $E_{k,\ell}=[k]\times [\ell]$ for all $k,\ell\in\NN$, and $\Delta = [\lambda]\times\{\nu+1\}$. Since $N$ is realizing, if holds that $\lambda\lneq \nu$. Therefore, $\rk(E_{i,\nu+1}|E_{\mu,\nu}\cup \Delta)=0$ for all integers $1\le i\le \mu$. By Proposition~\ref{Prop:relrk}~(iii), this implies that $\rk(E_{i,\nu+1}|E_{\mu+1,\nu}\cup \Delta)=0$ for all $1\le i\le \mu$. By using that the rank is invariant under the $\frakS(\mu+1)$-action on rows, this implies that $\rk(E_{i,\nu+1}|E_{\mu+1,\nu}\cup \Delta)=0$ for all $1\le i\le \mu+1$. Therefore, $\drk_2(\mu+1,\nu)\le \lambda = \drk_2 (\mu,\nu).$ The inequality for $\drk_1$ follows in complete analogy, obtained by exchanging the indices in the above.

\textbf{(iv):} Since $\matrex$ is realizing, there is a realizing index $N\in\NN^2$. By Remark~\ref{Rem:riqs}, there is $n \in\NN$ such that $(n,n)$ is realizing. By (i) and (iii), it follows that $0\le \drk(\nu+1,\nu+1)\le \drk(\nu+1,\nu)\le \drk(\nu,\nu)$, for all $\nu\ge n$, which implies the statement.
\end{proof}

Similarly as in the one-sided case, Proposition~\ref{Prop:drkmn} makes the average rank of a two-sided sequence well-defined:
\begin{Def}\label{Def:avgrkmn}
Let $\matrex$ be a realizing two-sided bipartite graph matroid limit. The \emph{average rank} $\ark(\matrex)$ of $\matrex$ is defined as the limit $\lim_{\nu\to\infty} \drk(\nu,\nu)$ from Proposition~\ref{Prop:drkmn}~(vi). We will usually denote the components of the average rank by $\ark (\matrex) =: (\ark_1(\matrex),\ark_2(\matrex))$.
\end{Def}

If $\matrex$ were invariant under exchanging indices, the realization size could be defined, as for the one-sided case, as the smallest $\nu$ for which $\drk (\nu,\nu)$ achieves its limit. However, as the behavior for the first index can be different from the second one, we make a definition which is slightly more complicated:

\begin{Def}\label{Def:relsizemn}
The \emph{realization size} $\rsize (\matrex)$ of $\matrex$ is defined as the unique number $(k,\ell)\in\NN^2$ such that:
Under the assumption that $\mu\ge k$, whether $(\mu,\nu)$ is realizing is independent of $\mu$, and $k$ is the smallest such integer. Under the assumption that $\nu\ge \ell$, whether $(\mu,\nu)$ is realizing is independent of $\nu$, and $\ell$ is the smallest such integer.
\end{Def}

This definition is justified by the following result which states that average rank and realization size are compatible with taking slices:

\begin{Prop}\label{Prop:onevstwo}
Let $\matrex$ be a two-sided injective complex of matroids, having average rank $(\ark_1,\ark_2)$, and realization size $(\rsize_1,\rsize_2)$. Assume that $\matrex$ is realizing. Then:
\begin{itemize}
\item[(i)] For all $m\ge \ark_2,$ it holds that $\ark^{(2)}_m=\ark_2$; for all $n\ge \ark_1,$ it holds that $\ark^{(1)}_n=\ark_1$.
\item[(ii)] For all $m\gneq \ark_2,$ it holds that $\rsize^{(2)}_m=\rsize_2$; for all $n\gneq \ark_1,$ it holds that $\rsize^{(1)}_n=\rsize_1$.
\end{itemize}
\end{Prop}
\begin{proof}
(i) is a direct consequence of Proposition~\ref{Prop:drkmn}, observing that $\drk_{(\matrex_{\mu\nu})}=\left(\drk_{(\matrex_{* n})},\drk_{(\matrex_{m*})}\right)$. For (ii), observe that $(\mu,\nu)\in\NN^2$ is realizing if and only if the bipartite graph $K_{m,n}$ is dependent in $\matrex$, and that is the case if and only if $K_{m,n}$ is dependent in $\matrex_{k,*}$, or $\matrex_{*,\ell}$, for any $k\ge m$ and $\ell\ge n$. The statement then follows from Theorem~\ref{Thm:relsize}~(ii).
\end{proof}

\begin{Rem}\label{Rem:avgrkmn}
Proposition~\ref{Prop:onevstwo}~(i) implies an asympotitic behavior for the average rank:
\begin{eqnarray*}
\ark_1 (\matrex) & = & \lim_{m\rightarrow \infty} \ark^{(1)}_m(\matrex) \\
\ark_2 (\matrex) & = & \lim_{n\rightarrow \infty} \ark^{(2)}_n(\matrex)
\end{eqnarray*}
Proposition~\ref{Prop:onevstwo}~(ii) gives the analogue for the realization size:
\begin{eqnarray*}
\rsize_1 (\matrex) & = & \lim_{m\rightarrow \infty} \rsize^{(1)}_m(\matrex) \\
\rsize_2 (\matrex) & = & \lim_{n\rightarrow \infty} \rsize^{(2)}_n(\matrex)
\end{eqnarray*}
\end{Rem}

A more quantitative statement is given by the following two-sided analogue to the analytic formula in Proposition~\ref{Prop:largerank}:

\begin{Prop}\label{Prop:largerankmn}
Let $\matrex$ be a two-sided bipartite graph matroid limit, assumed to be realizing, with average rank $\ark(\matrex)=(\ark_1,\ark_2)$ and realization site $\rsize(\matrex)=(\rsize_1,\rsize_2)$. Then, there is a unique $\crk\in\NN$, such that
\begin{align*}
\rk(K_{\mu,\nu})&=(\mu-\rsize_1)\cdot \ark_1 + (\nu-\rsize_2)\cdot \ark_2 + \crk\quad\mbox{for all}\;\mu\ge\rsize_1,\nu\ge\ark_1\;\mbox{and for all}\;\mu\ge\ark_2,\nu\ge\rsize_2\\
\rk(K_{\mu,\nu})&\gneq (\mu-\rsize_1)\cdot \ark_1 + (\nu-\rsize_2)\cdot \ark_2 + \crk\quad\mbox{for some}\;\nu\ge\ark_1\;\mbox{if}\;\mu\lneq\rsize_1,\\
\rk(K_{\mu,\nu})&\gneq (\mu-\rsize_1)\cdot \ark_1 + (\nu-\rsize_2)\cdot \ark_2 + \crk\quad\mbox{for some}\;\mu\ge\ark_2\;\mbox{if}\;\nu\lneq\rsize_2.
\end{align*}
Moreover, it holds that $\crk = \rk(K_{\rsize_1,\rsize_2}).$
\end{Prop}
\begin{proof}
Note that the three inequality statements are equivalent to stating $\drk(\mu,\nu)=\ark(\matrex)$ for all $\mu\ge\rsize_1,\nu\ge\ark_1$ and for all $\mu\ge\ark_2,\nu\ge\rsize_2$, and $\drk(\mu,\nu)\gneq \ark(\matrex)$ for some $\nu\ge\ark_1$ if $\mu\lneq\rsize_1$ and some $\mu\ge\ark_2$ if $\nu\lneq\rsize_2$.

We again make use of the fact that $\drk(\mu,\nu)=(\drk_1(\mu),\drk_2(\nu))$, where $\drk_1$ is the growth function of $\matrex_{* \nu}$, and $\drk_2$ is the growth function of $\matrex_{\mu *}$.

By Proposition~\ref{Prop:largerank}, or alternatively, Proposition~\ref{Prop:drkmn}~(i) to~(iii), for fixed $\nu'\ge\ark_1$, there is a unique number $\rsize_1(\nu')\in\NN$ such that
$\drk_1(\mu,\nu')=\ark_1$ for $\mu\ge\rsize_1$ and $\drk_1(\mu,\nu')\gneq \ark_1$ for $\mu\lneq\rsize_1(\nu')$. Similarly, for each fixed $\mu'\ge \ark_2$, there is a unique number $\rsize_2(\mu')\in\NN$ such that
$\drk_2(\mu,\nu)=\ark_2$ for $\nu\ge\rsize_2(\mu')$ and $\drk_2(\mu,\nu)\gneq \ark_2$ for $\nu\lneq\rsize_2(\mu')$. Note that the $\ark_i$ could be a-priori different for each $\mu',\nu'$, but they in fact are not due to the fact that dependence and independence of $K_{\mu,\nu}$ does not depend on $\mu'$ or $\nu'$, see Lemma~\ref{Lem:graphmatroidlimitmn}.

By Proposition~\ref{Prop:onevstwo}, $\rsize_1(\nu')$ and $\rsize_2(\mu')$ are equal to $\rsize_1$ and $\rsize_2$, under these conditions. Therefore, the formulae follow.

Substituting $(\mu,\nu)=(\rsize_1,\rsize_2)$, one obtains $\crk = \rk(K_{\rsize_1,\rsize_2}).$
\end{proof}

\begin{Def}
Let $\matrex$ be a two-sided bipartite graph matroid limit, assumed to be realizing. The number $\crk (\matrex) = \rk(K_{\rsize_1,\rsize_2})$ will be called the \emph{realization rank} of $\matrex$.
\end{Def}

\subsection{Circuits and Bases of Two-Sided Bipartite Graph Limits}

In this section, we study the possible properties of circuits in two-sided limits.

From the proof that the average rank and realization sizes are well-defined, we directly obtain the existence of elementary circuits in the two-sided case.

\begin{Cor}\label{Cor:elcrtmn}
Let $\matrex$ be a two-sided bipartite graph matroid limit, assumed to be realizing, with average rank $(\ark_1,\ark_2)$ and realization size $(\rsize_1,\rsize_2)$. Then, $K_{\rsize_1+1,\ark_1+1}$ and $K_{\ark_2+1,\rsize_2+1}$ are circuit graphs of $\matrex$.
\end{Cor}
\begin{proof}
This follows directly from combining (i) and (ii) in Proposition~\ref{Prop:onevstwo} with Theorem~\ref{Thm:relsize}~(iv).
\end{proof}

\begin{Def}\label{Def:elcrtmn}
Let $\matrex$ be a two-sided bipartite graph matroid limit, assumed to be realizing, with average rank $(\ark_1,\ark_2)$ and realization size $(\rsize_1,\rsize_2)$. Then, $[\rsize_1+1]\times [\ark_1+1]$ is called the \emph{$1$-elementary circuit}, and $[\ark_2+1]\times [\rsize_2+1]$ is called the \emph{$2$-elementary circuit} of $\matrex$. The circuit graph $K_{\rsize_1+1,\ark_1+1}$ is called the \emph{$1$-elementary circuit graph}, and $K_{\ark_2+1,\rsize_2+1}$ the \emph{$2$-elementary circuit graph} of $\matrex$.
\end{Def}

As in the one-sided case, we can prove minimum degree bounds on the vertices of circuit graphs. For that, we introduce a combinatorial definition.
\begin{Def}\label{Def:kcore}
Let $(k,\ell)\in \NN^2$, and let $\calG=(V,W,E)$ be a bipartite graph.  The $(k,\ell)$-core of $\calG$ is the inclusion-wise maximal
vertex-induced subgraph $\calG' = (V',W',E')$ of $\calG$ with the property that every vertex in $V'$
has degree at least $k$, and every vertex in $W'$ has degree at least $\ell$
\end{Def}

\begin{Thm}\label{Thm:deglb}
Let $\matrex$ be the limit of a two-sided bipartite graph matroid limit. Let $N=(\mu,\nu)\in\NN^2$ be a realizing index, let $\tau_1 = \drk_1(\mu-1,\nu)$ and $\tau_2 = \drk_2(\mu,\nu-1)$. Then:
\begin{itemize}
\item[(i)] Let $C=(V,W,E)$ be a circuit graph of $\matrex$, with signature $N$. A vertex in $V$ has vertex degree at least $\tau_1+1$, and a vertex in $W$ has vertex degree at least $\tau_2+1$.
\item[(ii)] For each basis graph $B=(V,W,E)$ with signature $N$, a vertex in $V$ has minimum vertex degree $\tau_1$ and a vertex in $W$ has minimum vertex degree $\tau_2$.
\item[(iii)] For any finite graph $\calG$ with signature $N$,
any edge outside of the $(\tau_1,\tau_2)$-core
of $\calG$ is a bridge.
\end{itemize}
Note that $(\tau_1,\tau_2)\ge \ark (\matrex)$, therefore the bounds are at least the average rank of $\matrex$.
\end{Thm}
\begin{proof}
This is straightforward from Lemma~\ref{Lem:addany}.
\end{proof}

In what follows, elements of $\NN^2$ are compared with the usual Pareto (component-wise) partial order on vectors.

\begin{Def}
Let $\matrex$ be a two-sided bipartite graph matroid. We will denote $\calR(\matrex):=\{N\in\NN^2\;:\;N\;\mbox{is realizing for}\;\matrex\}\subseteq \NN^2.$ Further, we will denote its two projections on $\NN$ by $\calR_1(\matrex):=\{m\;:\;\exists n, (m,n)\in\calR(\matrex)\}$ and $\calR_2 (\matrex):=\{n\;:\;\exists m, (m,n)\in\calR(\matrex)\}$. We call an element of $\calR(\matrex)$, minimal with respect to the inequality partial order on $\NN^2$ a \emph{minimally realizing index} for $\matrex$. The set of minimally realizing indices of $\matrex$ will be denoted by $\partial\calR(\matrex)$. As before, we will write $\partial\calR_1(\matrex):=\{m\;:\;\exists n, (m,n)\in\partial\calR(\matrex)\}$ and $\partial\calR_2 (\matrex):=\{n\;:\;\exists m, (m,n)\in\partial\calR(\matrex)\}$.
\end{Def}

\begin{Thm}\label{Thm:arkvsreal}
Let $\matrex$ be a realizing two-sided bipartite graph matroid limit, with average rank $\ark$. Then, the following are the same:
\begin{itemize}
\item[(i)] The average rank $\ark$.
\item[(ii)] $\left(\min \calR_2(\matrex)-1,\min \calR_1(\matrex)-1\right)$.
\item[(iii)] $\left(\min \partial\calR_2(\matrex)-1,\min \partial\calR_1(\matrex)-1\right)$.
\end{itemize}
\end{Thm}
\begin{proof}
All equivalences follow from the existence of the two elementary circuits and the definition of $\calR_i$, considering Theorems~\ref{Thm:avgrk}, \ref{Thm:relsize}, and Proposition~\ref{Prop:onevstwo}.
\end{proof}

\begin{Thm}\label{Thm:rsizevsreal}
Let $\matrex$ be a realizing two-sided bipartite graph matroid limit, with realization size $\rsize$. Then, the following are the same:
\begin{itemize}
\item[(i)] The realization size $\rsize$.
\item[(ii)] $\left(\max \partial\calR_1(\matrex) - 1,\max \partial\calR_2(\matrex) - 1\right)$.
\end{itemize}
\end{Thm}
\begin{proof}
All equivalences follow from the existence of the two elementary circuits and the definition of $\calR_i$, considering Theorem~\ref{Thm:relsize}, and Proposition~\ref{Prop:onevstwo}.
\end{proof}

For reference, the
structural invariants developed in this section are indicated in the following diagram:

\centerline{\includegraphics[scale=.75]{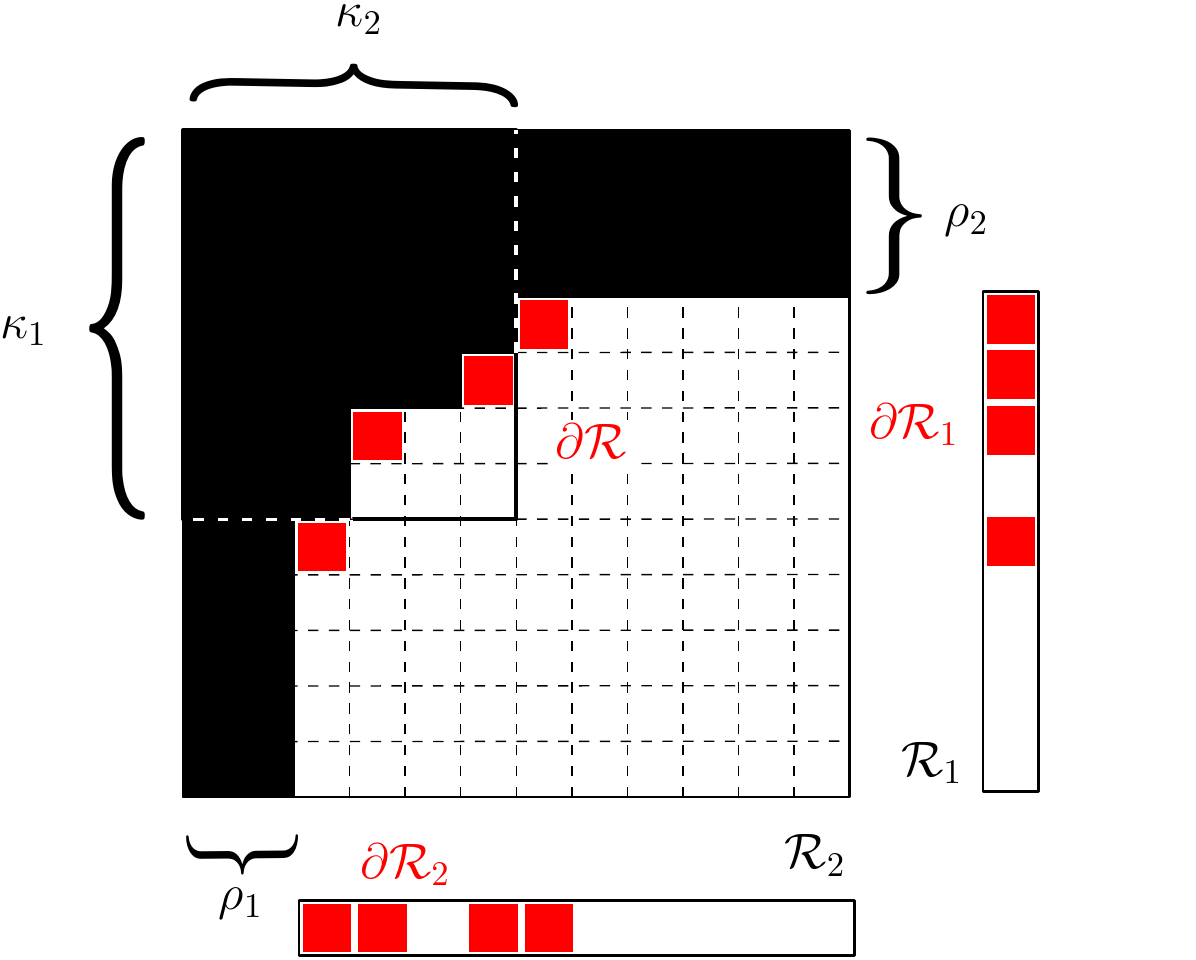}}

The black area indicates the complement of the realizing indices (see Proposition \ref{Prop:staircase-realizing})
and the red squares indicate minimally realizing indices.

Recall that, in the one-sided case, the average rank, realization size, elementary circuit,
and the staircase basis give the asymptotic structure of the limit matroid.  The two-sided
case is more subtle.  We start by introducing the analogue of the staircase basis.
\begin{Def}
Let $\matrex$ be a two-sided bipartite graph matroid limit that is realizing. The \emph{staircase basis sequences}
$\calS_{m*}$ and $\calS_{*n}$ are the sequences of staircase bases of the slices $\matrex_{m*}$ and $\calS_{*n}$,
respectively.
\end{Def}
\begin{Lem}\label{Lem:staircase-limits}
Let $\matrex$ be a two-sided bipartite graph matroid limit that is realizing.  Then the
staircase basis sequences $\calS_{m*}$ and $\calS_{*n}$ converge.
\end{Lem}
\begin{proof}
We prove this for $\calS_{m*}$, since the other case is nearly identical.  Denote by
$d^{(m)}_i$ the jump sequence of $\calS_{m*}$.  Then, by Proposition~\ref{Prop:drkmn},
for every $m\in \NN$, all the $d^{(m)}_i$ are at most $\rsize_2$, and $d^{(m+1)}_i\le d^{(m)}_i$.
It follows that the jump sequences, and thus the sequence of staircase bases, converge.
\end{proof}
We introduce notation for the limit staircase bases.
\begin{Def}
Let $\matrex$ be a two-sided bipartite graph matroid limit that is realizing.  We define $\calS^2 := \lim \calS_{m*}$
and $\calS^1 = \lim \calS_{*n}$.
\end{Def}
The main difference between the one-sided and two-sided cases is that, to understand the
realizing indices, we need the whole staircase basis sequence.
\begin{Prop}\label{Prop:staircase-realizing}
Let $\matrex$ be a two-sided bipartite graph matroid limit that is realizing.  The set of
realizing indices $\calR(\matrex)$ is exactly
\[
\NN^2 \setminus \bigcup_{m\in \NN} \calS_{m*} = \NN^2 \setminus \bigcup_{n\in \NN}\calS_{*n}
\]
\end{Prop}
We defer the proof of Proposition \ref{Prop:staircase-realizing} for the moment, since it will
follow from two auxiliary results.
\begin{Lem}\label{Lem:realizing-indices-ideals}
Let $\matrex$ be a two-sided bipartite graph matroid limit that is realizing.  The set of
realizing indices $\calR(\matrex)$ is a finite union of ideals
\[
\calR(\matrex) = \bigcup_{r\in \partial \calR(\matrex)} \{ r'\in\NN : r' \ge r\}
\]
\end{Lem}
\begin{proof}
That $\partial \calR(\matrex)$ is finite is a straightforward application of Dickson's Lemma.  The
compatibility condition of the complex underlying $\matrex$ then implies that if $r\in\NN$ is
realizing, so is the ideal $\{ r'\in\NN : r' \ge r\}$.
\end{proof}
\begin{Lem}\label{Lem:realizing-indices-jumps}
Let $\matrex$ be a two-sided bipartite graph matroid limit that is realizing, and let
$(\mu,\nu)\in \partial\calR(\matrex)$ be a minimally realizing index.  The $\nu-1$
appears in the jump sequence of $\calS_{\mu*}$ and $\mu-1$ appears in the
jump sequence of $\calS_{*\nu}$.
\end{Lem}
\begin{proof}
By definition, the graphs $K_{\mu,\nu-1}$ and $K_{\mu-1,\nu}$ are independent but
$K_{\mu,\nu}$ is not.
\end{proof}
Proposition \ref{Prop:staircase-realizing} is now straightforward.
\begin{proof}[Proof of Proposition \ref{Prop:staircase-realizing}]
Clearly none of the staircase bases $\calS_{m*}$ and $\calS_{*n}$ can
intersect $\calR(\matrex)$.  On the other hand Lemma \ref{Lem:realizing-indices-jumps}
says that every element of $\partial\calR(\matrex)$ appears as a corner of
at least one $\calS_{m*}$ and at least one $\calS_{*m}$.  According to Lemma \ref{Lem:realizing-indices-ideals},
this is sufficient to conclude that any index that is not realizing is covered by some element of
each staircase basis sequence.
\end{proof}

For two-sided sequences, we can give a more detailed description of the small
circuits.  To do this, we study the evolution of the staircase basis sequences
more closely.  The following diagram, which corresponds to the $3$-dimensional
bipartite rigidity matroid, shows the process.

\centerline{\includegraphics[scale=.8]{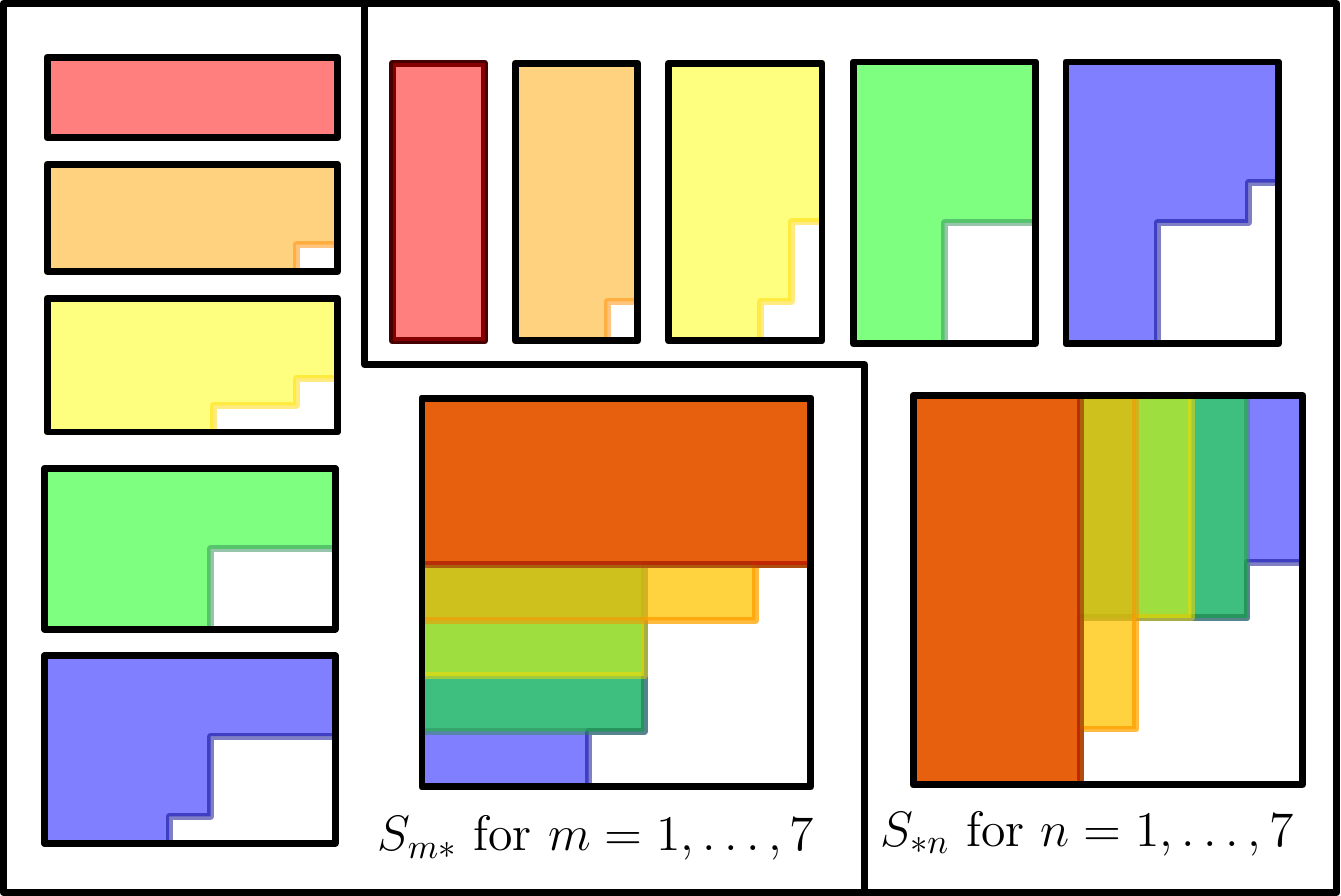}}

Observe that the union of the staircase bases from either sequences gives the same
set, as show in Proposition~\ref{Prop:staircase-realizing}.  The realizing circuits
from Theorem~\ref{Thm:realizing-circuits} below correspond to the first appearance
of each corner.

We are now in a position to analyze circuits associated with corners in the staircase
basis sequence.
\begin{Thm}\label{Thm:realizing-circuits}
Let $\matrex$ be a realizing, two-sided bipartite graph matroid limit. Let $m\in \NN$
be the smallest number such that $(\mu,\nu)$ is a corner of the staircase basis $\calS_{m*}$.
Then the set $X = (\calS_{m*}\cap( [m]\times[\nu-1]))\cup ([\mu]\times\{\nu\})$ contains a circuit $C\subset X$
in $\matrex$, such that: $X$ and $C$ have the same signature, and $\frakS(X)$ fixes $C$.

A similar statement holds for corners in the sequence $\calS_{*n}$.
\end{Thm}
\begin{proof}
Let $X = (\calS_{m*}\cap ([m]\times[\nu-1])) \cup ([\mu]\times\{\nu\})$.  By construction,
$X$ is dependent, and thus it contains a circuit $C$.  Suppose, for a contradiction,
that $C\subsetneq X$ is a proper subset.

If $C$ has smaller signature than $X$ then it
is independent: if $C$ spans fewer columns, then it is a subset of $\calS_{m*}$; if it
spans fewer rows, then, by the way $m$ was selected, it will be a subset of $\calS_{m-1*}$.
The resulting contradiction implies that $C$ and $X$ have the same signature.

Now suppose that $X\setminus C$ is not empty, and let $(i,j)\in X\setminus C$.
If there is an element $\sigma$ of $\frakS(X)$ such that $(i,j)\in \sigma(C)$, then $C\cup \sigma(C)$
is the union of two circuits.  By Proposition \ref{Prop:addany}, both of these contain
all the entries in the set $[\mu]\times\{\nu\}$.  Eliminating any of these implies, by
Proposition \ref{Prop:addany}, a circuit that is a proper subset of
$X\setminus [\mu]\times\{\nu\}$, which is a contradiction.
\end{proof}

\begin{Cor}\label{Cor:realizing-circuits}
In the notations of Theorem \ref{Thm:realizing-circuits}, if the set $X$ is of the form
\begin{itemize}
\item[(i)] $[m]\times [n]$
\item[(ii)] $[m]\times [n]\cup [m+1]\times [\tau]$, with $\tau < n$
\end{itemize}
then for the circuit it holds that $C = X$.
\end{Cor}
\begin{proof}
In case (i), the symmetry group of $X$ acts transitively, so if $X\setminus C$ is non-empty, then
by Theorem \ref{Thm:realizing-circuits}, then $C$ is empty, a contradiction.

For case (ii), we have a case distinction.  Let $(i,j)\in X\setminus C$.  If $(i,j)\le (m,\tau)$, then
Theorem \ref{Thm:realizing-circuits} implies that the entire rectangle $[m]\times [\tau]$ is in
$X\setminus C$.  It then follows that $C$ contains vertices of degree $1$, contradicting
Lemma \ref{Lem:addany}.  Otherwise, by similar reasoning, either a row or column of $X$
is missing form $C$.  Thus, $C$ has smaller signature than $X$, which is also a contradiction.
\end{proof}

\begin{Def}
We define the circuits appearing in the conclusion of Theorem \ref{Thm:realizing-circuits} to be
the \emph{realizing circuits} of a two-sided bipartite graph matroid limit $\matrex$.  (Observe that
the elementary circuits are included in this class.)
\end{Def}

We give characterizations of special realizing circuits:
\begin{Thm}\label{Thm:elcrtmn}
Let $\matrex$ be a two-sided bipartite graph matroid limit, with growth function $\drk$. Let $(\mu,\nu)\in\NN^2$. Let $(\tau_1,\tau_2)=\drk (\mu-1,\nu-1)+(1,1), \varsigma_1=\drk_1 (\mu-1,\nu)+1,$ and $\varsigma_2=\drk_2 (\mu,\nu-1)+1$. Consider the following statements:
\begin{itemize}
\item[(i)] $(\mu,\nu)\in \partial\calR(\matrex),$ that is, $(\mu,\nu)$ is a minimally realizing index for $\matrex$.
\item[(ii)] $\tau_1\lneq \varsigma_1$, and $K_{\mu-1,\nu}$ is independent.
\item[(iii)] $\tau_2\lneq \varsigma_2$, and $K_{\mu,\nu-1}$ is independent.
\item[(iv)] It holds that $\tau_2\lneq \nu,$ the set $[\mu-1]\times[\nu]\cup [\mu]\times[\tau_2]$ is a circuit of $\matrex$, the corresponding graph $K_{\mu-1,\nu}\cup K_{\mu,\tau_2}$ is a circuit graph of $\matrex$.
\item[(v)] It holds that $\tau_1\lneq \mu,$ the set $[\tau_1]\times[\nu]\cup [\mu]\times[\nu-1]$ is a circuit of $\matrex$, the corresponding graph $K_{\tau_1,\nu}\cup K_{\mu,\nu-1}$ is a circuit graph of $\matrex$.
\end{itemize}
Then, (ii) implies (iv), and (iii) implies (v). Moreover, (i) is equivalent to the conjunction ``(iv) and (v)''.
\end{Thm}
\begin{proof}
(iv) and (v) $\Rightarrow$ (i): Since $\tau_2\lneq \nu$, the graph $K_{\mu,\nu}$ is dependent, so $(\mu,\nu)$ is a realizing index. Since $K_{\mu-1,\nu}\cup K_{\mu,\tau_2}$ is a circuit, all $K_{k,\ell}$ with $k\lneq \mu, \ell\le \nu$ are proper subgraphs, therefore independent. Similarly, since $K_{\tau_1,\nu}\cup K_{\mu,\nu-1}$ is a circuit, all $K_{k,\ell}$ with $k\le \mu, \ell\lneq \nu$ are independent. Therefore, $(\mu,\nu)$ is a minimally realizing index.\\

(i) $\Rightarrow$ (iv), (v) can be obtained from Corollary~\ref{Cor:realizing-circuits}. Similarly, (ii) $\Rightarrow$ (iv),(v) are consequences of Corollary~\ref{Cor:realizing-circuits}.
\end{proof}

We close with bounding the signature of all circuits, generalizing Theorem~\ref{Thm:circuitsize} to the two-sided case:

\begin{Thm}\label{Thm:circuitsizemn}
Let $\matrex$ be a two-sided bipartite graph matroid limit, with average rank $(\ark_1,\ark_2)$, realization size $(\rsize_1,\rsize_2)$ and realisation rank $\crk$. Let $C$ be a circuit of $\matrex$, with signature $(k,\ell)$. Then,
\begin{align*}
k&\le \crk - \ark_1\cdot \rsize_1 + \ark_2\cdot (\ell -\rsize_2) + 1\\
\ell &\le \crk - \ark_2\cdot \rsize_2 + \ark_1\cdot (k -\rsize_1) + 1\\
k&\ge\frac{(\rsize_1\cdot\ark_1+\rsize_2\cdot\ark_2-\crk-1)(\ark_2+1)}{\ark_1\cdot \ark_2 -1}\\
\ell &\ge\frac{(\rsize_1\cdot\ark_1+\rsize_2\cdot\ark_2-\crk-1)(\ark_1+1)}{\ark_1\cdot \ark_2 -1}
\end{align*}
In particular, if $\ark_1=\ark_2, \rsize_1=\rsize_2$, e.g.~when transposition is an isomorphism of $\matrex$, one has
\begin{align*}
k&\le \crk + \ark\cdot (\ell - 2\rsize) + 1\\
\ell &\le \crk + \ark\cdot (k - 2\rsize) + 1\\
k,\ell&\ge \frac{2\rsize\cdot\ark-\crk-1}{\ark -1},
\end{align*}
where writing $\ark=\ark_1=\ark_2$ and $\rsize=\rsize_1=\rsize_2$ by abuse of notation.
\end{Thm}
\begin{proof}
The first two equations follow from applying Proposition~\ref{Prop:onevstwo} to compare $\matrex$ to the limit matroids on the ground sets $[k]\times \NN$ and $\NN\times [\ell]$, then applying Theorem~\ref{Thm:circuitsize}. The second two equations are obtained from substituting the first two into each other, then solving once for $k$, once for $\ell$. The last three equations are obtained from substituting $\ark=\ark_1=\ark_2$ and $\rsize=\rsize_1=\rsize_2$ into the first four.
\end{proof}

We return to our running examples:
\paragraph{Determinantal matroids}
\begin{Prop}\label{Prop:detM-arkmn}
Consider the infinite determinantal matroid $\detM(\NN\times \NN,r)$, as obtained in Definition~\ref{Def:limitmatrmn}, with average rank $(\ark_1,\ark_2)$, realization size $(\rsize_1,\rsize_2)$ and realization rank $\crk$. Then,
$$\ark_1=\ark_2 = r,\quad \rsize_1=\rsize_2  = r,\quad \crk = r^2.$$
Both the $1$- and $2$-elementary circuit graphs are equal to $K_{r+1,r+1}$.
\end{Prop}
\begin{proof}
This can be deduced from the dimension formula in Remark~\ref{Rem:detfacts}~(i), and Proposition~\ref{Prop:largerank}, or, alternatively Propositions~\ref{Prop:onevstwo} and~\ref{Prop:detM-ark}.
\end{proof}

Since transposition is an automorphism of the the determinantal matroid, application of Theorem~\ref{Thm:circuitsizemn} yields no improvmenet over Proposition~\ref{Prop:circuitsizedet}.

\begin{Prop}\label{Prop:detM-real}
The set of minimally realizing indices for the matroid $\detM(\NN\times \NN,r)$ is $\partial\calR =\{(r+1,r+1)\}$. All realizing circuit graphs coincide with the elementary circuit graph $K_{r+1,r+1}$.
\end{Prop}
\begin{proof}
$\detM(\NN\times \NN,r)$ is a realizing matroid, see for example Proposition~\ref{Prop:detM-arkmn}. By substituting the numbers from Proposition~\ref{Prop:detM-arkmn} into Theorems~\ref{Thm:arkvsreal} and~\ref{Thm:rsizevsreal}, one obtains $(r+1,r+1)\le N\le (r+1,r+1)$ for all $N\in \partial\calR$, and $\partial\calR\neq \varnothing$, which implies the statement.
\end{proof}

\paragraph{Bipartite rigidity matroids}
\begin{Prop}\label{Prop:CMM-arkmn}
Consider the infinite bipartite rigidity matroid $\CMM(\NN\times \NN,r)$, as obtained in Definition~\ref{Def:limitmatrmn}, with average rank $(\ark_1,\ark_2)$, realization size $(\rsize_1,\rsize_2)$ and realization rank $\crk$. Then,
$$\ark_1=\ark_2 = r,\quad \rsize_1=\rsize_2  = \binom{r+1}{2},\quad \crk = r\cdot (m+n) - {r\choose 2}.$$
The $1$-elementary circuit graph is $K_{\binom{r+1}{2}+1,r+1}$, while the $2$-elementary circuit graph is $K_{r+1,\binom{r+1}{2}+1}$.
\end{Prop}
\begin{proof}
This can be deduced from the dimension formula in Remark~\ref{Rem:CMfacts}~(i), and Proposition~\ref{Prop:largerank}, or, alternatively Propositions~\ref{Prop:onevstwo} and~\ref{Prop:CMM-ark}.
\end{proof}
Since transposition is an automorphism of the the determinantal matroid, application of Theorem~\ref{Thm:circuitsizemn} yields no improvement over Proposition~\ref{Prop:circuitsizeCM}.

\begin{Prop}\label{Prop:CMM-real}
The set of minimally realizing indices for the matroid $\CMM(\NN\times \NN,1)$ is $\partial\calR=\{(2,2)\}$, and all realizing circuit graphs coincide with $K_{2,2}$. For $r\ge 2$, the set of minimally realizing indices for the matroid $\CMM(\NN\times \NN,r)$ is $\partial\calR =\left\{(\binom{r+1}{2}+1,r+1), (r+2,r+2), (r+1,\binom{r+1}{2}+1)\right\}$. The $1$-realizing circuit graphs and the $2$-realizing circuit graphs obtained from $\partial\calR$ coincide: they are $K_{\binom{r+1}{2}+1,r+1}, K_{r+2,r+2}$ and $K_{r+1,\binom{r+1}{2}+1}$. The remaining realizing circuit graphs are $K_{k,r+1}\cup K_{r+1,\ell}$ with $k,\ell\gneq r+1$ and $k+\ell = \binom{r+1}{2}+r+2$.
\end{Prop}
\begin{proof}
$\CMM(\NN\times \NN,r)$ is a realizing matroid, see for example Proposition~\ref{Prop:detM-arkmn}. By substituting the numbers from Proposition~\ref{Prop:CMM-arkmn} into Theorems~\ref{Thm:arkvsreal} and~\ref{Thm:rsizevsreal}, one obtains $(r+1,r+1)\le N\le (r+1,r+1)$ for all $N\in \partial\calR$, and $\partial\calR\neq \varnothing$, which implies the statement for the realizing circuits obtained from $\partial\calR$. The remaining realizing circuits are obtained from the dimension formula in Remark~\ref{Rem:CMfacts}~(i).
\end{proof}

\subsection{Infinite Graph matroids}\label{sec:size.graphs}
In this section, we briefly sketch how the results of the previous section extend to graph matroids. The proofs of the statements can be obtained in an analogous way to the previously presented ones. Some of the results on one-sided limits can be re-obtained by bipartition from the graph setting.

\begin{Def}\label{Def:graphseqsym}
Let $(\matrex_\nu)$ be an injective sequence of graph matroids, such that $\matrex_\nu$ has ground set $E_\nu= {[\nu]\choose 2}.$
The sequence $(\matrex_\nu)$ is called \emph{graph matroid sequence},
and the limit $\matrex=\varinjlim_{\nu\in \NN} \matrex_\nu$ a \emph{graph matroid limit}, or the limit of $(\matrex_\nu)$.
\end{Def}

\begin{Def}\label{Def:drksym}
Let $\matrex$ be a graph matroid limit. We define the \emph{growth function} of $\matrex$ as
$$\drk_\matrex(\cdot) : \NN \rightarrow \NN,\quad \nu\mapsto \rk (K_{\nu+1} | K_\nu).$$
If clear from the context, we will just write $\drk(\cdot)$, omitting the dependence on $\matrex$.
\end{Def}

\begin{Prop}\label{Prop:drksym}
Let $\matrex$ be a graph matroid limit. Then, the growth function $\drk(\cdot)$ of $\matrex$ is non-negative and non-increasing. That is, $0\le \drk(\nu+1)\le \drk(\nu)$ for all $\nu\in\NN$.
\end{Prop}

\begin{Def}
Let $\matrex$ be a graph matroid limit with growth function $\drk$. We define the \emph{average rank} of $\matrex$ to be $\ark (\matrex)=\lim_{\nu\rightarrow\infty} \drk(\nu)$.
\end{Def}

\begin{Prop}\label{Prop:largeranksym}
Let $\matrex$ be a graph matroid limit on the ground set ${\NN\choose 2}$, with average rank $\ark$. Then, there are unique numbers $\crk,\rsize\in\NN$, such that
\begin{align*}
\rk(K_{\nu})&=(\nu-\rsize)\cdot \ark + \crk\quad\mbox{for all}\;\nu\ge\rsize\\
\rk(K_{\nu})&\gneq(\nu-\rsize)\cdot \ark + \crk\quad\mbox{for all}\;\nu\lneq\rsize.
\end{align*}
Moreover, $\crk = \rk(K_{\rsize})$.
\end{Prop}

Note however that $K_\rsize$ will in general not be a circuit, nor the symmetrization $K_{\ark+1,\rsize+1}^\sym$ which is obtained by gluing vertices.

\section{Circuit Polynomials and Top-Degrees}\label{sec:symmetries-and-circuits}

By the unique correspondence between algebraic and coordinate representations of matroids, given in Theorem~\ref{Thm:crypto}, circuits in algebraic matroids can be studied by means of commutative algebra. While the viewpoint of coordinate matroids is most useful for this section, all of what follows will also be valid for the equivalent cryptomorphic realizations by linear, basis, or algebraic matroids.

\subsection{Top-degree}

The top-degree is a multivariate generalization of the the degree of a polynomial which has been introduced in~\cite[Tutorial~48]{KR}. We slightly generalize it to the case of potentially infinite sets of variables, in order make it applicable later to circuits with potentially infinite symmetries. This short exposition on top-degree has mostly notational purpose in order to capture the circuit invariants, which will follow in the next sections, in a concise way.

\begin{Not}
Let $I$ be an index set, let $\xbf =\{X_i\;:\; i\in I\}$ be a set of coordinates. We will denote by $K[\xbf]$ the corresponding polynomial ring in potentially infinitely many variables.
\end{Not}

\begin{Not}
Let $\nb\subseteq I$ be a finite sub-multiset of $I$, i.e., $\nb$ is a subset of $I$ where elements can occur with multiplicities.\footnote{We adopt the usual multiset convention that a sub-multisets $\nb$ of $I$ is a subset of $I$ with possible multiplicities; i.e., a function $I\rightarrow \mathbb{N}$. For $i\in\nb$, we denote by $\nb(i)$ the multiplicity of $i$ in $\nb$, and by $\cup$ and $\cap$ the usual multiset union and intersection, and by $\uplus$ the multiset sum.}
Then, by convention, we will write
$$X^\nb:=\prod_{i\in \nb} X_i.$$
Also, we will denote inverses by
$$X^{-\nb}=\left(X^{\nb}\right)^{-1}.$$
Furthermore, for a multiset $\nb$ and $k\in\NN$, we will denote
$$k\ast \nb := \bigcup_{i=1}^k \nb.$$
\end{Not}

\begin{Def}
Let $f\in S$, assume
$$f=\sum_{\nb \subseteq I}a_{\nb} X^\nb \quad\mbox{with}\;a_{\nb}\in K,$$
where by definition of $K[\xbf]$ only finitely many $a_{\nb}$ are non-zero. Then, the multiset
$$\topdeg f = \bigcup_{a_{\nb}\neq 0} \nb$$
is called {\it top-degree} of $f$.
We denote the set associated to the multiset $\topdeg f$ by $\supp f$, called the {\it support} of $f$.
\end{Def}

We want to remark that the top-degree is, in essence, a compact way to write the collection of single degrees:
\begin{Lem}\label{Lem:deg-topdeg}
Let $f\in K[\xbf]$. For $i\in I$, denote by $d_i:=\deg_{X_i} f$ the degree of $f$ considered in the single variable $X_i$. Then,
$$\topdeg f = \bigcup_{i\in I} d_i\ast \{i\}.$$
\end{Lem}
\begin{proof}
This follows from the definitions of degree, and top-degree, observing that taking unions of multisets is equivalent to taking the maximum of $\deg_{X_i}$.
\end{proof}

\begin{Def}
Let $\frakI\subseteq R.$ Then the set of all top-degrees $\topdeg(\frakI)=\{\topdeg g\;:\; g\in\frakI\}$ is called the {\it exponent set} of $\frakI.$ Similarly, the set of all supports $\supp(\frakI)=\{\supp g\;:\; g\in\frakI\}$ is called the {\it support set} of $\frakI.$ Both $\topdeg (\frakI)$ and $\supp (\frakI)$ have the natural inclusion partial order attached to its elements.
\end{Def}

\subsection{Circuit polynomials}
Every circuit in a coordinate matroid has a unique polynomial attached to it, the circuit polynomial:

\begin{Lem}\label{Lem:circuitpoly}
Consider a coordinate matroid $\matrex = \matrcrd (\frakP, \xbf),$ with fixed realization $(\frakP,\xbf)$. Let $C$ be a circuit of $\matrex$. Then, there is (up to multiplication with a unit in $K$) a unique polynomial $\theta_C $ such that
$$\frakP\cap K[\xbf(C)] = \theta_C \cdot K[\xbf (C)]$$
\end{Lem}
\begin{proof}
Since $C$ is a circuit of $\matrex,$ the ideal $\frakP\cap K[\xbf(C)]$ is of height $1$ in $K[\xbf(C)].$ Therefore, for example by \cite[I.\S 7, Proposition 4]{Mumford}, $\frakP\cap K[\xbf(C)]$ is principal from which the statement follows.
\end{proof}

Lemma~\ref{Lem:circuitpoly} states that $\theta_C $ is an invariant of the circuit $C$, making the following well-defined:
\begin{Def}\label{Def:circuitpoly}
Keep the notation of Lemma~\ref{Lem:circuitpoly}. Then, we will call $\theta_C$ \emph{circuit polynomial} of the circuit $C$, associated to the realization $(\frakP, \xbf)$. Understanding that there are in general more than one circuit polynomial, all constant multiples of each other, we will say that $\theta_C$ is {\it the} circuit polynomial if the subsequent statement does not depend on the choice of the particular multiple.
\end{Def}

Note again that by the cryptomorphisms given in Theorem~\ref{Thm:crypto}, circuits in algebraic or basis matroids as well have unique circuit polynomials attached to them as invariants.

\begin{Rem}\label{Rem:circuitpoly}
Existence and uniqueness of the circuit polynomial, as given in Lemma~\ref{Lem:circuitpoly}, has been already proven as Lemma~1.0 and Corollary~1.1 of~\cite{DL87}. However, no further structural statements have been made there, as the purpose of the discussion there is different.
\end{Rem}

\begin{Lem}\label{Lem:suppI}
Let $\matrex=\matrcrd (\frakP, \xbf)$ be some coordinate matroid. Then, $\supp (\frakP)=\calI(\matrex)^C$.
\end{Lem}
\begin{proof}
One has $\zbf\in \supp (\frakP)$ if and only if there is a non-zero polynomial $P\in \frakP\cap K[\zbf]$ such that $P(\zbf)=0$. This is equivalent to $\zbf$ being algebraically dependent, which is true if and only if $\zbf=\xbf(S)$ for $S\not\in \calI (\matrex)$.
\end{proof}

An important property of circuit polynomials is that they have minimal top-degree and minimal support:
\begin{Prop}\label{Prop:circminsupp}
Let $\matrex=\matrcrd (\frakP, \xbf)$ be some coordinate matroid. Then:
\item[(i)] The set of circuits $\calC(\matrex)$ is exactly the set of minima of $\supp (\frakP)$ w.r.t.~inclusion partial order.
\item[(ii)] Let $C$ be a circuit of $\matrex$. Then, $\topdeg \theta_C$ is minimal in $\topdeg (\frakP)$, and $\supp \theta_C$ is minimal in $\supp (\frakP)$.
\end{Prop}
\begin{proof}
(i) follows from the definition/fact that a circuit is an inclusion-minimal set in $\calI(\matrex)^C$, and Lemma~\ref{Lem:suppI}.\\

(ii) $\theta_C$ is (up to constant multiple) the unique element in $\frakP$ with support $\supp \theta_\zbf (\frakP)$. Therefore, it suffices to prove the statement for the supports, since inclusions are preserved under passing from multisets to the underlying sets, thus from top-degree to support. But the latter follows from Lemma~\ref{Lem:suppI}, or directly from (i).
\end{proof}

The top-degrees of the circuit polynomials are further invariants associated to any matroid. While the circuits specify the dependence
structure, describing solvability, the top-degrees yields insight on the degree structure, describing the number of solutions. Naive
generalizations or converses of Proposition~\ref{Prop:circminsupp} fail, as the multigraded part associated to minimal top-degrees
does not need to be a principal ideal. For matroids with symmetries, such as graph matroids, we introduce notation to depict the
circuit and the degree of its polynomial simultaneously:

\begin{Not}
Let $\matrex = \matrcrd (\frakP, \xbf)$ be some coordinate matroid on the ground set $[m]\times [n]$, so that the
coordinates are  indexed as $\xbf = \{X_{ij}\;:\; 1\le i\le m, 1\le j\le n\}$. Let $C$ be a circuit in $\calM$.
We will depict $C$ as a \emph{top-degree mask}: the mask corresponding to $C$ is a matrix $A$ in $\NN^{m\times n}$,
with entries $A_{ij} = \deg_{X_{ij}} \theta_\xbf(C) (\frakP).$

If, in addition, $\calM$ is a bipartite graph matroid, we can also depict $C$ as a
\emph{weighted bipartite graph}; the graph is the one with the top-degree mask as its adjacency
matrix and the edge-weights are the variable degrees.
The symmetric top-degree mask and the weighted graph for circuits in graph matroids are defined in analogous manner.
\end{Not}
Note that when zeroes in the top-degree mask are replaced with $\circ$, and all other entries with $\bullet$,
one obtains the ordinary mask corresponding to the circuit $C$. Similarly, the (bipartite) graph associated with
$C$ is obtained by removing the weights.
\begin{Ex}
We return to Examples~\ref{Ex:detm442} and ~\ref{Ex:detmsym442}. In the case of $\detM(4,4,2)$, the top-degree masks and weighted graphs do not carry much new information. With only one exception, all of the appearing variables have degree $1$, as pictured below:

\centerline{\nmatrixaMod \hspace{4mm} \includegraphics[scale=.7]{bipGr1} \hspace{1cm}
\nmatrixb \hspace{4mm}\includegraphics[scale=.7]{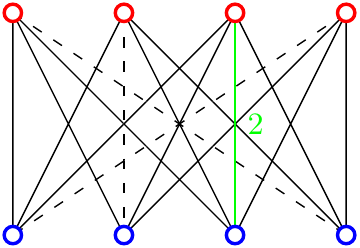}}

On the other hand, the symmetric case $\detMsym(4,4,2)$ has top-degree mask and weighted graph carrying much more
information about the circuit polynomials than the unlabeled counterparts, as depicted in Table \ref{detsymtd}.

\begin{footnotesize}
\begin{table}[h]
\begin{center}
\begin{tabular}{cccc}
\includegraphics[scale=.7]{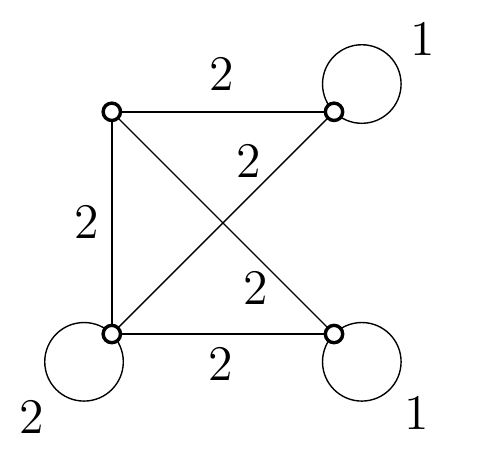} &
\includegraphics[scale=.7]{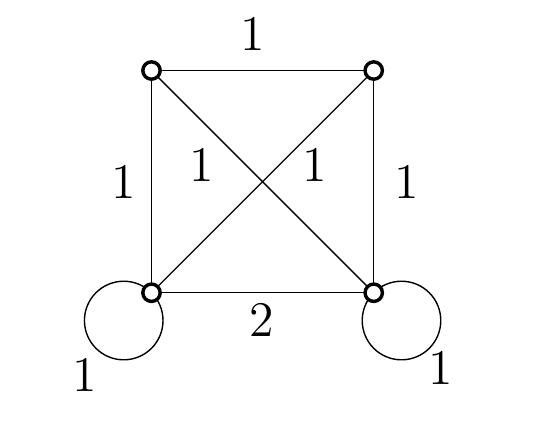} &
\includegraphics[scale=.7]{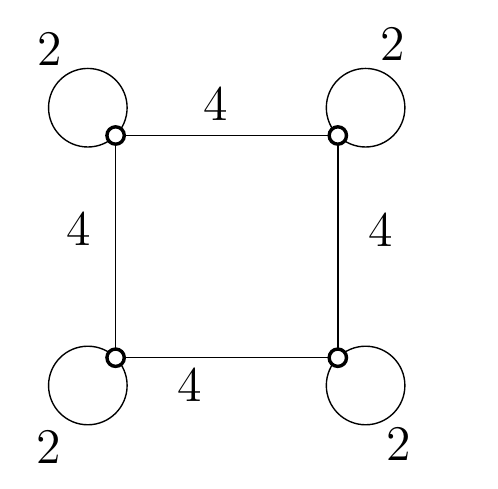} &
\includegraphics[scale=.7]{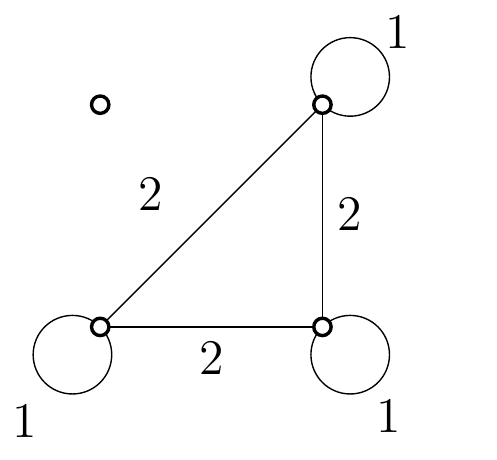} \\
\nsymmMatrixa & \nsymmMatrixb & \nsymmMatrixc & \nsymmMatrixd
\end{tabular}
\end{center}
\caption{Top-degree Masks and weighted graphs for $\detMsym(4,4,2)$}
\label{detsymtd}
\end{table}%
\end{footnotesize}
\end{Ex}

\subsection{The Multihomogenization}\label{sec:symmetries.MH}
The concepts of multihomogeneity in this section are an extension of the definitions by Kreuzer and Robbiano in \cite{KR} to the infinite case.

For multihomogenization, we introduce a copy $Y_i$ for every variable $X_i$ in the original set of coordinates.
\begin{Not}\label{Not:MHring}
Let $I$ be an index set, let $\xbf =\{X_i\;:\; i\in I\}$ be a set of coordinates. Let $\ybf =\{Y_i\;:\; i\in I\}$ be another such set. We will denote the ideal $\langle \ybf\rangle \subseteq S$ which is generated by all $Y_i$ by $\frakY.$ For abbreviation, we will write $R=K[\xbf]$ and $S=R[\ybf]$.
\end{Not}

The top-degree is extended to polynomials in both sets of variables as follows:
\begin{Not}
Let $f\in S$, let $f'\in R$ be the polynomial obtained by the substitution $Y_i=X_i$ for all $i\in I$. Then, by convention, we define
\begin{align*}
\topdeg f &:= \topdeg f'\\
\supp f &:= \supp f'.
\end{align*}
\end{Not}

\begin{Def}
Let $f\in R.$ The polynomial $f^\vartriangle\in S$ defined as
$$f^\vartriangle = \sum_{\nb \subseteq I}a_\nb \cdot X^\nb \cdot Y^{\topdeg f\setminus\nb}= \sum_{\nb \subseteq I}a_\nb \cdot X^\nb \cdot Y^{\topdeg f}\cdot Y^{-\nb}$$
is called {\it multihomogenization} of $f$.

Conversely, $g\in S$ is called multihomogeneous if $\topdeg (g)=\topdeg(M)$ for each (non-zero) monomial $M$ occurring in $g$.
\end{Def}

\begin{Rem}
A more intrinsic but less constructive way to define the multihomogenization of $f\in R$ is as the result of homogenizing $f$ subsequently with respect to each $X_i$-grading with a new homogenizing variable. Since only finitely many $X_i$ occur in $f$, there are only finitely many homogenizations to perform. The top-degree notion on $S$ then also canonically arises from the top-degree on $R$.
\end{Rem}

\begin{Lem}\label{Lem:multhom}
Let $f\in R$. Then, $f^\vartriangle$ is multihomogeneous, of (ordinary) degree $\# \topdeg f$, and $\topdeg f^\vartriangle = \topdeg f.$
\end{Lem}
\begin{proof}
The statements are elementary to check and follow from the definitions.
\end{proof}

\begin{Lem}\label{Lem:addmulthom}
Let $f,g\in S$ be multihomogeneous. Then:
$$\topdeg (f+g)\subseteq \topdeg f\cup \topdeg g\quad\mbox{and}\quad\topdeg(f\cdot g)=\topdeg f\uplus \topdeg g.$$
Moreover, if $\topdeg f\neq \topdeg g$, then
$$\topdeg (f+g) = \topdeg f\cup \topdeg g.$$
\end{Lem}
\begin{proof}
The statements are elementary to check and follow from the definitions.
\end{proof}

\begin{Cor}\label{Cor:multhommin}
Let $\calC\subseteq \calP (I).$ Assume all elements of $\calC$ are minimal with respect to the inclusion partial order on multisets.
Let $f_\nb\in S,\nb\in \calC$ be a family of multihomogenous polynomials such that $\topdeg f_\nb = \nb.$ Then, the strict inclusion
$$\langle f_\nb\;;\;\nb\in\calC\setminus\{c\}\rangle \subsetneq\langle f_\nb\;;\;\nb\in\calC \rangle$$
of ideals in $S$ holds for all $c\in\calC$.
\end{Cor}
\begin{proof}
We prove the statement by contradiction. I.e., assume that there is a $c\in\calC$ such that
$$\langle f_\nb\;;\;\nb\in\calC\setminus\{\cb\}\rangle = \langle f_\nb\;;\;\nb\in\calC \rangle.$$
This implies that there is a finite subset $C\in\calC\setminus\{\cb\}$ such that
$$f_\cb = \sum_{\nb\in C}f_\nb g_\nb\quad\mbox{with homogenous}\;g_\nb\in S\setminus\{0\}.$$
Lemma~\ref{Lem:addmulthom} implies, together with the fact that the elements of $\calC$ are minimal, that
$$\cb \supseteq \bigcup_{\nb\in C} \nb ,$$
which implies $\nb\subseteq \cb$ for any $\nb\in C$, which is a contradiction to minimality.
\end{proof}

\begin{Def}
Let $\frakI\subseteq R$ be an ideal. We define
$$\frakI^\vartriangle = \langle f^\vartriangle\;;\; f\in \frakI\rangle$$
to be the ideal generated by all multihomogenizations of elements in $\frakI$.
\end{Def}

\begin{Def}
The canonical quotient morphism of rings
$$.^\triangledown: S\rightarrow S/\langle Y_i-1\;;\; i\in I \rangle = R$$
where the identification on the right hand side is the canonical one, will be denoted by the superscript $[]^\triangledown.$ That is, for an element $f\in S$ or an ideal $\frakI\subseteq S$, we will denote the canonical image in $R$ by $f^\triangledown$ resp.~$\frakI^\triangledown$.
\end{Def}

\begin{Lem}\label{Lem:homdehom}
For $f\in R$ and any multihomogenous polynomial $h\in R$, it holds that
$$f= f^{\vartriangle\triangledown}\quad\mbox{and}\quad (h\cdot f)^\vartriangle = h\cdot f^\vartriangle.$$
For ideals $\frakI\subseteq R$ and homogenous $\frakJ\subseteq S$, it holds that
$$ \frakI=\frakI^{\vartriangle\triangledown}\quad\mbox{and}\quad (\frakJ : \frakY)=\frakJ^{\triangledown\vartriangle}.$$
\end{Lem}
\begin{proof}
The statements are elementary to check and follow from the definitions.
\end{proof}

\subsection{MH-bases}

An MH-basis is a generating set in which no variable cancellations need to occur to generate the ideal:

\begin{Def}
Let $\frakI\subseteq R$ be an ideal. Then, a generating set $\{f_j\;:\; j\in J\}$ is called {\it MH-basis} of $\frakI$ if for every $h\in \frakI$, there exists a finite subset $J'\subseteq J$ and $g_i\in S, i\in J'$ such that
$$h=\sum_{i\in J'} f_i\cdot g_i\quad\mbox{and}\quad \topdeg f_i \uplus \topdeg g_i \subseteq \topdeg h\;\mbox{for all}\; i.$$
\end{Def}

The definition of MH-bases does not use homogenization. However, there is a concise characterization of MH-bases in terms of homogenization which can serve as alternate definition:

\begin{Prop}\label{Prop:MHbasis}
Let $\frakI\subseteq R.$ The following are equivalent:
\begin{description}
\item[(i)] $\{f_1, f_2, \dots \}$ is an MH-basis of $\frakI.$
\item[(ii)] $\{f_1^\vartriangle,f_2^\vartriangle,\dots \}$ is a generating set for $\frakI^\vartriangle.$
\end{description}
\end{Prop}
\begin{proof}
For readability, we prove the statement in the case of a finite MH-basis $\{f_1, f_2, \dots, f_k \}$. The proof of the general statement is in complete analogy.\\

First note that $S$ and $\frakI^\vartriangle$ decompose as direct sums
\begin{align*}
S = \bigoplus_{\nb\subseteq I} S_\nb\quad\mbox{and}\quad\frakI^\vartriangle = \bigoplus_{\nb\subseteq I} \frakI^\vartriangle_\nb,
\end{align*}
where the direct sum runs over all sub-multisets of $I$, and where $S_\nb,\frakI^\vartriangle_\nb$ denote the sub-$K$-vector spaces spanned by the multihomogenous top-degree-$\nb$-elements of $S,\frakI^\vartriangle$. For $\mb\subseteq I$, we will denote by
$$S_{\le \mb} = \bigoplus_{\nb\subseteq \mb} S_\nb\quad\mbox{and}\quad \frakI_{\le \mb} = \bigoplus_{\nb\subseteq \mb} \frakI^\vartriangle_\nb$$
(ii)$\Rightarrow$ (i): By definition, for every homogenous $h\in \frakI^\triangle$, there exist $g_1,\dots, g_k\in S$ such that
$$h=\sum_{i=1}^k f_i^\vartriangle \cdot g_i.$$
Let $\mb_i = \topdeg h - \topdeg f_i$. By the direct sum decompositions above, one can choose the $g_i$ in $S_{\le \mb_i}.$ Dehomogenizing, one obtains
$$h^\triangledown=\sum_{i=1}^k f_i^{\vartriangle\triangledown} \cdot g_i^\triangledown =\sum_{i=1}^k f_i\cdot g_i^\triangledown.$$
Now $.^{\vartriangle\triangledown}$ is surjective as well as $.^\vartriangle$, thus for each $\ell\in\frakI$ there exists $h$ such that $\ell =h^\triangledown$. Furthermore, $\topdeg g_i^\triangledown\subseteq \topdeg g_i$. Therefore, $\{f_1,\dots, f_k\}$ is an MH-basis of $\frakI$.\\

(i)$\Rightarrow$ (ii): By definition, for every $h\in \frakI$, there exist $g_1,\dots, g_k\in R$ such that
$$h=\sum_{i=1}^k f_i\cdot g_i\quad\mbox{and}\quad \topdeg f_i \uplus \topdeg g_i \subseteq \topdeg h\;\mbox{for all}\; i.$$
Homogenizing, one obtains that
$$h^\vartriangle=\sum_{i=1}^k f_i^\vartriangle\cdot g_i^\vartriangle\cdot w_i \quad\mbox{for some}\; w_i\in \frakY.$$
Since any $\ell\in \frakI^\vartriangle$ can be written as a finite sum of such $h^\vartriangle$, the set $\{f_1^\vartriangle,\dots, f_k^\vartriangle\}$ is a generating set for $\frakI^\vartriangle.$
\end{proof}

\begin{Prop}\label{Prop:minMH}
Let $\frakI\subseteq R,$ let $\{f_j\;:\; j\in J\}$ be an MH-basis of $\frakI$. Let $\nb$ be a minimum of $\topdeg(\frakI).$ Then, there is a $j$ such that $\topdeg f_j = \nb.$
\end{Prop}
\begin{proof}
This is a direct consequence of Corollary~\ref{Cor:multhommin} and Proposition~\ref{Prop:MHbasis}.
\end{proof}

Naive converses of Proposition~\ref{Prop:minMH} - i.e., that a set of elements having all minimal top-degrees is an MH-basis, or that all elements of a MH-basis have minimal top-degrees - are false. One important corollary is however the following:

\begin{Cor}\label{Cor:minMH}
Let $\frakP\subseteq R$ be a prime ideal, let $\{f_j\;:\; j\in J\}$ be an MH-basis of $\frakI$. Then, an MH-basis $\{f_j\}$ for $\frakP$ contains all circuit polynomials $\theta_C$ for circuits $C$ in $\matrcrd (\frakP, \xbf).$
\end{Cor}

\subsection{Finiteness of top-degrees}\label{sec:finiteness}
In Theorem~\ref{Thm:circuitsize}, if has been shown by purely combinatorial arguments that for one-sided bipartite graph limit
$\matrex$, there are a finite number of circuit graphs. If $\matrex$ is, in addition, algebraic, a similar finiteness statement holds for the set of minimal top-degrees.

For this, we will consider the following concept of finiteness, which is the algebraic counterpart to a finite set of graphs:
\begin{Def}
Let $\xbf =\{X_{ij}\;:\;i\in [m], j\in \NN\}$, consider the canonical $\frakS(\NN)$-action on the second index of $X_{ij}$ which extends canonically to a group action on $K[\xbf]$. An ideal $\frakP\subseteq K[\xbf]$ is called $\frakS (\NN)$-\emph{stable}, if $\sigma (\frakP) = \frakP$ for all $\sigma\in\frakS(\NN)$. Furthermore, a set $S\subseteq\in K[\xbf]$, such that the set of orbits $S/\frakS(\NN)$ is finite, is called $\frakS (\NN)$-\emph{finite}. The ideal $\frakP$ is called $\frakS (\NN)$-\emph{finitely generated} if it has a $\frakS(\NN)$-finite set of generators.
\end{Def}

We will make use of the following finiteness theorem from~\cite{HS} to the coordinate realization of said graph limit:

\begin{Thm}[name={\cite{HS}}]\label{Thm:HSfiniteness} %
Let $\xbf =\{X_{ij}\;:\;i\in [m], j\in \NN\}$, consider the polynomial ring $K[\xbf]$ in the infinitely many variables $\xbf$. Then, $K[\xbf]$ is a Noetherian $K[\xbf][\frakS(\NN)]$-module.\\

In particular, if $\frakP\subseteq K[\xbf]$ is a $\frakS (\NN)$-stable ideal, then, $\frakP$ is $\frakS(\NN)$-finitely generated.
\end{Thm}
\begin{proof}
This is a reformulation of Theorem~1.1 in~\cite{HA07}.
\end{proof}

\begin{Cor}\label{Cor:MHbasisfiniteness}
Under the assumptions of~\ref{Thm:HSfiniteness}, the ideal $\frakP$ has an $\frakS(\NN)$-finite MH-basis.
\end{Cor}
\begin{proof}
The $\frakS(\NN)$-action canonically extends to the ring $K[\xbf,\ybf]$ in which the multihomogenization is considered, compare Notation~\ref{Not:MHring}. In particular, the multihomogenization $\frakP^\vartriangle$ is again $\frakS (\NN)$-stable. Therefore, by Theorem~\ref{Thm:HSfiniteness}, there exists a $\frakS(\NN)$-finite set of generators $S$ of $\frakP^\vartriangle$. By Proposition~\ref{Prop:MHbasis}, the dehomogenization $S^\triangledown$ is a $\frakS(\NN)$-finite MH-basis of $\frakP$.
\end{proof}

The finiteness theorem can be directly derived from Corollary~\ref{Cor:MHbasisfiniteness}:

\begin{Thm}\label{Thm:topdegfiniteness}
Let $\matrex$ be an algebraic one-sided bipartite graph matroid limit, with coordinate realization $(\frakP,\xbf)$. Then, up to the graph symmetry, there set of inclusion-minimal top-degrees in $\topdeg(\frakP)$ is $\frakS(\matrex)$-finite. Differently phrased, there is only a finite number of minimal top-degrees in $\frakP$, up to the canonical graph symmetry.
\end{Thm}
\begin{proof}
By Corollary~\ref{Cor:MHbasisfiniteness}, the ideal $\frakP$ has a $\frakS(\NN)$-finite, therefore $\frakS(\matrex)$-finite MH-basis, which by Proposition~\ref{Prop:minMH} contains all the inclusion-minimal elements of $\topdeg(\frakP)$.
\end{proof}

\subsection{Multivariate Galois Theory}\
As the results in the previous sections concern algebraic matroids, they can be translated into statements about complexes of field extensions, in particular their minimal polynomials. Namely, the concept of circuit polynomial generalizes the concept of minimal polynomial for a finite set of field extensions. We briefly recapitulate the main result on the minimal polynomial:

\begin{Thm}\label{Thm:minpoly}
Let $K$ be a field and $\alpha$ algebraic\footnote{that is, there exists a polynomial $P\in K[X]$ such that $P(\alpha)=0$} over $K$. Write $L=K(\alphabf)$.

Then, there is an irreducible polynomial $m\in K[X]$, of degree\footnote{as usual, we denote by $[L:K]$ the dimension $\dim_K L$, where $L$ is considered as a $K$-vector space} $\deg m= [L:K]$, \\ satisfying $m(\alpha)=0$.

Moreover, such an $m$ is unique up to multiplication with a unit in $K^\times.$

Furthermore, for all $f\in K[X]$ such that $f(\alpha)=0$, it holds that $\deg f\ge \deg m$, and $m$ divides $f$.
\end{Thm}

Since $m$ is unique up to multiplication with a constant, one usually takes the minimal polynomial of $\alpha$ over $K$ to be the unique $m$ with leading coefficient $1$. Since this somewhat arbitrary convention hinders the analogy slightly, we make a surrogate definition:

\begin{Def}
Let $K$ be a field, $\alpha$ algebraic over $K$. Any $m\in K[X]$ of degree $\deg m= \dim_K L,$ satisfying $m(\alpha)=0$ is called an \emph{infimal polynomial} of $\alpha$ over $K$.
\end{Def}

The existence of a circuit polynomial, guaranteed by Lemma~\ref{Lem:circuitpoly}, implies an analogue for multiple field extensions, which is a strict generalization:

\begin{Thm}\label{Thm:multpoly}
Let $K$ be a field and $\alphabf:=\{\alpha_1,\dots, \alpha_n\}$ be a collection of elements over $K$. Write $L:=K(\alphabf)$ and $L_i:= K(\alphabf\setminus\{\alpha_i\})$. Assume that the $\alpha_i$ form an algebraic circuit; that is, $\trdeg L_i/K = n-1$ for all $i$, and $L/L_i$ is algebraic for all $i$.

Then, there is an irreducible\footnote{that is, irreducible over $K$} polynomial $\theta\in K[X_1,\dots, X_n]$, of top-degree $\topdeg \theta= \bigcup_{i=1}^n [L:L_i]\ast \{i\}$, such that $\theta(\alphabf)=0$.

Moreover, such a $\theta$ is unique up to multiplication with a unit in $K^\times.$

Furthermore, for all $f\in K[X_1,\dots, X_n]$ such that $f(\alphabf)=0$, it holds (component-wise) that $\topdeg f\ge \topdeg \theta$, and $\theta$ divides $f$.

Finally, for all $i$, the polynomial $\theta(\alpha_1,\dots, \alpha_{i-1},X_i,\alpha_{i+1},\dots, \alpha_n)\in K[X_i]$ is an infimal polynomial for $\alpha_i$ over $L_i$.
\end{Thm}
\begin{proof}
Existence of $\theta$, and uniqueness up to a multiplicative unit, follows from taking the coordinate realization cryptomorphic to the algebraic realization $\matralg_K (\alphabf)$, then applying Lemma~\ref{Lem:circuitpoly} to it, noticing that by assumption $[n]$ is a single circuit.

The fact that for all $f\in K[X_1,\dots, X_n]$ such that $f(\alpha)=0$ one has that $\theta$ divides $f$ follows from the fact that the ideal $K[X]\cap \frakP$ in the proof of Lemma~\ref{Lem:circuitpoly} is principal. That implies $\topdeg f\ge \topdeg \theta$ by Lemma~\ref{Lem:addmulthom}. Furthermore, it implies that $\theta$ is irreducible.

For the remaining statements, we note that $\theta_i:=\theta \left( \alpha_1,\dots,\alpha_{i-1},X_i,\alpha_{i+1},\dots, \alpha_n \right)$ is an infimal polynomial by Theorem~\ref{Thm:minpoly}. Therefore, for any $i$, it holds that $\theta_i = m_i\cdot h_i$ for some infimal polynomial $m_i$ of $\alpha_i$ over $L_i$, and some $h_i\in L_i[X_i]$. Since the infimal polynomial $m_i$ is unique up to  multiplication in $L_i$, we can clear denominators and assume that $h_i,m_i\in K[\alpha\setminus\{\alpha_i\}][X_i]$. Since by assumption $\trdeg L_i/K = n-1$, we can replace $\alpha_j$ by the corresponding variables $X_j$, to obtain polynomials $\overline{h}_i,\overline{m}_i\in K[X_1,\dots, X_n]$ such that $\overline{h}_i(\alphabf)\cdot \overline{m}_i (\alphabf)=0$. Since $\theta_i = m_i\cdot h_i$ by assumption, $\theta$ must therefore divide $\overline{h}_i\cdot \overline{m}_i$. Since $\theta = \overline{m}_i\cdot \overline{h}_i$ by assumption, $\theta$ must be equal to exactly one of $\overline{h}_i,\overline{m}_i$. Since $h_i$ was arbitrary, and therefore $\overline{h}_i$ also is, it follows that $\theta = \overline{m}_i$, implying that $\theta_i$ is an infimal polynomial for $\alpha_i$ over $L_i$
and therefore $\deg_{X_i} \theta = [L:L_i]$, the latter implying the statement about the top-degree by Lemma~\ref{Lem:deg-topdeg}.
\end{proof}

First note that Theorem~\ref{Thm:minpoly} is implied by Theorem~\ref{Thm:multpoly}, by taking $n=1$. On the other hand, Theorem~\ref{Thm:multpoly} is strictly stronger, since it states that the infimal polynomials all uniquely lift to the circuit polynomials - up to multiplicative constant, which is now in $K^\times$ instead of $L_i^\times$. That is a statement which cannot be inferred from standard Galois theory.

Furthermore, the assumption in Theorem~\ref{Thm:multpoly} that the $\alpha_i$ form an algebraic circuit is not a huge restriction, since from matroid theory, in a set of $\alpha_i$ which is algebraically dependent, one can always pick subsets that are algebraic circuits. However, it is a restriction in the sense that a similar lifting will not occur in general if there are more than one way to pick a subset which is an algebraic circuit, as the following example illustrates:

\begin{Ex}[name ={8.15 in \cite{Stu02}}]
Let $I \subseteq \CC[a,b,c,d]$ be given by:
\[ I = \langle ad - bc, ac^4 - b^3 d^2, a^3c^2 - b^5,b^2d^3 - c^5,a^2c^3 -b^4 d\rangle.\]
The first generator has minimal top-degree $(1,1,1,1)$ despite not being a circuit.
\end{Ex}

\section{Algebraic and Combinatorial Structure Theorems}\label{sec:structure}
In this section, we will consider matroids which are both (bipartite) graph matroids, and algebraic. Bringing together the concept of circuit polynomial for algebraic matroids, and the symmetry statments for graph matroids, we obtain several structural statements on the circuit polynomials, and inductive relations between circuits.

\subsection{Symmetries of Circuit Polynomials}
\label{sec:structure.symc}
In this section we analyze how symmetries of a circuit translate into symmetries of the corresponding circuit polynomial. Namely, uniqueness of the circuit polynomial implies that the automorphism group $\frakS(S)$ acts on the unit group of the ground field $K$:

\begin{Prop}\label{Prop:circpolysym}
Let $\matrex$ be a bipartite graph matroid, having a coordinate realization $\matrex = \matrcrd (\frakP, \xbf)$ over the ground field $K$. Let $C$ be a circuit of $\matrex$, and let $\theta_C\in K[\xbf(C)]$ be the corresponding circuit polynomial. Then, the map
$$\frakS(C) \rightarrow K^\times,\quad \sigma \mapsto \frac{\sigma \left(\theta_C\right)}{\theta_C}$$
is a group homomorphism. Furthermore, there are a unique minimal $n\in\NN$, and a group homomorphism $\nu: \frakS(C)\rightarrow \ZZ /n \ZZ$, such that
$$\sigma \theta_C = \zeta_n^{\nu (\sigma)}\theta_C,$$
where $\zeta_n\in \CC$ is an $n$-th root of unity. Moreover, if $\chr K = 0$, there exists a set of monomials $S\subset K[\xbf(C)]$, such that
$$\theta_C = \sum_{\sigma\in \frakS(C)} \zeta_n^{\nu(\sigma)}  \sum_{M\in S} \sigma(M),$$
which is unique up to $\sigma$-action on each element of $S$ if $S$ is chosen with minimal cardinality.\\
\end{Prop}
\begin{proof}
Let $\sigma\in \frakS (C)$ arbitrary. By definition of $\frakS(C)$, it holds that $\sigma (C) = C$, so by construction, the polynomial $\sigma \theta_C$ is a circuit polynomial of $C$. Since circuit polynomials are unique up to a multiplicative constant in $K^\times$, as it follows from Lemma~\ref{Lem:circuitpoly}, this gives rise to the claimed group homomorphism $\varphi: \frakS(C)\rightarrow K^\times$. Since $\frakS(C)$ is finite, the image $\varphi(\frakS(C))$ must be finite as well, therefore contained in the some multiplicative group generated by $\zeta_n$ for some $n\in\NN$, which can be chosen uniquely minimal. The remaining considerations follow from substitutions.
\end{proof}

\begin{Cor}\label{Cor:topdegsym}
In the situation of Proposition~\ref{Prop:circpolysym}, $\topdeg \theta_C$ is invariant under the action of $\frakS(C)$. That is, the top-degree mask of $C$ is invariant under the action of $\frakS(C)$.
\end{Cor}
\begin{proof}
This follows from the last statement in Proposition~\ref{Prop:circpolysym}.
\end{proof}
Therefore, top-degrees can be associated with the corresponding edges of the bipartite graph.

The circuit polynomial associated with rectangular circuits is particularly symmetric:
\begin{Prop}\label{Prop:elcircstr}
Let $C=[k]\times [\ell]$ be a circuit of an algebraic bipartite graph matroid $\matrex = \matrcrd (\frakP, \xbf)$ over the ground field $K$. Let
$\theta_C\in K[\xbf(C)]$ be the corresponding circuit polynomial, and consider $\frakS(C) = \frakS (k) \times \frakS (\ell)$. Then, there are two cases:\\

{\bf Case 1:} If $k,\ell\neq 4$, then there is $n\in\NN$, and a group homomorphism $\nu: \frakS(C)\rightarrow \ZZ /2 \ZZ$, such that
$$\sigma \theta_\xbf = (-1)^{\nu (\sigma)}\theta_\xbf = \gamma (\sigma) \theta_\xbf,$$
and $\gamma (\sigma_1,\sigma_2) = \sgn (\sigma_1)^a\cdot \sgn (\sigma_2)^b,$ with $a,b\in \{0,1\}$ and $\sgn:\frakS (\cdot)\rightarrow \{-1,1\}$ the usual sign/parity function.\\
Moreover, if $\chr K = 0$, there exists a set of monomials $S\subset K[\xbf(C)]$, such that
$$\theta_C = \sum_{\sigma\in \frakS(C)} (-1)^{\nu(\sigma)}  \sum_{M\in S} \sigma(M)= \sum_{\sigma\in \frakS(C)} \gamma (\sigma)  \sum_{M\in S} \sigma(M),$$
which is unique up to $\sigma$-action on each element of $S$ if $S$ is chosen with minimal cardinality.\\

{\bf Case 2:} If one of $k,\ell$ is equal to $4$, there is $n\in\NN$, and a group homomorphism $\nu: G\rightarrow \ZZ /6 \ZZ$, such that
$$\sigma \theta_C = \zeta_6^{\nu (\sigma)}\theta_\xbf= \gamma (\sigma) \theta_\xbf,$$
where $\zeta_6\in \CC$ is a sixth root of unity, and $\gamma (\sigma_1,\sigma_2) = f (\sigma_1)^a\cdot g (\sigma_2)^b,$ with $a,b\in \{0,1,2\}$, and each of $f,g$ is either the sign/parity function $\frakS (\cdot)\rightarrow \{-1,1\}$, or the triparity function $\frakS(4)\rightarrow \{1,\zeta_3,\zeta_3^2\}$. Moreover, if $\chr K = 0$,  there exists a set of monomials $S\subset K[\xbf(C)]$, such that
$$\theta_C = \sum_{\sigma\in \frakS(C)} \zeta_6^{\nu(\sigma)}  \sum_{M\in S} \sigma(M) = \sum_{\sigma\in \frakS(C)} \gamma (\sigma)  \sum_{M\in S} \sigma(M),$$
which is unique up to $\sigma$-action on each element of $S$ if $S$ is chosen with minimal cardinality.
\end{Prop}
\begin{proof}
Existence of a homomorphism $\varphi:\frakS(C) \rightarrow K^\times$ follows from Proposition~\ref{Prop:circpolysym}, therefore the image $\varphi(\frakS(C))$ must be a quotient $\frakS(k)/M\times \frakS(\ell)/N$ by normal divisors $M$ and $N$. If $k$ is not divisible by $4$, then $M$ is either the identity group, or the alternating group, if $k$ is divisible by $4$, then $M$ can additionally be the Klein four group; the analogue statement holds for $\ell$ and $n$. Therefore, the image $\varphi (\frakS(C))= \varphi(\frakS (k)\times \frakS (\ell))$ must be isomorphic to a quotient of $(\ZZ /2 \ZZ)^2$ if $k,\ell\neq 4,$ to one of the groups $(\ZZ /2 \ZZ)^2, \ZZ /2 \ZZ\times \ZZ /3 \ZZ$ if exactly one of $k,\ell$ is $4$, and to one of the groups $(\ZZ /2 \ZZ)^2, \ZZ /2 \ZZ\times \ZZ /3 \ZZ, \ZZ /3 \ZZ^2$ if $k=\ell = 4$. Since the image must also be a finite subgroup of $K^\times$, therefore a subgroup of the cyclotomic subgroup of $K^\times,$ it can only be isomorphic to a subgroup of $\ZZ /2 \ZZ$ if $k,\ell\neq 4$, and to a subgroup of $\ZZ /6 \ZZ$ otherwise.
\end{proof}

\begin{Cor}
As in Proposition~\ref{Prop:elcircstr}, let $C=[k]\times [\ell]$ be a rectangular circuit. Then, $\topdeg \theta_C$ is the multiset $d\ast C$ for some $d\in\NN$. Phrased differently, the non-zero entries in the top-degree-mask of $C$ are all equal.
\end{Cor}
\begin{proof}
This follows from Corollary~\ref{Cor:topdegsym} and the fact that for each $c_1,c_2\in C$ there is $\sigma\in\frakS(C)=\frakS(k)\times \frakS(\ell)$ such that $\sigma c_1 = c_2$.
\end{proof}

\begin{Rem}
In the situation of Proposition~\ref{Prop:elcircstr}, one obtains a determinantal or permanental formula
if $(\sigma,\sigma)(M)=M$ for all $M\in S$ and $\sigma\in\frakS(m)$, since then
\[
\sum_{\sigma\in G} (-1)^{\nu(\sigma)}  \sum_{M\in S} \sigma(M) =
\sum_{\sigma\in\frakS (n)} (\sgn \sigma)^{a+b}  \sum_{M\in S} \sigma(M)
\]
\end{Rem}

\subsection{Constructing circuits with the $(t,1)$-move}\label{Sec:21move}
In this section, we study how classes of circuits that can be generated by the following inductive move:
\begin{Def}\label{Def:t1move}
Let $B$ be a bipartite graph including the edge $ij$ and let $t\in \NN$ be a
parameter with $t\le r$. The $(t,1)$-{\it move} transforms $B$ as follows:

Remove the edge $ij$. Add a copy of $K_{1,t}$ on new vertices, and connect each new vertex to $B$ so that:
\begin{enumerate} \compresslist
\item The graph remains bipartite.
\item Each new vertex has degree $r+1$.
\item There is a pair of vertices $m,n$ such that $(mj)$ and $(in)$ are edges, and the sets of neighbors of $m$ and $n$ are properly contained in the sets of neighbors of $i$ and $j$ respectively.
\end{enumerate}

We define a {\it partial $(t,1)$-move} as a typical $(t,1)$-move, without the removal of the edge $ij$.

Edges and vertices contained in a graph before performing a move will be called \emph{old}, while edges
added in a move will be called \emph{new}.
\end{Def}
The main result of this section is:
\begin{Prop}\label{Prop:t1move}
Let $C$ be a circuit graph in $\detM(m\times n,r)$ that is spanned by a basis graph.
Then applying a $(t,1)$-move over any edge of $C$ produces a circuit graph in
$\detM((m+1)\times (n+t),r)$ that is spanned by a basis graph.
\end{Prop}
We defer the proof until we have established some required intermediate results.
\begin{Lem}\label{Lem:partialt1}
Let $B$ be a basis graph in $\vec D(m\times n,r)$ and let
$C$ be the graph obtained by performing the partial $(t,1)$-move on $B$. Then,
$C$ contains a unique circuit graph $D$.  Additionally, the graph $D$ contains
all the new edges.
\end{Lem}
\begin{proof}
$C$ must contain some circuit, since a partial $(t,1)$-move adds $r + rt + 1$ edges while the rank
rises only $r + rt$; therefore, we have defect $-1$.

Any circuit in $C$ must contain a new edge, since the set of old edges is independent. However, if any new
edge is included, {\it all} new edges must be included in order to ensure that every new vertex has
degree $\geq r+1$.  Otherwise, this circuit would contradict Theorem \ref{Thm:avgrk} (iv).  %

If two distinct circuits were contained in $C$, both would contain all new edges. Via the circuit elimination axiom, we
could then eliminate a new edge, finding a new circuit in $C$ without all the new edges, leading us to a contradiction.
\end{proof}

\begin{Lem}\label{Lem:t1basis}
Let $B$ be a basis graph of $\detM(m\times n,r)$, and let $B'$ be the graph resulting from a $(t,1)$-move on $B$.
Then, $B'$ is a basis graph of $\detM((m+1) \times (n + t), r)$.
\end{Lem}
\begin{proof}
A partial $(t,1)$-move produces a unique circuit $D$, by Lemma \ref{Lem:partialt1}. Let $P_D(X)$ denote
the associated circuit polynomial.

If the edge $ij$ is in the circuit, then removing it as part of the $(t,1)$-move will result in an independent set,
which will be a basis graph by cardinality.

Suppose the edge $ij$ is not in $D$. We refer to the parametric definition of the determinantal matroid in Definition
\ref{Def:detmatroid} (a).  In this setting, $X_{ij} = \sum_{k =1}^r U_{ik}V_{jk}$ is excluded from the circuit.

The $(t,1)$-move introduces a new set of transcendentals, call them $\{ \bf U_m, V_n, \ldots, V_{n+t- 1}\}$. Assume $m,n$ are the vertices stipulated by Definition \ref{Def:t1move}.3.
Specialize $\bf U_m = U_i$ and $\bf V_n = V_j$; any polynomial identities will also be true when specialized to a particular value.
Now one of the new edges, known to be in the circuit, takes the same value as $X_{ij}$. Therefore, $P_D(X)$ can be
considered to include $X_{ij}$ and exclude $X_{mn}$.
Because the neighbors of $m,n$ are also neighbors of $i,j$ respectively, the variables in $P_D(X)$ are contained in $B$.
This gives  us a circuit excluding a new edge, a contradiction.
\end{proof}

\begin{proof}[Proof of Proposition \ref{Prop:t1move}]
Let $D$ denote the new graph. We need to show that: \begin{enumerate} \compresslist
\item Defect is preserved.
\item $D$ is spanned by a basis graph.
\item $D$ contains exactly one circuit $C'$.
\item All edges of $D$ are contained in $C'$
\end{enumerate}

For (1), the change in rank is $r + rt$, while the move adds a net $(r+1)t  + (r + 1 - t) - 1= r + rt$ edges;
therefore, defect is preserved.

For (2), $ B = C \setminus ij$ is a basis graph of $\detM(m\times n,r)$. By Lemma \ref{Lem:t1basis}, the $(t,1)$-move
produces a basis graph from $B$; therefore, $D$ is spanned by a basis graph of $\detM((m+1)\times (n+t ),r)$.

For (3), we reason as in the proof of Lemma \ref{Lem:partialt1}.  First, all new edges must be in any circuit
contained in $D$.  If there is more than one such circuit, eliminating an edge between them produces
a circuit supported on a proper subgraph of $C$, which is a contradiction.

Finally, for (4), we have already shown that $D$ contains all the new edges.
We suppose by contradiction that some old edge $k\ell$ is excluded from $C'$. Then $C \setminus k\ell$ would be a basis
of $\detM(m\times n,r)$. By Lemma \ref{Lem:t1basis}, the $(t,1)$-move constructs a basis, in contradiction
with the hypothesis that $D \setminus k\ell$ contains the circuit.
\end{proof}

\begin{Rem}
One may attempt to generalize the $(t,1)$-move to a $(s,t)$-move adding a copy of $K_{s,t}$ and connecting to the circuit graph giving each new vertex degree $r+1$. However, for $s, t > 1$, the rank changes by $r(s + t)$ while a net $(r + 1)(s + t) - st - 1$ edges are added; so the defect changes by $s + t - st -1$. For $s,t > 1$, the defect is not preserved as $s + t - st -1 $ is strictly negative. An example of the failure is given below:

\begin{figure}[h]
\centerline{\edgeSplitFail  }
\end{figure}

While the upper left $3\times 3$ minor is a circuit of $\detM (3 \times 3, 2)$, the constructed graph is a {\it basis} of $\detM (5 \times 5,2)$.
\end{Rem}

\begin{Rem}
Observe that starting with an elementary circuit and applying the $(r,1)$-move $k$ times creates a circuit of signature
$(r + k, rk + 1)$. For $k = m-r$, this is equal to $(m, r(m-r)+1)$, which is precisely the bound given in Proposition~\ref{Prop:circuitsizedet}, proving it to be strict for the case of the determinantal matroid.
\end{Rem}

\section{A Dictionary of Circuits}\label{sec:examples}

In this section, we aim to describe the set of circuits of algebraic graph matroids in as much detail as possible. In particular, the different circuit graphs and their symmetries, and the number of circuits in the algebraic matroid corresponding to each graph. Even in cases where a full combinatorial understanding is inaccessible, we will be interested in the total number of circuits.

\subsection{Symmetries and Number of Circuits for Bipartite Graph Matroids}\label{sec:examples.bipsym}
We introduce some notation for counting circuits, for the two different ways of counting - as sets, and as graphs.

\begin{Not}
Let $\matrex$ be a bipartite graph matroid limit with canonical symmetric group $G\cong \frakS(\NN\times \NN)$. We denote by
$$\calC_{k,\ell}(\matrex):=\{C/G\;:\; C\in\calC(\matrex), \#\vsupp C = (k,\ell)\}$$
the set of (non-isomorphic) circuit graphs of $\matrex$ with signature $(k,\ell)$. We will denote its cardinality by $\circnum_{k,\ell}(\matrex):=\# \calC_{k,\ell}(\matrex).$ Where clear from the context, the dependence on $\matrex$ will be omitted.
\end{Not}

\begin{Prop}\label{Prop:countformula}
Let $\matrex_{\mu\nu}$ be an injective complex of bipartite graph matroids, with bipartite graph matroid limit $\matrex$, having symmetric group $G\cong \frakS(\NN\times \NN)$. Then, the following formulae hold:
\begin{description}
\item[(i)] \[ \# \calC ( \matrex_{m,n}/G) = \sum_{k=1}^m \sum_{\ell=1}^n  \circnum_{k,\ell}(\matrex).\]
\item[(ii)] \[ \# \calC ( \matrex_{mn}) = \sum_{k=1}^m \sum_{\ell=1}^n \left[ \sum_{C \in \calC_{k,\ell}(\matrex)} \frac{1}{\# \frakS (C)} \right] \frac{m!\cdot n!}{(m-k)!(n-\ell)!}. \]
\end{description}
\end{Prop}
\begin{proof}
(i) follows from the fact that $\calC ( \matrex_{m,n}/G)$ is, by definition, a disjoint union of all $\calC_{k,\ell} ( \matrex)$ with $k\le m,\ell\le n$.\\
(ii) Pick a circuit $C\in  \calC ( \matrex_{m,n})$ with signature $(k,\ell)$, let $G_{mn}$ be the canonical symmetry group of $\matrex_{mn}$. Consider the stabilizer subgroup $\Stab (C)$ of $G_{mn}$. By definition of support, $\sigma \vsupp (C)\in \vsupp (C)$ for all $\sigma\in \Stab (C)$. Therefore, $\Stab (C) = S_1\times S_2$, with $S_2 \cong \frakS ([m-k]\times [n-\ell)$, and $S_1$ is a subgroup of $\frakS (\vsupp (C))$ isomorphic to the automorphism group $\frakS (C/G)$. In particular, $\# \Stab (C) = \# S_1\cdot \# S_2 = \# \frakS(C/G)\cdot (m-k)!(n-\ell)!$ Applying the orbit-stabilizer theorem to the $G_{mn}$-action, we obtain
$$\# ( C/G_{mn} ) = \frac{\# G_{mn}}{\#\Stab(C)} = \frac{1}{\frakS(C/G)} \cdot \frac{m!n!}{(m-k)!(n-\ell)!}. $$
Noting that $\calC_{k, \ell}(\matrex)$ contains exactly orbits $C/G_{mn}$ of signature $(k,\ell)$, a summation over all circuit graphs and all signatures yields the formula.
\end{proof}

\begin{Rem}
From the theory in the previous sections, it follows that many of the $\calC_{k,\ell}(\matrex)$ are empty, therefore the summations in Proposition~\ref{Prop:countformula} can be restricted. Namely:
\begin{description}
\item[(i)] By Theorem~\ref{Thm:avgrk}~(ii), there are no circuits of signature $(k,\ell)$ with $k \leq \ark_1$ or $\ell\leq \ark_2$.
\item[(ii)] By Theorem~\ref{Thm:circuitsizemn},
\begin{align*}
k&\le \crk - \ark_1\cdot \rsize_1 + \ark_2\cdot (\ell -\rsize_2) + 1\\
\ell &\le \crk - \ark_2\cdot \rsize_2 + \ark_1\cdot (k -\rsize_1) + 1.
\end{align*}
\end{description}
In our examples, we will often be able to put tighter bounds on the summation.
\end{Rem}

Furthermore, the summation of Proposition~\ref{Prop:countformula}~(ii) can be seen as in analogy to Proposition~\ref{Prop:countformula}~(i). The combinatorics is completely captured in the coefficient in square brackets, which can be interpreted as the average fraction a circuit graph contributes to the number of circuits of a completely unsymmetric graph. Therefore, we introduce a notational abbreviation:

\begin{Not}
Let $\matrex$ be a bipartite graph matroid limit. We will write
$$\stabavg_{k,\ell}(\matrex):=\sum_{C \in \calC_{k, \ell}(\matrex)} \frac{1}{\# \frakS (C)}.$$
Where clear from the context, the dependence on $\matrex$ will be omitted.
\end{Not}

\begin{Rem}
With this notation, formula~(ii) of Proposition~\ref{Prop:countformula} becomes
$$\# \calC ( \matrex_{mn}) = \sum_{k=1}^m \sum_{\ell=1}^n \stabavg_{k,\ell}\cdot \frac{m!\cdot n!}{(m-k)!(n-\ell)!},$$
and therefore the graphless analogue of formula~(i), since the factor $\frac{m!\cdot n!}{(m-k)!(n-\ell)!}$ can be interpreted as coming from the lost graph symmetry.
\end{Rem}

\subsection{The Determinantal Matroid}
We determine some circuits and their associated invariants of the determinantal matroids. For rank one, we can completely characterize those:

\begin{Thm}\label{Prop:r1circCount}
Consider the bipartite graph matroid $\matrex=\detM (\NN \times \NN,1)$. Denote by $C_{2n},n\in \NN$ the (bipartite) cycle of length $2n$, containing $n$ vertices of each of the two classes. Then:
\begin{description}
\item[(i)] $\calC (\matrex) = \{C_{2n}\;:\;n\in\NN\}$
\item[(ii)] $\calC_{k,\ell} (\matrex) = \varnothing$ if $k\neq \ell$
\item[(iii)] $\calC_{k,k} (\matrex) = \{C_{2k}\}$
\item[(iv)] $\circnum_{k,\ell}(\matrex) = \delta_{k,\ell}$
\item[(v)] $\stabavg_{k,\ell}(\matrex) = \delta_{k,\ell}\cdot(4k)^{-1}$
\end{description}
The top-degree is all ones for the variables in the circuit support.
\end{Thm}
\begin{proof}
(i) follows from Proposition~2.6.45~(i) in~\cite{KTTU12}.\\
(ii) and (iii) follow from the fact that the signature of $C_{2k}$ is $(k,k)$.\\
(iv) follows from (ii) and (iii).\\
(v) follows from the fact that $\frakS(C_{2k})$ is isomorphic to the dihedral group $D_{2k}$, whose cardinality is $4k$.
\end{proof}

\begin{Prop}[Rank 2]
The list below contains {\it all} circuit graphs with signature $k \leq l \leq 5$, along with $\frakS(X)$. Some circuits are the transpose of circuits in this list; this is indicated in the chart. As particular consequences,
\[ \#\calC(\detM(3 \times n, 2)) = \binom{n}{3},  \hspace{6mm} \#\calC(\detM(4 \times n, 2)) = 4\binom{n}{3} + 96 \binom{n}{4} + 840 \binom{n}{5}, \hspace{6mm} \#\calC(\detM(5 \times 5, 2)) =  65, \: 650. \]

\pagebreak
\begin{footnotesize}
\begin{longtable}{| c | c  c | c  c |}
\hline
{\bf Signature} 			& {\bf Mask}				&										& & \\ \hline
\multirow{2}{*}{$(3,3)$} 	& \multirow{2}{*}{\nmatrixa} 	& $\frakS(X) = \frakS(3) \times \frakS(3)$   				& & \\[3mm]
&						& $\deg(\theta_X) = 3$					 	& & \\[5mm] \hline

\multirow{2}{*}{$(4,4)$} 	& \multirow{2}{*}{\nmatrixb} 	& $\frakS(X) = \frakS(3)$ 			  				& & \\[3mm]
&						& $\deg(\theta_X) = 5$					 	& & \\[8mm] \hline

\multirow{3}{*}{$(4,5)$} 	& \multirow{3}{*}{\nmatrixc} 	& $\frakS(X) = (\ZZ/2\ZZ)^3$
& \multirow{3}{*}{\nmatrixd}	& $\frakS(X) = \ZZ/2\ZZ \times \frakS(3)$			 		\\[3mm]
&						& $\deg(\theta_X) = 5$
&						& $\deg(\theta_X) = 7$							\\[3mm]
&						& \text{\bf *and transpose}
&						& \text{\bf *and transpose}							\\[3mm] \hline
\multirow{10}{*}{$(5,5)$}	& \multirow{2}{*}{\nmatrixk} 	& $\frakS(X) = (\ZZ/2\ZZ)^4$
& \multirow{2}{*}{\nmatrixe}	& $\frakS(X) = (\ZZ/2\ZZ)^5 $			 			\\[3mm]
&						& $\deg(\theta_X) = 5$
&						& $\deg(\theta_X) = 5$						\\[14mm] \cline{2-5}

& \multirow{3}{*}{\nmatrixf} 	& $\frakS(X) = \ZZ/4\ZZ \times \ZZ/2\ZZ$
& \multirow{3}{*}{\nmatrixl}	& $\frakS(X) = \ZZ / 3\ZZ \times (\ZZ/2\ZZ)^2 $		\\[3mm]
&						& $\deg(\theta_X) = 7$
&						& $\deg(\theta_X) = 7$						\\[3mm]
&						&
&						& \text{\bf *and transpose}							\\[6mm] \cline{2-5}								& \multirow{3}{*}{\nmatrixg} 	& $\frakS(X) = \ZZ/2\ZZ$
& \multirow{3}{*}{\nmatrixj}	& $\frakS(X) = (\ZZ/2\ZZ)^2$		 			\\[3mm]
&						& $\deg(\theta_X) = 7$
&						& $\deg(\theta_X) = 7$						\\[3mm]
&						& \text{\bf *and transpose}
&						& \text{\bf *and transpose}								\\[6mm] \cline{2-5}
& \multirow{2}{*}{\nmatrixm} 	& $\frakS(X) = 1$
& \multirow{2}{*}{\nmatrixh}	& $\frakS(X) = 1$			 			\\[3mm]
&						& $\deg(\theta_X) = 7$
&						& $\deg(\theta_X) = 9$						\\[14mm] \cline{2-5}

& \multirow{2}{*}{\nmatrixi} 	& $\frakS(X) = \ZZ/3\ZZ \times (\ZZ/2\ZZ)^2$ 		& & \\[3mm]
&						& $\deg(\theta_X) = 12$						& & \\[14mm] \hline

\caption{Circuits of the matroid $\detM (5 \times 5,2)$. *Degree of the circuit polynomial is included for reference.}
\label{55dettable}
\end{longtable}
\end{footnotesize}

\end{Prop}

\begin{proof}
For $k = 3$, the maximal minor is the only circuit by Theorem \ref{Thm:deglb}. As for $k = 4$, the same argument limits us to $4$ or $5$ columns with one element in each column, after which an explicit check yields the circuits above.

Let $X^C$ be the complement of a circuit $X$ in $K_{m,n}$. By Theorem \ref{Thm:deglb}, all vertices of $X$ have degree at least 3, so the vertices of $X^C$ have degree at most $2$; therefore, $X^C$ must be a union of paths and cycles. Rank of the matroid implies that $X$ can contain at most $r(m+n - r) + 1 = 17$ elements, which means $X^C$ has at least $8$ edges. All that remains is to list all graphs $X^C$ fitting these criteria and manually check whether $X$ is a circuit using Macaulay2 \cite{M2}.

To compute $\frakS(X)$, we compute the automorphism group of the complement, which has factors: \begin{enumerate} \compresslist
\item $(\ZZ/k\ZZ \times \ZZ/2\ZZ)$ for every $2k$-cycle,
\item $(\ZZ/2\ZZ)$ for every even path.
\item $\frakS(k)$ for each set of $k$ identical components (identical, including the same coloring of vertices).
\end{enumerate}
We note that the graphs listed have transpose symmetry if and only if the number of $2k$-paths with row vertex endpoints is the same as the number of $2k$-paths with column vertex endpoints, for all $k$. The graphs marked with  {\bf *and transpose} are those where (i) the signature is asymmetric, or (ii) the graph fails the condition on $2k$-paths.
\end{proof}

We can use similar techniques to compute the circuits for a few small signatures with $r=3$.

\begin{Prop}[Rank 3] \label{r3det} The list below contains {\it all} circuit graphs $X$ with signature $k \leq 5$, along with $\frakS(X)$. As particular consequences,
\[ \#\calC(\detM(4 \times n, 3)) = \binom{n}{4},  \hspace{3mm} \#\calC(\detM(5 \times n, 3)) = 5 \binom{n}{4} + 600 \binom{n}{5} + 13,\: 320 \binom{n}{6} + 65, \: 100 \binom{n}{7}. \]

\begin{footnotesize}
\begin{longtable}{| c | c  c | c  c |}
\hline
{\bf Signature} 			& {\bf Mask}				&										& & \\ \hline
\multirow{2}{*}{$(4,4)$} 	& \multirow{2}{*}{\fmatrixa} 	& $\frakS(X) = \frakS(4) \times \frakS(4)$				   	& & \\[3mm]
&						& $\deg(\theta_X) = 4$					 	& & \\[8mm] \hline

\multirow{2}{*}{$(5,5)$} 	& \multirow{2}{*}{\fmatrixb} 	& $\frakS(X) = \frakS(3) \times (\ZZ/2\ZZ)^3$ 			& & \\[3mm]
&						& $\deg(\theta_X) = 7$					 	& & \\[14mm] \hline
\multirow{2}{*}{$(5,6)$} 	& \multirow{2}{*}{\fmatrixc} 	& $\frakS(X) = (\ZZ/2\ZZ)^4$
& \multirow{2}{*}{\fmatrixd}	& $\frakS(X) = \ZZ/2\ZZ \times \frakS(3)$			 	\\[3mm]
&						& $\deg(\theta_X) = 7$
&						& $\deg(\theta_X) = 10$						\\[14mm] 	\cline{2-5}

& \multirow{2}{*}{\fmatrixe} 	& $\frakS(X) = \frakS(5)$
& 						& 							 			\\[3mm]
&						& $\deg(\theta_X) = 15$
&						& 										\\[14mm] \hline
\multirow{8}{*}{$(5,7)$}	& \multirow{2}{*}{\fmatrixf} 	& $\frakS(X) = \frakS(3) \times (\ZZ/2\ZZ)^4 $
& \multirow{2}{*}{\fmatrixg}	& $\frakS(X) = (\frakS(3))^2 \times (\ZZ/2\ZZ)^2 $		\\[3mm]
&						& $\deg(\theta_X) = 7$
&						& $\deg(\theta_X) = 7$						\\[14mm] \cline{2-5}

& \multirow{2}{*}{\fmatrixh} 	& $\frakS(X) = \frakS(3) \times (\ZZ/2\ZZ)^2$
& \multirow{2}{*}{\fmatrixi}		& $\frakS(X) = \frakS(4) \times \frakS(3)$		 			\\[3mm]
&						& $\deg(\theta_X) = 10$
&						& $\deg(\theta_X) = 13$						\\[14mm]  \hline
& \multirow{2}{*}{\fmatrixj} 	& $\frakS(X) = \frakS(3) \times (\ZZ/2\ZZ)^3$
& \multirow{2}{*}{\fmatrixk}	& $\frakS(X) = \frakS(3) \times (\ZZ/2\ZZ)^3$			 \\[3mm]
&						& $\deg(\theta_X) = 13$
&						& $\deg(\theta_X) = 18$						\\[14mm] \hline

\caption{Circuits of the matroid $\detM (5 \times n, 3)$. *Degree of the circuit polynomial is included for reference.}
\label{r3ctab}
\end{longtable}
\end{footnotesize}
\end{Prop}

\begin{proof}
Again, we let $X^C$ be the complement of a circuit $X$ in $K_{5,n}$. All vertices of $X$ have degree at least 4, so the column vertices of $X^C$ have degree at most $1$. This means that $X^C$ is a union of star graphs, with central vertex a row vertex, and isolated vertices. Rank of the matroid implies that $X$ can contain at most $r(m+n - r) + 1 = 3n + 7$ elements, which means $X^C$ has at least $2n - 7$ edges. Checking suitable graphs amounts to checking partitions of at most $7$, with at most $5$ nonzero parts. The graphs that are verified by Macaulay2 \cite{M2} as circuits are listed in the table.

For $\frakS(X)$, we again compute $\frakS(X^C)$. As mentioned above, the unions of star graphs map canonically to partitions, by sending each star to the number of its edges or leaves. Let $\lambda = \lambda_1^{a_1} \ldots \lambda_n^{a_n}$ be the partition corresponding to $X^C$. The stabilizer of $X^C$ has factors:
\begin{enumerate} \compresslist
\item $\prod_{i = 1}^n (\frakS({\lambda_i}))^{a_i} \times \frakS({a_i})$, permuting the leaves in each star and permuting the stars of the same size.
\item $\frakS(k) \times \frakS(l)$, where $k$ is the number of isolated row vertices of $X^C$ and $l$ is the number of isolated column vertices.
\end{enumerate}
\end{proof}

\vspace{-5mm}

\begin{Rem}
We record $c_{k,l}$ and $\beta_{k,l}$ computed for the matroid $\detM(\NN \times \NN,r)$, for some $r$.

For $r = 1$: $c_{k,l} = \delta_{k,l}$ and $\beta_{k,l}= \delta_{k,l}/4k$.

For $r = 2$ and $3$, the following tables list the first few values:
\[ r = 2:
\begin{array}{|c|c|c|} \hline
(k,l) & c_{k,l} & \beta_{k,l}  \\ \hline
(3,3)  & 1   & 1/36  \\
(4,4)  & 1  & 1/6  \\
(4,5)  & 2  & 5/24 \\
(5,4)  & 2  & 5/24 \\
(5,5)  & 12  & 127/32 \\ \hline
\end{array} \hspace{1cm}
r = 3: \begin{array}{|c|c|c|} \hline
(k,l) & c_{k,l}   & \beta_{k,l}  \\ \hline
(4,4)  & 1   & 1/576  \\
(5,5)  & 1  & 1/24  \\
(5,6)  & 3  & 37/240 \\
(5,7)  & 6  & 31/288 \\ \hline
\end{array}
\]

For all $r$: $c_{r+1,r+1} = 1$ and $\beta_{r+1,r+1} = 1/(r + 1)!^2$

\hspace{1cm} $c_{r+2,r+2} = 1$ and $\beta_{r+2,r+2} = 1/(3!\cdot (r-1)!^2)$.
\end{Rem}

\subsection{The Bipartite Rigidity Matroid}

\begin{Prop}
Let $\matrex(n) = \CMM(m \times n, r)$. The formula from Proposition \ref{Prop:countformula} counts circuits of $\matrex(n)$ with the tighter summation bounds of:
\[  r+1 \leq k \leq m; \hspace{1cm} \frac{1}{r}\left(\binom{r+1}{2} +k -1 \right) \leq l \leq rk- \binom{r+1}{2} +1. \]
\label{Prop:CMMcount}
\end{Prop} \vspace{-1cm}
\begin{proof}
The lower bound on $k$ -- $ \CMM(m \times n, r)$ has $\ark = r$, from Proposition \ref{Prop:CMM-ark}.

The upper bound on $l$ is from Proposition \ref{Prop:circuitsizeCM}. The lower bound again uses the transpose symmetry, switching $k$ and $l$ in the same inequality.
\end{proof}

The defining ideal for the bipartite rigidity matroid involves many more variables; as such, it is computationally much heavier. We computed circuits via the linear realization for some small examples; therefore, top-degrees are excluded.

\begin{Prop}[Rank 1]
The matroid $\CMM(m \times n, 1)$ is the graphic matroid on $K_{m,n}$. The top degree for a cycle of size $k$ has all appearing coordinates bounded below by $2^{k-1}$.
\end{Prop}

\begin{proof}

No proper subset is dependent, by Proposition \ref{Prop:CMM-ark}. The rank of the matroid is $2k - 1$, for $k \geq 2$ on the induced subgraph by Remark \ref{Rem:CMfacts} the set itself {\it is} dependent, thus a circuit.

By the same token, any circuit has vertices of degree $\geq 2$, implying that it contains a cycle, which we already know to be a circuit. This means the cycles are the only circuits.
The lower bound on the top-degree can be given by the number of generic solutions of a $1$-dimensional rigidity framework: after laying down the first edge, we have $k-1$ vertices from which the next edge can point left or right: This gives us $2^{k-1}$ real solutions generically, thus bounding each coordinate of the top-degree below by $2^{k-1}$.

\end{proof}

\begin{Prop}[Rank 2]
The circuit set of $\CMM(m \times n, 2)$ is the set of bipartite graphs constructed from $K_{4}$ via edge-splits and circuit gluing. In particular, we find the following masks for $\CMM(4\times n, 2)$, along with automorphism groups:

\begin{footnotesize}
\begin{longtable}{| c | c  c  c |}
\hline
{\bf Signature} 			& {\bf Mask and Stabilizer}					&										&  \\ \hline
\multirow{2}{*}{$(3,4)$} 	& \brmatrixone 		& 			   							&  \\
& $\frakS(X) = \frakS(3) \times \frakS(4)$		& 				 						&  \\[3mm] \hline

\multirow{2}{*}{$(4,4)$} 	& \brmatrixa 		& 										&  \\
& $\frakS(X) = (\frakS(2))^3$ &									 	&  \\[3mm] \hline
\multirow{2}{*}{$(4,5)$} 	& \brmatrixb 	& \brmatrixc & \brmatrixd \\
&  $\frakS(X) = (\frakS(2))^2$	& $\frakS(X) = (\frakS(2))^4$ &	$\frakS(X) = \frakS(4)$		\\[3mm] \hline

\multirow{4}{*}{$(4,6)$}	& \brmatrixe 	&  \brmatrixf &  \brmatrixg 							\\
&$\frakS(X) = (\frakS(3))^2 \times (\frakS(2))^2$
& $\frakS(X) = (\frakS(2))^4$
& $\frakS(X) = (\frakS(2))^3 \times \frakS(3)$								\\[3mm]

& \brmatrixh 	& \brmatrixi & \\
& $\frakS(X) = (\frakS(2))^2$
& $\frakS(X) = (\frakS(3))^2$		& 			\\[3mm] \hline

\caption{Circuits of the matroid $\CMM (4 \times n, 2)$.}
\label{r2cmctab}
\end{longtable}
\end{footnotesize}
\end{Prop}

\begin{proof}
Degree considerations restrict our attention to star graphs as circuit complements, as in the proof of Proposition \ref{r3det}. Computation of stabilizers also follows the logic there.
\end{proof}

\subsection{Symmetries and Number of Circuits for Graph Matroids}
The invariants introduced in section~\ref{sec:examples.bipsym} for bipartite graph matroids can be readily adapted to the graph matroid case:

\begin{Not}
Let $\matrex$ be a graph matroid limit with canonical symmetric group $G\cong \frakS(\NN\times \NN)$. We denote by
$$\calC_{k}(\matrex):=\{C/G\;:\;: C\in\calC(\matrex), \#\vsupp C = (k)\}$$
the set of (non-isomorphic) circuit graphs of $\matrex$ with signature $k$. We will denote its cardinality by $\circnum_{k}(\matrex):=\# \calC_{k}(\matrex).$ We will furthermore write
$$\stabavg_{k}(\matrex):=\sum_{C \in \calC_{k}(\matrex)} \frac{1}{\# \frakS (C)}.$$
Where clear from the context, the dependence on $\matrex$ will be omitted.
\end{Not}

\begin{Prop}\label{Prop:countformulasymm}
Let $\matrex_{\nu}$ be an injective complex of bipartite graph matroids, with bipartite graph matroid limit $\matrex$, having symmetric group $G\cong \frakS(\NN)$. Then, the following formulae hold:
\begin{description}
\item[(i)] $$\# \calC ( \matrex_{m,n}/G) = \sum_{k =1}^n  \circnum_{k}$$.
\item[(ii)]
$$\# \calC ( \matrex_{n}) = \sum_{k=1}^n \stabavg_k \cdot \frac{ n!}{(n-\ell)!}$$
\end{description}
\end{Prop}
\begin{proof}
The proof is completely analogous to that of Proposition~\ref{Prop:countformula}.
\end{proof}

\begin{Rem}
There are two major differences when enumerating the circuits of a graph matroid sequence as opposed to a bipartite graph sequence:

\begin{enumerate}
\item {\bf The symmetry group}. We have one set of vertices to permute, as opposed to two sets that we permute separately. Instead of $\frakS(\NN \times \NN)$ we now use $\frakS(\NN)$.
\item {\bf Bound on circuit support}. In the asymmetric case, when $m$ is fixed, the signature of any circuit in the sequence is bounded by a constant. Here, there may be no such bound, so specific values of $n$ may be accessible while no general formula exists.
\end{enumerate}

\end{Rem}

\subsection{The Symmetric Determinantal Matroid}

\begin{Prop}[Rank 1]
The circuits of the matroid $\vec D^\oslash (5 \times 5, 1)$ are given by the graphs in the following table, recorded with stabilizer, degree and a circuit in the asymmetric determinantal matroid with the desired symmetrization.

\begin{footnotesize}
\begin{longtable}{| c | c  c  c |}
\hline
{\bf Signature} 			& {\bf Graph}							& {\bf Mask of Asymmetric Circuit}					& {\bf Stabilizer \& Degree}	\\ \hline
\multirow{2}{*}{$2$} 		&\multirow{2}{*}{\includegraphics[scale = .7, trim=0 -5 0 -5]{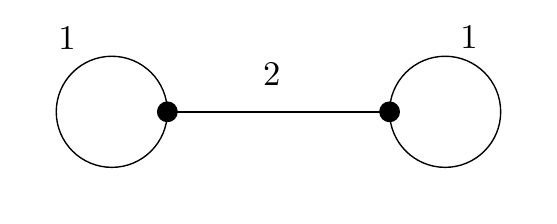}} 	& \multirow{2}{*}{\smatrixa }	& $\frakS(X) =  \ZZ/2\ZZ$		\\[3mm]
&   														&						& $\deg(\theta_X) = 2$		\\[5mm] \hline

\multirow{6}{*}{$3$} 		&\multirow{2}{*}{\includegraphics[scale = .7, trim=0 -5 0 -5]{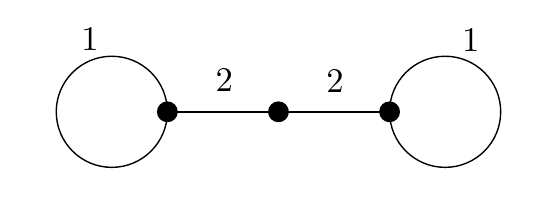}} 	& \multirow{2}{*}{\smatrixb}	& $\frakS(X) =  \ZZ/2\ZZ$		\\[3mm]
&   														&						& $\deg(\theta_X) = 3$		\\[5mm] \cline{2-4}
&\multirow{2}{*}{\includegraphics[scale = .6, trim=0 -5 0 -5]{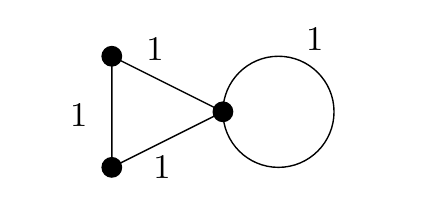}} 	& \multirow{2}{*}{\smatrixc }	 & $\frakS(X) =  \ZZ/2\ZZ$		\\[3mm]
&   														&						& $\deg(\theta_X) = 2$		\\[5mm] \hline

\multirow{4}{*}{$4$} 		&\multirow{2}{*}{\includegraphics[scale = .6, trim=0 -5 0 -5]{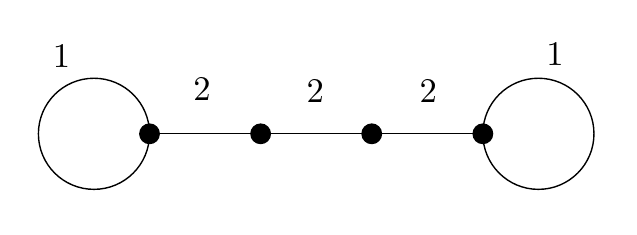}} 	& \multirow{2}{*}{\smatrixd }	& $\frakS(X) =  \ZZ/2\ZZ$		\\[3mm]
&   														&						& $\deg(\theta_X) = 4$		\\[8mm] \cline{2-4}
&\multirow{2}{*}{\includegraphics[scale = .5, trim=0 -5 0 -5]{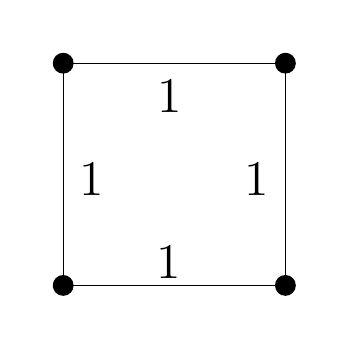}} 	& \multirow{2}{*}{\smatrixe }	& $\frakS(X) = D_4$		\\[3mm]
&   														&						& $\deg(\theta_X) = 2$		\\[8mm] \cline{2-4}
&\multirow{2}{*}{\includegraphics[scale = .5, trim=0 -5 0 -5]{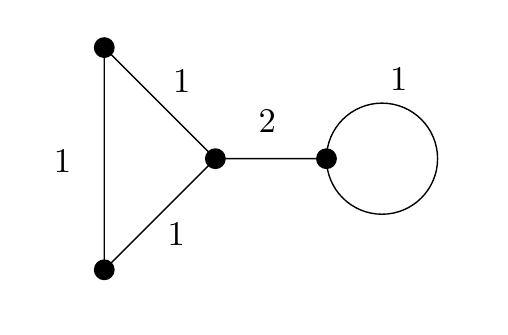}} 	& \multirow{2}{*}{\smatrixf }	 & $\frakS(X) =  \ZZ/2\ZZ$		\\[3mm]
&   														&						& $\deg(\theta_X) = 3$		\\[8mm] \hline
\multirow{8}{*}{$5$} 		&\multirow{2}{*}{\includegraphics[scale = .7, trim=0 -5 0 -5]{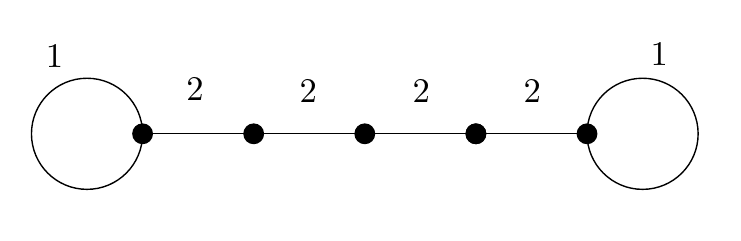}} 	& \multirow{2}{*}{\smatrixg }	 & $\frakS(X) =  (\ZZ/2\ZZ)^3$		\\[3mm]
&   														&						& $\deg(\theta_X) = 5$		\\[12mm] \cline{2-4}
&\multirow{2}{*}{\includegraphics[scale = .7, trim=0 -5 0 -5]{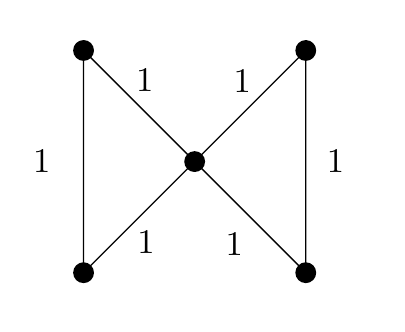}} 	& \multirow{2}{*}{\smatrixh }	& $\frakS(X) =  \ZZ/2\ZZ$		\\[3mm]
&   														&						& $\deg(\theta_X) = 3$		\\[12mm] \cline{2-4}
&\multirow{2}{*}{\includegraphics[scale = .7, trim=0 -5 0 -5]{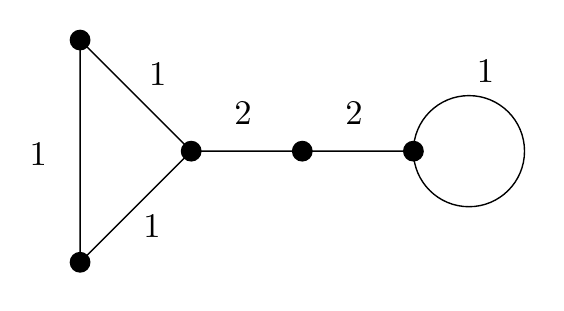}} 	& \multirow{2}{*}{\smatrixj }	& $\frakS(X) =  \ZZ/2\ZZ$		\\[3mm]
&   														&						& $\deg(\theta_X) = 4$		\\[12mm] \cline{2-4}
&\multirow{2}{*}{\includegraphics[scale = .7, trim=0 -5 0 -5]{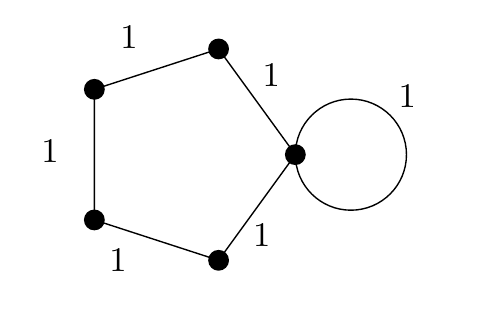}} 	& \multirow{2}{*}{\smatrixi }	& $\frakS(X) =  \ZZ/2\ZZ$		\\[3mm]
&   														&						& $\deg(\theta_X) = 3$		\\[12mm] \hline

\caption{Circuits of the matroid $\detMsym (5 \times 5,1)$.}
\label{55detsymtable}
\end{longtable}
\end{footnotesize}

\end{Prop}

One may observe that each of these graphs may be constructed inductively from the three graphs that are symmetrizations of $2\times 2$ minors. The top-degree can also be described (in these cases) as the size of the pre-image of each variable under symmetrization of a circuit in $\vec D(m \times n, 1)$.

\subsection{The Rigidity Matroid}

\begin{Prop}[Rank 1]
The matroid $\CMMsym(n \times n, 1)$ is the graphic matroid on $K_n$. The top degree for a cycle of size $k$ has all appearing coordinates equal to $2^{k-1}$.\end{Prop}
\begin{proof}
The proof follows as in the bipartite case; the only difference is the ground set.
\end{proof}

The bases of the matroid $\CMMsym(n \times n, 2)$ are characterized by Laman's theorem \cite{L70}. The set of circuits were more recently described in \cite{CMW}, as being constructed from $K_4$ via a finite set of moves. We use these moves, along with computation of the matroid in its linear realization, to list some circuits for rank $2$:
\begin{Prop}[Rank 2]
The set of circuits of $\CMMsym(6 \times 6, 2)$ is given in the table below. In particular, $\#\calC(\CMMsym(6 \times 6, 2)) =  642$.

\begin{footnotesize}
\begin{longtable}{| c | c c c  c |}
\hline
{\bf Signature} 		& {\bf Graph and Stabilizer}							& 	&	&				\\ \hline
\multirow{2}{*}{$4$}	& \includegraphics[scale = .8]{rigR2K4}			&	&	&				\\[3mm]
& $\frakS(X) =  \frakS(4)$								&	&	&				\\ \hline
\multirow{2}{*}{$5$}	& \includegraphics[scale = .8]{rigR2K5} 			&	&	&				\\[3mm]
& $\frakS(X) =  D_4$									&	&	& 				\\ \hline
\multirow{2}{*}{$6$}	&\includegraphics[scale = .4]{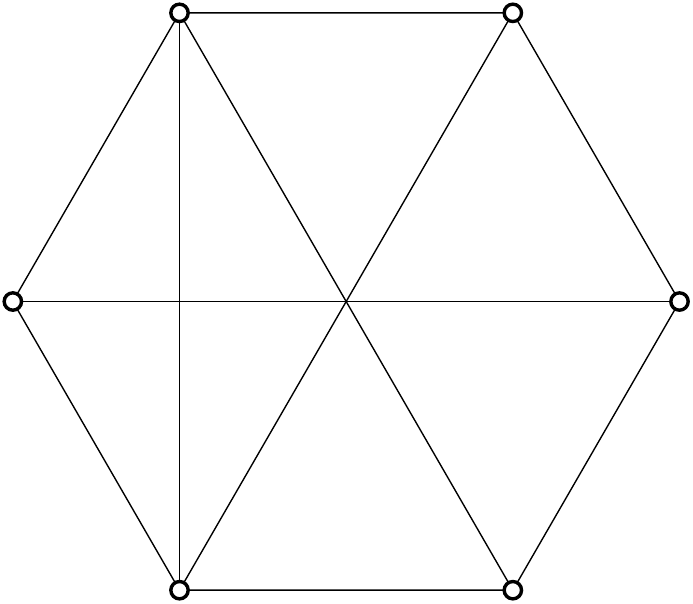}
& \includegraphics[scale = .5]{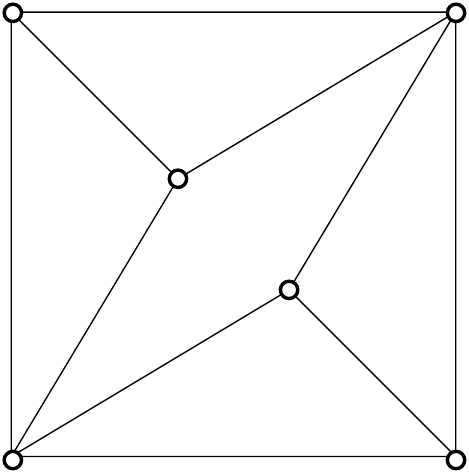}
& \includegraphics[scale = .5]{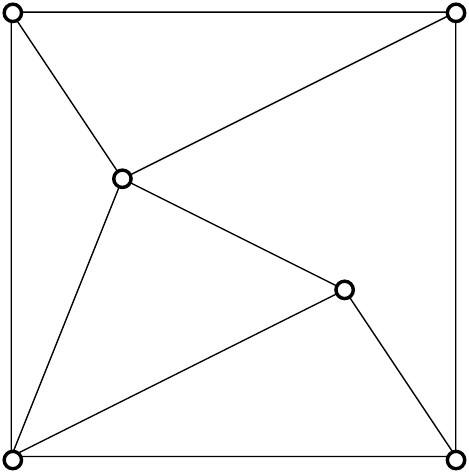}
& \includegraphics[scale = .6]{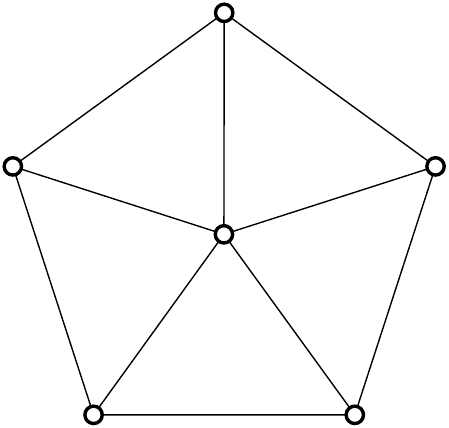}	\\[3mm]
& $\frakS(X) =  \frakS(3) \times \frakS(2)$		& $\frakS(X) =  \frakS(2) \times D_4$		&	$\frakS(X) =  \frakS(2)$	&$\frakS(X) = D_5$		\\[3mm] \hline
\caption{Circuits of the matroid $\detMsym (5 \times 5,1)$.}
\label{CMMsymtable}
\end{longtable}
\end{footnotesize}
\end{Prop}

\bibliographystyle{plainnat}

\end{document}